\documentclass[12pt,a4paper]{article}
\usepackage{amsmath, amsfonts, amsthm, amssymb, graphicx, color, dsfont}
\usepackage{enumerate, tikz, booktabs}
\usetikzlibrary{calc}
\usetikzlibrary{intersections}
\usetikzlibrary{decorations.pathmorphing}

\definecolor{col0}{RGB}{255,0,0}
\definecolor{col1}{RGB}{0,200,0}
\definecolor{col2}{RGB}{0,0,255}
\definecolor{col3}{RGB}{255,255,0}
\definecolor{col4}{RGB}{200,0,200}

\makeatletter
\tikzset{
	no text/.style={
		execute at begin node = \begin{lrbox}{\pgfutil@tempboxa},
			execute at end node   = \end{lrbox}
	},
	colour0/.style={draw=black},
	colour1/.style={draw=col2, dashed},
	colour2/.style={draw=col4, dashed, line width=2pt},
	colour2b/.style={draw=white, thin},
	colour3/.style={draw=col0,decorate,decoration={snake,amplitude=1pt}},
	colour4/.style={draw=col1, line width=2pt},
	colour4b/.style={draw=white,thin},
	every node/.style={draw,circle,inner sep=1pt, font=\scriptsize}
}
\makeatother

\def\bs{{\bigskip}}
\def\s{{\smallskip}}

\newtheorem{theorem}{Theorem}

\newtheorem{proposition}{Proposition}
\newtheorem{lemma}{Lemma}
\newtheorem{corollary}{Corollary}

\newtheorem*{conjecture}{Conjecture}

\newcounter{myobservation}

\newcommand*\samethanks[1][\value{footnote}]{\footnotemark[#1]}

\usepackage{hyperref}

\begin{document}
	
	\title
	{Hamiltonian cycles and 1-factors in 5-regular graphs}
	\author{{\sc Nico VAN CLEEMPUT\thanks{Department of Applied Mathematics, Computer Science and Statistics, Ghent University, Krijgslaan 281 - S9, 9000 Ghent, Belgium. E-mail address: nicolas.vancleemput@ugent.be}}\; and {\sc Carol T. ZAMFIRESCU\samethanks}\;\thanks{Department of Mathematics, Babe\c{s}-Bolyai University, Cluj-Napoca, Roumania. E-mail address: czamfirescu@gmail.com}}
	\date{}
	
	\maketitle
	\begin{center}
		\vspace{2mm}
		\begin{minipage}{125mm}
			{\bf Abstract.} It is proven that for any integer $g \ge 0$ and $k \in \{ 0, \ldots, 10 \}$, there exist infinitely many 5-regular graphs of genus~$g$ containing a 1-factorisation with exactly $k$ pairs of 1-factors that are \emph{perfect}, i.e.\ form a hamiltonian cycle. For $g = 0$ and $k = 10$, this settles a problem of Kotzig from 1964. Motivated by Kotzig and Labelle's ``marriage'' operation, we discuss two gluing techniques aimed at producing graphs of high cyclic edge-connectivity. We prove that there exist infinitely many planar 5-connected 5-regular graphs in which every 1-factorisation has zero perfect pairs. On the other hand, by the Four Colour Theorem and a result of Brinkmann and the first author, every planar 4-connected 5-regular graph satisfying a condition on its hamiltonian cycles has a linear number of 1-factorisations each containing at least one perfect pair. We also prove that every planar 5-connected 5-regular graph satisfying a stronger condition contains a 1-factorisation with at most nine perfect pairs, whence, every such graph admitting a 1-factorisation with ten perfect pairs has at least two edge-Kempe equivalence classes.
			
			\s
			
			
			\s
			
			
		\end{minipage}
	\end{center}
	
	\vspace{15mm}

	\section{Introduction}
	
	\noindent Edge-colourings are fundamental objects of study in graph theory---for a panoramic view, see~\cite{CCJST19}. Therein, it is pointed out that already K\H{o}nig, Shannon, and Vizing used a simple yet powerful structural tool to manipulate edge-colourings: so-called edge-Kempe switches, which permute the colours of the edges of a bichromatic cycle, and are named after Kempe who used the vertex version of these switches in his attempt to prove the Four Colour Theorem~\cite{Ke79}. Regular graphs have received special attention, as for many problems these are the most difficult cases. We give recent examples and refer the reader to further references therein. In~\cite{bH14}, belcastro and Haas count edge-Kempe equivalence classes in 3-edge-colourable \emph{cubic} (i.e.\ 3-regular) graphs. Asratian and Casselgren~\cite{AC16} treated edge-colourings allowing one or two more colours than the chromatic index, and dealt in particular with \emph{quartic} (i.e.\ 4-regular) graphs; they settled the smallest previously unsolved case of a problem of Vizing on edge-Kempe switches. McDonald, Mohar, and Scheide~\cite{MMS12} proved that all 4-edge-colourings of a (sub)cubic graph are edge-Kempe equivalent. A good survey on reconfigurations of vertex- and edge-colourings is \cite{MN}. For edge-colourings and edge-Kempe switching in regular triangulations of the torus and applications in statistical mechanics, we refer to~\cite{MS10}, and for a paper using edge-Kempe equivalence classes to show that the standard Monte Carlo algorithms are nonergodic, see~\cite{Ce17}. Not only was Kempe working on planar graphs, various recent contributions treat the planar case. We refer to~\cite{bH14} and Mohar's \cite[Section~3]{Mo06}. Combining regularity and planarity, a conjecture of Seymour states that if $G$ is a $k$-regular planar graph, then $G$ is $k$-edge-colourable if and only if $G$ is oddly $k$-edge-connected (defined below). This conjecture is still open in general, but several cases have been resolved~\cite{CES15}. We now introduce the notions needed in later sections.
	
	A graph is \emph{$k$-regular} if all of its vertices are of degree~$k$. In a graph $G$, a \emph{$k$-factor} is a spanning $k$-regular subgraph of $G$. (A 1-factor is often called a \emph{perfect matching}.) A \emph{$1$-factorisation} of $G$ is a partition of its edge set into 1-factors. A pair of 1-factors whose union forms a hamiltonian cycle is said to be \emph{perfect}, and a 1-factorisation is \emph{perfect} if all of its pairs of 1-factors are perfect. We will often speak simply of pairs, suppressing the fact that these are pairs of 1-factors, and we shall sometimes give 1-factors a certain colour---all colourings in this paper are \emph{proper} edge-colourings, i.e.\ incident edges must receive different colours. Following Knuth~\cite{Kn19}, when a graph $G$ admits a perfect 1-factorisation, we say that $G$ is \emph{perfectly hamiltonian}, while Kotzig named such graphs ``fortement hamiltonien''. (When we speak of a $d$-edge-coloured graph $G$ that is perfectly hamiltonian, we tacitly assume that $G$ is coloured in such a way that every two colour classes induce a hamiltonian cycle.) Chapter 5 of C.-Q. Zhang's treatise on cycle double covers~\cite{Zh12} is devoted to ``Kotzig graphs'', which are cubic perfectly hamiltonian graphs. Clearly, every planar 5-regular perfectly hamiltonian graph admits a hamiltonian cycle double and quadruple cover, while a triple cover is impossible. In general, planar 5-regular graphs need not have a hamiltonian cycle, even if 3-connectedness is imposed.
	
	Every uniquely cubic 3-edge-colourable graph has exactly three hamiltonian cycles as any pair of colours forms a hamiltonian cycle, and any hamiltonian cycle different from one of these three cycles would induce a new edge-colouring (Thomason notes that the inverse does not hold~\cite{Th82}). Every cubic graph containing exactly three hamiltonian cycles is perfectly hamiltonian, for instance the tetrahedron as well as any graph obtained from it by successively replacing vertices by triangles. But planar cubic perfectly hamiltonian graphs with more than three hamiltonian cycles are also available, for instance the dodecahedron. Partially answering a question of Kotzig and Labelle~\cite{KL78}, Mazzuoccolo showed that a cubic graph of even size is perfectly hamiltonian if and only if its line graph is a perfectly hamiltonian graph~\cite{Ma08}, which provides us with a rich family of planar quartic perfectly hamiltonian graphs. Regular perfectly hamiltonian graphs form a hierarchy: if one has a $k$-regular such graph, one immediately also has a perfectly hamiltonian graph that is $\ell$-regular for all $\ell \le k$. Therefore, restricted to the plane, by Euler's formula the most difficult case is the \emph{quintic} (i.e.\ 5-regular) case and this will be our main focus in this paper. Our starting point is the following question of Kotzig~\cite[p.~162, probl\`{e}me~19]{Ko64}; see also \cite[p.~104, probl\`{e}me~10]{KL79}.
	
	\bs
	
	\noindent \textbf{Problem (Kotzig), 1964.} \emph{Do planar $5$-regular perfectly hamiltonian graphs exist?}
	
	\bs
	
	\noindent In Section~2 we solve this question in the affirmative, characterising the orders for which such graphs exist and proving that there is an at least exponentially increasing number of such graphs. A brief remark on the chromatic index of planar quintic graphs: Seymour conjectured in around 1973 that if $G$ is a $k$-regular planar graph, then $G$ is $k$-edge-colourable if and only if $G$ is oddly $k$-edge-connected, where a graph $G$ is \emph{oddly $k$-edge-connected} if for all odd cardinality $X \subset V(G)$ the cardinality of the set of all edges of $G$ with an end in $X$ and an end in $V(G)\setminus X$ is at least $k$; see~\cite{CES15} for more details. For a solution of the case $k = 5$ various authors refer to the unpublished manuscript~\cite{Gu03} of Guenin. As this result dates to at least 2003\footnote{See the technical reports of the Combinatorics and Optimization department at the University of Waterloo:  https://uwaterloo.ca/combinatorics-and-optimization/research-combinatorics-and-optimization/technical-reports/technical-reports-2003} and has still not appeared in print, we do not consider the problem settled for the quintic case. Guenin's work, if valid, implies that planar 5-connected quintic graphs have chromatic index 5. We point out that planar 3-connected quintic graphs with chromatic index 6 exist. Whether planar quintic graphs of connectivity 4 can have chromatic index 6 is unknown to the authors.
	
	If we drop the planarity requirement in Kotzig's problem, then $K_6$, the complete graph on six vertices, yields a solution as every 1-factorisation of $K_6$ has ten perfect pairs (and, with an operation we shall discuss, one can obtain from $K_6$ an infinite family). In Section~3 we present enumeration results on perfect pairs in planar quintic graphs and discuss edge-Kempe equivalence classes. 
	All of our results on edge-colourings of quintic graphs can be translated to vertex-colourings of certain 8-regular graphs (by considering their line-graphs), as for instance Mohar does in a different setting in~\cite{Mo06}, where he looks into vertex- and edge-colourings with one or two more colours than the chromatic index.
	
	Our notation generally follows Diestel's book~\cite{Di16}, only that here a complete graph on $n$ vertices shall be denoted with $K_n$, and a path between distinct vertices $v$ and $w$ will be called a \emph{$vw$-path}. Many results in this paper rely on computer programs to find the initial building blocks or eliminate small cases. All necessary programs, instructions on how to use them as well as the resulting graphs can be obtained at \url{http://www.github.com/nvcleemp/hc_1f_5reg}.

	\section{A solution to a problem of Kotzig}\label{sec:kotzig}

	A straightforward, but efficiently implemented branch-and-bound algorithm was used to systematically construct all proper edge-colourings up to equivalence by trying all colours and bounding when conflicting colours are detected. When applied to the planar quintic graphs generated in~\cite{HMR11}, the program yielded the following lemma---the eleven 1-factorisations required for its proof are given in Figure~\ref{fig:20_10} for the ten perfect pairs case and in Figure~\ref{fig:20} in the Appendix for the other counts of perfect pairs. Already this lemma gives an affirmative answer to Kotzig's aforementioned problem, but we shall give a much stronger result below.
	
	\begin{lemma}\label{lem:perfect_pairs_for_all_k}
		There exists a planar quintic graph $H$ such that $H$ has a $1$-factorisation containing exactly $k$ perfect pairs for every $k \in \{ 0, \ldots, 10 \}$.
	\end{lemma}
	
	\begin{figure}
		\begin{center}
			\begin{tikzpicture}[very thick]
				
				\begin{scope}[no text]
					
					\node (0) at (3,0) {1};
					\node (1) at (0,3) {2};
					\node (2) at (0,-3) {3};
					\node (8) at (-3,0) {9};
					
					\node (4) at (1.5,0) {5};
					\node (6) at (0,1.5) {7};
					\node (10) at (0,-1.5) {11};
					\node (11) at (.5,-.75) {12};
					\node (12) at (.5,0) {13};
					\node (13) at (.5,.75) {14};
					\node (14) at (-.5,.9) {15};
					\node (15) at (-1.5,.4) {16};
					\node (16) at (-1.5,-.4) {17};
					\node (17) at (-.5,-.9) {18};
					\node (18) at (-.5,-.3) {19};
					\node (19) at (-.5,.3) {20};
					
					\node (5) at ($.5*(0)+.5*(1)$) {6};
					\node (3) at ($.5*(0)+.5*(2)$) {4};
					\node (7) at ($.5*(1)+.5*(8)$) {8};
					\node (9) at ($.5*(2)+.5*(8)$) {10};
				\end{scope}
				\draw[colour0] (0) edge[bend right, colour0] (1);
				\draw[colour0] (2) -- (3);
				\draw[colour0] (4) -- (11);
				\draw[colour0] (5) -- (13);
				\draw[colour0] (6) -- (7);
				\draw[colour0] (8) -- (9);
				\draw[colour0] (10) -- (17);
				\draw[colour0] (12) -- (18);
				\draw[colour0] (14) -- (19);
				\draw[colour0] (15) -- (16);
				
				\draw[colour1] (0) edge[bend left, colour1] (2);
				\draw[colour1] (1) -- (5);
				\draw[colour1] (3) -- (4);
				\draw[colour1] (6) -- (14);
				\draw[colour1] (7) -- (8);
				\draw[colour1] (9) -- (17);
				\draw[colour1] (10) -- (11);
				\draw[colour1] (12) -- (13);
				\draw[colour1] (15) -- (19);
				\draw[colour1] (16) -- (18);
				
				\draw[colour2] (0) -- (3);
				\draw[colour2b] (0) -- (3);
				\draw[colour2] (1) edge[bend right, colour2] (8);
				\draw[colour2b] (1) edge[bend right, colour2b] (8);
				\draw[colour2] (2) -- (10);
				\draw[colour2b] (2) -- (10);
				\draw[colour2] (4) -- (5);
				\draw[colour2b] (4) -- (5);
				\draw[colour2] (6) -- (13);
				\draw[colour2b] (6) -- (13);
				\draw[colour2] (7) -- (14);
				\draw[colour2b] (7) -- (14);
				\draw[colour2] (9) -- (16);
				\draw[colour2b] (9) -- (16);
				\draw[colour2] (11) -- (17);
				\draw[colour2b] (11) -- (17);
				\draw[colour2] (12) -- (19);
				\draw[colour2b] (12) -- (19);
				\draw[colour2] (15) -- (18);
				\draw[colour2b] (15) -- (18);
				
				\draw[colour3] (0) -- (4);
				\draw[colour3] (1) -- (7);
				\draw[colour3] (2) -- (9);
				\draw[colour3] (3) -- (10);
				\draw[colour3] (5) -- (6);
				\draw[colour3] (8) -- (16);
				\draw[colour3] (11) -- (12);
				\draw[colour3] (13) -- (19);
				\draw[colour3] (14) -- (15);
				\draw[colour3] (17) -- (18);
				
				\draw[colour4] (0) -- (5);
				\draw[colour4b] (0) -- (5);
				\draw[colour4] (1) -- (6);
				\draw[colour4b] (1) -- (6);
				\draw[colour4] (2) edge[bend left, colour4] (8);
				\draw[colour4b] (2) edge[bend left, colour4b] (8);
				\draw[colour4] (3) -- (11);
				\draw[colour4b] (3) -- (11);
				\draw[colour4] (4) -- (12);
				\draw[colour4b] (4) -- (12);
				\draw[colour4] (7) -- (15);
				\draw[colour4b] (7) -- (15);
				\draw[colour4] (9) -- (10);
				\draw[colour4b] (9) -- (10);
				\draw[colour4] (13) -- (14);
				\draw[colour4b] (13) -- (14);
				\draw[colour4] (16) -- (17);
				\draw[colour4b] (16) -- (17);
				\draw[colour4] (18) -- (19);
				\draw[colour4b] (18) -- (19);
				
			\end{tikzpicture}
		\end{center}
		\caption{A planar quintic graph on 20 vertices with a $1$-factorisation containing exactly \(10\) perfect pairs. This graph answers the problem posed by Kotzig in 1964 in the affirmative.}\label{fig:20_10}
	\end{figure}
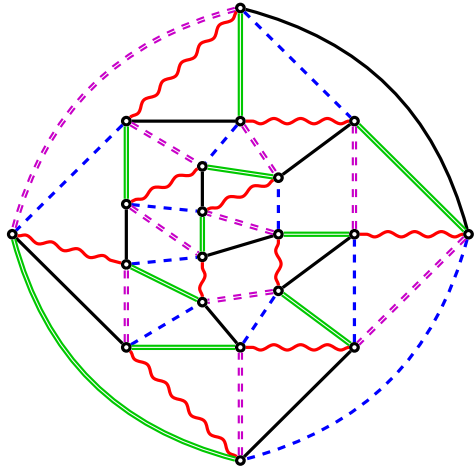

	\subsection{Marriage}
	
	Kotzig and Labelle~\cite{KL79} called the following procedure, for cubic graphs, ``mariage'' (a French word). They pointed out that the marriage of two perfectly hamiltonian cubic graphs produces a cubic perfectly hamiltonian graph. We now generalise this observation.
	
	\begin{lemma}\label{lem:marriage}
		Let $G$ and $H$ be disjoint $d$-regular $d$-edge-coloured graphs such that $G$ is perfectly hamiltonian and $H$ contains exactly $k$ perfect pairs. Consider $x \in V(G)$ and $y \in V(H)$, and let $N(x) = \{ x_1, \ldots, x_d \}$ and $N(y) = \{ y_1, \ldots, y_d \}$. Let the colour of $xx_i$ and $yy_i$ be $i$. Then $$G_x \, \omega \, H_y := (V(G) \setminus \{ x \} \cup V(H) \setminus \{ y \}, E(G - x) \cup E(H - y) \cup \{ x_iy_i \}_{i=1}^d)$$ is a $d$-regular $d$-edge-coloured graph containing exactly $k$ perfect pairs.
	\end{lemma}
	
	\begin{proof}
		Consider a pair of colours $(i, j)$ in an edge-colouring of $G_x \,\omega\, H_y$ obtained as described above. It is perfect iff $G$ contains a hamiltonian $i$-$j$-coloured $x_ix_j$-path and $H$ contains a hamiltonian $i$-$j$-coloured $y_iy_j$-path. Thus, $(i, j)$ is perfect in \(G_x \,\omega\, H_y\) iff $(i, j)$ is perfect in both $G$ and $H$. As $G$ is perfectly hamiltonian, it is whether $(i, j)$ is perfect (or not) in $H$ which determines whether $(i, j)$ is perfect in $G_x \,\omega\, H_y$.
	\end{proof}
	
	Whenever the choice of $x \in V(G)$ and $y \in V(H)$ is irrelevant for a particular argument, we simply write $G \,\omega\, H$, as Kotzig and Labelle did. If $G$ and $H$ are planar, then the operation is performed such that $G \,\omega\, H$ is planar. We can always permute the colours in either \(G\) or \(H\) so that the cyclic orders of the colours around \(x\) and \(y\) match.
	
	We note that a seemingly natural generalisation of this lemma stating that the marriage of \(d\)-regular \(d\)-edge-coloured graphs having \(k \) and \(\ell \) perfect pairs, respectively, yields a graph with \(\min\{k,\ell\}\) perfect pairs, does not hold for arbitrary $k$ and $\ell$. For instance, consider the case where \(G\) has the perfect pairs \((1,2),\ldots,(1,d)\), and \(H\) has the perfect pairs \((1,2), (3,4), \ldots, (2\left\lfloor\frac{d}{2}\right\rfloor-1,2\left\lfloor\frac{d}{2}\right\rfloor)\).
	
	\begin{theorem}\label{thm:infinite_any_genus}%
		For any integer $g \ge 0$, there exist infinitely many $5$-connected quintic graphs of genus~$g$, each of which contains a $1$-factorisation with exactly $k$ perfect pairs for all $k \in \{ 0, \ldots, 10 \}$.
	\end{theorem}
	
	\begin{proof}
		Let $k \in \{ 0, \ldots, 10 \}$. We call $\Lambda^k$ the edge-coloured graph $H$ from Lemma~\ref{lem:perfect_pairs_for_all_k} which has been edge-coloured such that it contains exactly $k$ perfect pairs. Consider the edge-coloured graph
		$$\Lambda := \Lambda^k \,\omega\, \underbrace{(\ldots(((\Lambda^{10} \,\omega\, \Lambda^{10}) \,\omega\, \Lambda^{10}) \,\omega\, \ldots ) \,\omega\, \Lambda^{10})}_{\text{Marriage of at least $g$ copies of $\Lambda^{10}$}}.$$
		$\Lambda$ is a planar quintic edge-coloured graph which by Lemma~\ref{lem:marriage} contains exactly $k$ perfect pairs. We can perform the marriage above for exactly the same vertices for each \(k\) in order to always obtain the same underlying planar quintic uncoloured graph that then contains a \(1\)-factorisation with exactly \(k\) perfect pairs for all \(k \in \{ 0, \ldots, 10 \}\).
		
		In $K_6$, every 1-factorisation has ten perfect pairs. The genus of the vertex-deleted subgraph of $K_6$, which is $K_5$, is 1. Let $K^1, \ldots , K^g$ be pairwise disjoint copies of $K_6$ and $y_i \in V(K^i)$ arbitrary. Consider $g$ pairwise non-adjacent vertices $x_i$ in $\Lambda$. Then
		$$\Gamma := ( \ldots (( \Lambda_{x_1} \,\omega\, K^1_{y_1})_{x_2} \,\omega\, K^2_{y_2})_{x_3} \,\omega\, K^3_{y_3} \ldots )_{x_g} \,\omega\, K^g_{y_g}$$
		is a quintic edge-coloured graph which by Lemma~\ref{lem:marriage} contains exactly $k$ perfect pairs. Its genus is at least $g$ because it contains $g$ pairwise disjoint copies of $K_5$. We use the (local) embedding of $K_5$ given in Figure~\ref{fig:k5} to show that $\Gamma$ indeed has genus exactly $g$. It is routine to verify that $\Gamma$ is 5-connected, for instance with Menger's Theorem.
		
		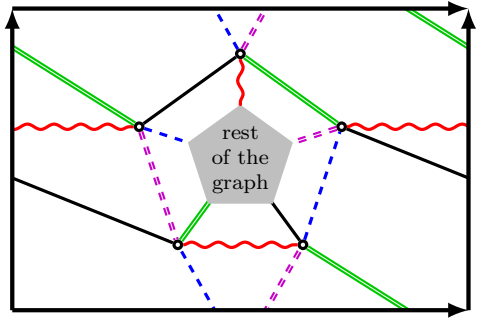
\begin{figure}[!ht]
			\begin{center}
				\begin{tikzpicture}[very thick,>=latex]
					\newcommand{\pentradius}{1.4}
					\newcommand{\ipentradius}{.7}
					\begin{scope}
						\clip (-3,-2) rectangle (3,2);
						
						\coordinate (pent1) at (90:\pentradius);
						\coordinate (pent2) at (162:\pentradius);
						\coordinate (pent3) at (234:\pentradius);
						\coordinate (pent4) at (306:\pentradius);
						\coordinate (pent5) at (18:\pentradius);
						
						\foreach \x in {1,...,5} \node (vpent\x) at (pent\x) {};
						
						\coordinate (ipent1) at (90:\ipentradius);
						\coordinate (ipent2) at (162:\ipentradius);
						\coordinate (ipent3) at (234:\ipentradius);
						\coordinate (ipent4) at (306:\ipentradius);
						\coordinate (ipent5) at (18:\ipentradius);

						\path (vpent5) edge[colour0] ($(pent3) + (6,0)$);
						\path ($(pent5) + (-6,0)$) edge[colour0] (vpent3);
						
						\path (vpent1) edge[colour1] ($(pent3) + (0,4)$);
						\path ($(pent1) + (0,-4)$) edge[colour1] (vpent3);
						
						\path (vpent1) edge[colour2] ($(pent4) + (0,4)$);
						\path ($(pent1) + (0,-4)$) edge[colour2] (vpent4);
						\path (vpent1) edge[colour2b] ($(pent4) + (0,4)$);
						\path ($(pent1) + (0,-4)$) edge[colour2b] (vpent4);
						
						\path (vpent5) edge[colour3] ($(pent2) + (6,0)$);
						\path ($(pent5) + (-6,0)$) edge[colour3] (vpent2);
						
						\path (vpent2) edge[colour4] ($(pent4) + (-6,4)$);
						\path ($(pent2) + (6,0)$) edge[colour4] ($(pent4) + (0,4)$);
						\path ($(pent2) + (6,-4)$) edge[colour4] (vpent4);
						\path (vpent2) edge[colour4b] ($(pent4) + (-6,4)$);
						\path ($(pent2) + (6,0)$) edge[colour4b] ($(pent4) + (0,4)$);
						\path ($(pent2) + (6,-4)$) edge[colour4b] (vpent4);
						
						\draw[colour0] (vpent1) -- (vpent2);
						\draw[colour1] (vpent4) -- (vpent5);
						\draw[colour2] (vpent2) -- (vpent3);
						\draw[colour2b] (vpent2) -- (vpent3);
						\draw (vpent3) edge[colour3] (vpent4);
						\draw[colour4] (vpent1) -- (vpent5);
						\draw[colour4b] (vpent1) -- (vpent5);
						
						\draw[colour0] (vpent4) -- (ipent4);
						\draw[colour1] (vpent2) -- (ipent2);
						\draw[colour2] (vpent5) -- (ipent5);
						\draw[colour2b] (vpent5) -- (ipent5);
						\draw (vpent1) edge[colour3] (ipent1);
						\draw[colour4] (vpent3) -- (ipent3);
						\draw[colour4b] (vpent3) -- (ipent3);
						
						\filldraw[draw=gray!50,fill=gray!50] (ipent1) -- (ipent2) -- (ipent3) -- (ipent4) -- (ipent5) -- (ipent1);
						\node[rectangle,draw=none,font=\scriptsize,text width=1cm,align=center] at (0,0) {rest of the graph};
						
					\end{scope}
					
					\draw[ultra thick,->] (-3, 2) -- ( 3, 2);
					\draw[ultra thick,->] (-3,-2) -- ( 3,-2);
					\draw[ultra thick,->] (-3,-2) -- (-3, 2);
					\draw[ultra thick,->] ( 3,-2) -- ( 3, 2);

				\end{tikzpicture}
			\end{center}
			\caption{The embedding of the vertex-deleted part of \(K_6\) (i.e.\ a \(K_5\)) that is used while performing the marriages from the proof of Theorem~\ref{thm:infinite_any_genus}.}\label{fig:k5}
		\end{figure}
		
	\end{proof}
	
	Perfectly hamiltonian graphs which are quintic immediately give rise to quartic and cubic perfectly hamiltonian graphs by removing a monochromatic perfect matching or the edges of a bichromatic hamiltonian cycle, respectively. Clearly, these removals conserve planarity. This observation can also be formulated in the following, perhaps more appealing way:
	
	\begin{corollary}\label{cor:underlying_cubic_and_quartic}
		There exist infinitely many planar $3$-connected cubic perfectly hamiltonian graphs which can be extended, by adding a suitable perfect matching of the complement, to planar $3$-connected quartic perfectly hamiltonian graphs, which by adding a further suitable perfect matching of the complement, can be extended to planar $3$-connected quintic perfectly hamiltonian graphs.
	\end{corollary}

	By the aforementioned result of Mazzuoccolo~\cite{Ma08}, the line-graphs of the cubic graphs presented in the above corollary immediately yield another infinite family of planar quartic perfectly hamiltonian graphs whenever the cubic graphs have even size (this occurs when in the proof of Theorem~\ref{thm:infinite_any_genus} the graph $\Lambda$ is obtained by marrying an odd number of copies of $\Lambda_{10}$).
	
	Let $G$ be a $d$-regular $d$-edge-colourable graph. Denote with $C(G) = \{ 1, \ldots, d \}$ the colours in an edge-colouring of $G$. Throughout this article, for an edge $e$ of $G$, let $c(e)$ be the colour of $e$. Consider $M \subset E(G)$, put $$\varphi_i(M) = | \{ e \in M : c(e) = i \} |,$$ and denote the set of all $i$-coloured edges in $G$ with $E_i$. When we write $E_i \cup E_j$, we refer to the 2-factor formed by all edges of colour $i$ or $j$. The following lemma, the proof of which is easy and omitted, will be useful. Note that the case \(\ne 0\) can be included since we require that \((i,j)\) is a perfect pair.
	
	\begin{lemma}\label{lem:parity_edge_cut}
		In a $d$-regular $d$-edge-coloured graph $G$, for any edge-cut $M$ of $G$ and any distinct $i, j \in C(G)$ such that \((i,j)\) is a perfect pair, we have $$\varphi_i(M) + \varphi_j(M) \equiv 0 \ {\rm mod} \ 2 \ \quad {\rm {\it and}} \quad \ne 0.$$
	\end{lemma}

	\begin{lemma}\label{lem:4connected}
		Every quintic perfectly hamiltonian graph $G$ is $4$-connected. If $G$ contains a $4$-vertex-cut $X$, then $G[X]$ is not a cycle. Moreover, there exists a planar quintic perfectly hamiltonian graph of connectivity $4$.
	\end{lemma}
	
	\begin{proof}
		We will use the following observation which is a direct consequence of Lemma~\ref{lem:parity_edge_cut}.
		
		\smallskip
		
		\noindent \textbf{Claim.} \emph{Consider a quintic perfectly hamiltonian $5$-edge-coloured graph $G$ and let $f$ be the (set-valued) function mapping the cardinality of an edge-cut $M$ of $G$ to the set of all $5$-tuples $(\varphi_i(M))_{i=1}^5$, modulo colour permutations. Then
			$f(k) = \emptyset$ for all $k \in \{ 0, \ldots, 4, 6 \}$ and $f(7) \subseteq \{ (1, 1, 1, 1, 3) \}$.}
		
		\smallskip
		
		The connectivity of $G$ cannot be 2: Let $X_2$ be a 2-vertex-cut in $G$. Then for every $v \in X_2$, there exists a component of $G - X_2$ which contains at least two vertices adjacent to $v$. Let the edges connecting these vertices to \(v\) be coloured with colours $1$ and $2$. Then $E_1 \cup E_2$ cannot form a hamiltonian cycle in $G$.
		
		The connectivity of $G$ cannot be 3: Let $X_3$ be a 3-vertex-cut in $G$. Then \(G - X_3\) has at most three components. If there are three components, then \(X_3\) forms an independent set as otherwise we immediately have a \(k\)-edge-cut with \(k\leq4\). Let \(C\) be a component of \(G - X_3\). By the Claim, the distribution of edges from \(C\) to the vertices of \(X_3\) is either \((1,1,3)\) or \((1,2,2)\). In the former case we can easily find a 4-edge-cut, while in the latter case we would have a 6-edge-cut. So we can assume that \(G - X_3\) has two components. At most two vertices of \(X_3\) can be adjacent: if two edges $e_1, e_2$ coloured $1, 2$, respectively, are present in $E(G[X_3])$, then $E_1 \cup E_2$ cannot form a hamiltonian cycle in $G$.
		We call $v \in X_3$ a \emph{$c$-vertex} if there is a component of $G - X_3$ connected to $v$ by exactly $c$ edges. If $X_3$ contains at least two 1-vertices, then $G$ contains an edge-cut with four or fewer edges. If there is exactly one 1-vertex, then either the other two vertices of $X_3$ are a 2-vertex and a 3-vertex, or the other two vertices of $X_3$ are both 2-vertices and are adjacent; in both cases we obtain a 6-edge-cut. The same conclusion holds if there is no 1-vertex, i.e.\ $X_3$ is composed of three 2-vertices. In every case, a contradiction to the Claim is obtained.

		We now prove that if a quintic perfectly hamiltonian graph $G$ contains a $4$-vertex-cut $X$, then $G[X]$ is not a cycle. Suppose that there is a separating 4-cycle $C$ in $G$ (reductio ad absurdum). If two edges of $C$ have the same colour, then we easily get a contradiction by considering the cycle corresponding to that colour and any other colour in $C$. Therefore, the edges of $C$ have pairwise distinct colours. Assume $G$ to be drawn in the plane---possibly with crossing edges---such that the component $I$ of $G - C$ residing inside of $C$ has at most as many edges connecting it to $C$ as the component $O$ residing outside, i.e.\ at most six. By Lemma~\ref{lem:parity_edge_cut}, only the case of five edges \emph{emanating inside}, i.e.\ between $C$ and $I$, is possible. Thus, seven edges \emph{emanate outside}, i.e.\ between $C$ and $O$, among which one colour, say $1$, occurs thrice (by the Claim), and must be different from the four colours of $E(C)$. Denote the vertices of $C$ incident with these three edges by $v_1, v_2, v_3$, and the fourth vertex of $C$ by $v_4$ such that the labels appear consecutively on $C$. The $1$-coloured edge incident with $v_4$ emanates inside. Let the colour of $v_2v_3$ be $2$. Then the $2$-coloured edge incident with $v_4$ must emanate outside (otherwise $E_1 \cup E_2$ is non-hamiltonian), and the $2$-coloured edge incident with $v_1$ emanates inside. Let $v_1v_2$ be $3$-coloured. Inspecting the cycle $E_2 \cup E_3$, we conclude that the $3$-coloured edge incident with $v_3$ emanates outside, and the $3$-coloured edge incident with $v_4$ emanates inside. But then $E_1 \cup E_3$ has at least two components, a contradiction.
		
		For the last statement, see Figure~\ref{fig:quintic_ph_4conn} in the Appendix.
	\end{proof}

	Kotzig~\cite{Ko64} proved that planar quartic perfectly hamiltonian graphs with $2k$ vertices exist iff $k \ge 3, k \ne 4$. We now give a quintic analogue of this result, while the cubic analogue reads ``iff $k \ge 2$'' and is easily established.
	
	\begin{theorem}\label{thm:orders}%
		There exists a planar quintic perfectly hamiltonian graph on~$2k$ vertices if and only if $k \ge 10$.
	\end{theorem}
	
	\begin{proof}
		In this proof, the lower index of a graph \(G_n\) denotes the order of the graph. Let $G_{20}$ be the planar quintic perfectly hamiltonian edge-coloured graph from Figure~\ref{fig:20_10}, and $G_{22}$, $G_{24}$, $G_{26}$, $G_{28}$, $G_{30}$, $G_{32}$, $G_{34}$, $G_{36}$ the planar quintic perfectly hamiltonian edge-coloured graphs given in Figure~\ref{fig:orders} in the Appendix. Set $${\cal G} := \{ G_{20}, G_{22}, G_{24}, G_{26}, G_{28}, G_{30}, G_{32}, G_{34}, G_{36} \}.$$ Then, using Lemma~\ref{lem:marriage}, iteratively applying the marriage operation starting from the graphs in \({\cal G}\) yields the existence of planar quintic perfectly hamiltonian graphs on~$2k$ vertices for every $k \ge 10$ since \(G\omega G_{20}\) has 18 vertices more than \(G\).
		
		For the lower bound we used the planar 4-connected quintic graphs---we may restrict the connectivity due to Lemma~\ref{lem:4connected}---generated in \cite{HMR11} and the computer program described at the beginning of this section to check that no such graph on at most 18 vertices is perfectly hamiltonian.
	\end{proof}
	
	\begin{theorem}\label{thm:exponential_perfectly_hamiltonian}
		The number of planar quintic perfectly hamiltonian graphs grows at least exponentially.
	\end{theorem}
	
	\newcommand{\connectblock}{ %
		\begin{tikzpicture}
			\foreach \y in {0pt, 2pt, 4pt, 6pt, 8pt} \draw (0,\y) -- (.75em,\y);
			\useasboundingbox ($(current bounding box.south west) + (-.25em,0)$) rectangle ($(current bounding box.north east) + (.25em,0)$);
		\end{tikzpicture} %
	}
	\newcommand{\closeleft}{ %
		\begin{tikzpicture}
			\foreach \y in {0pt, 2pt, 4pt, 6pt, 8pt} \draw (0,4pt) -- (.5em,\y);
			\useasboundingbox (current bounding box.south west) rectangle ($(current bounding box.north east) + (.1em,0)$);
		\end{tikzpicture} %
	}
	\newcommand{\closeright}{ %
		\begin{tikzpicture}
			\foreach \y in {0pt, 2pt, 4pt, 6pt, 8pt} \draw (0,\y) -- (.5em,4pt);
			\useasboundingbox ($(current bounding box.south west) + (-.1em,0)$) rectangle (current bounding box.north east);
		\end{tikzpicture} %
	}
	
	\begin{proof}
		Consider the two building blocks \(A\) and \(B\) shown in Figure~\ref{fig:exponential_building_blocks}. If, in either block, we connect all the edges going to the left to a single vertex and all the edges going to the right to a single vertex, we obtain the two planar quintic perfectly hamiltonian graphs shown in Figure~\ref{fig:22} in the Appendix. These two graphs have essential edge-connectivity~8 (we recall that in a graph $G$, an edge-cut $M$ of $G$ is \emph{essential} if $G - M$ contains at least two non-trivial components, and that for a positive integer $k$, a graph is \emph{essentially $k$-edge-connected} if it does not have an essential edge-cut $M$ with fewer than $k$ edges), so every 5-edge-cut is \emph{trivial}, i.e.\ its removal leaves $K_1$ as a component. We construct a planar quintic perfectly hamiltonian graph by taking a sequence of these building blocks, connecting the dangling edges and closing the sequence on both sides by connecting the edges on either side to a single vertex. The resulting graph is perfectly hamiltonian because this operation corresponds to repeatedly performing the marriage of the two graphs on 22 vertices from Figure~\ref{fig:22} in the Appendix. Consider
		\[
		\Sigma := \closeleft X_1 \connectblock X_2 \connectblock X_3 \connectblock \dots \connectblock X_k \closeright
		\]
		where \(X_i \in \{ A,B \}\).
		
		Now assume two different choices of \(A\)'s and \(B\)'s for the \(X_i\) yield isomorphic \(\Sigma^1\) and \(\Sigma^2\). We write
		\[
		\Sigma^i := \closeleft X_1^i \connectblock X_2^i \connectblock X_3^i \connectblock \dots \connectblock X_k^i \closeright.
		\]
		
		The only non-trivial 5-edge-cuts occurring in \(\Sigma^1\) and \(\Sigma^2\) are between the building blocks, so any isomorphism must map these onto each other. Therefore, by construction, we have \(X^1_1 \cong X^2_1\) or \(X^1_1 \cong X^2_k\), whence \(X^1_i \cong X^2_i\) or \(X^1_i \cong X^2_{k + 1 - i}\) for all \(i \in \{ 1, \ldots, k \}\). Note that \(A\not\cong B\) since \(A\) contains two facial quadrangles that share an edge but \(B\) does not. This yields an exponential number of pairwise different \(\Sigma\)'s, depending on the choice of \(A\)'s and \(B\)'s for the \(X_i\), and thus the statement. \end{proof}

	\begin{figure}[!ht]
		\begin{tikzpicture}
			
			\node (1) at (2,0) {};
			\node (3) at (0,3) {};
			\node (4) at (0,-3) {};
			\node (6) at (1,0) {};
			\node (8) at (0,1.8) {};
			\node (12) at (0,-1.8) {};
			\node (14) at (.3,0) {};
			\node (15) at (0,.6) {};
			\node (17) at (-2,0) {};
			\node (18) at (-1,0) {};
			\node (20) at (0,-.6) {};
			\node (21) at (-.3,0) {};
			
			\node (2) at ($.5*(3) + .5*(1)$) {};
			\node (9) at ($.5*(3) + .5*(17)$) {};
			\node (5) at ($.5*(4) + .5*(1)$) {};
			\node (11) at ($.5*(4) + .5*(17)$) {};
			
			\node (19) at ($.5*(11) + .5*(20)$) {};
			\node (13) at ($.5*(14) + .5*(5)$) {};
			\node (16) at ($.5*(18) + .5*(8)$) {};
			\node (7) at ($.5*(6) + .5*(8)$) {};
			
			\draw [thick] (3) -- (2) -- (1) -- (5) -- (4) -- (11) -- (17) -- (9) -- (3)
			(8) -- (7) -- (6) -- (13) -- (12) -- (19) -- (18) -- (16) -- (8)
			(15) -- (14) -- (20) -- (21) -- (15)
			(17) -- (18) -- (21) -- (14) -- (6) -- (1)
			(17) -- (16) -- (9) -- (8) -- (3)
			(1) -- (7) -- (2) -- (8)
			(18) -- (11) -- (19) -- (20) -- (13) -- (5) -- (12) -- (4)
			(16) -- (15) -- (7) (15) -- (6)
			(19) -- (21) (13) -- (14) (20) -- (12);
			\draw (3) -- (3,3) (2) -- (3,1.5) (1) -- (3,0) (5) -- (3,-1.5) (4) -- (3,-3)
			(3) -- (-3,3) (9) -- (-3,1.5) (17) -- (-3,0) (11) -- (-3,-1.5) (4) -- (-3,-3);
			
			\node[draw=none,font=\large] at (0,-3.8) {\(A\)};
		\end{tikzpicture}
		\hfill
		\begin{tikzpicture}
			\node (0) at (2,0) {};
			\node (2) at (1,0) {};
			\node (6) at (0,-3) {};
			\node (7) at (0,-1.8) {};
			\node (9) at (.3,0) {};
			\node (10) at (0,.6) {};
			\node (11) at (0,1.8) {};
			\node (12) at (0,3) {};
			\node (16) at (0,-.6) {};
			\node (17) at (-.3,0) {};
			\node (20) at (-2,0) {};
			\node (21) at (-1,0) {};
			
			\node (4) at ($.5*(12)+.5*(0)$) {};
			\node (19) at ($.5*(12)+.5*(20)$) {};
			\node (1) at ($.5*(6)+.5*(0)$) {};
			\node (14) at ($.5*(6)+.5*(20)$) {};
			
			\node (3) at ($.5*(11)+.5*(2)$) {};
			\node (8) at ($.5*(7)+.5*(2)$) {};
			\node (15) at ($.5*(7)+.5*(21)$) {};
			\node (18) at ($.5*(11)+.5*(21)$) {};
			
			\draw [thick] (12) -- (4) -- (0) -- (1) -- (6) -- (14) -- (20) -- (19) -- (12)
			(11) -- (3) -- (2) -- (8) -- (7) -- (15) -- (21) -- (18) -- (11)
			(10) -- (9) -- (16) -- (17) -- (10)
			(20) -- (21) -- (17) -- (9) -- (2) -- (0)
			(20) -- (15) -- (14) -- (7) -- (6)
			(0) -- (3) -- (4) -- (11) -- (12)
			(21) -- (16) -- (8) -- (1) -- (7) (8) -- (9) (15) -- (16)
			(2) -- (10) -- (18) -- (19) -- (11) (17) -- (18) (10) -- (3);
			
			\draw (12) -- (3,3) (4) -- (3,1.5) (0) -- (3,0) (1) -- (3,-1.5) (6) -- (3,-3)
			(12) -- (-3,3) (19) -- (-3,1.5) (20) -- (-3,0) (14) -- (-3,-1.5) (6) -- (-3,-3);
			
			\node[draw=none,font=\large] at (0,-3.8) {\(B\)};
		\end{tikzpicture}
		\caption{The two building blocks used to show that the number of planar quintic perfectly hamiltonian graphs grows at least exponentially.}\label{fig:exponential_building_blocks}
	\end{figure}
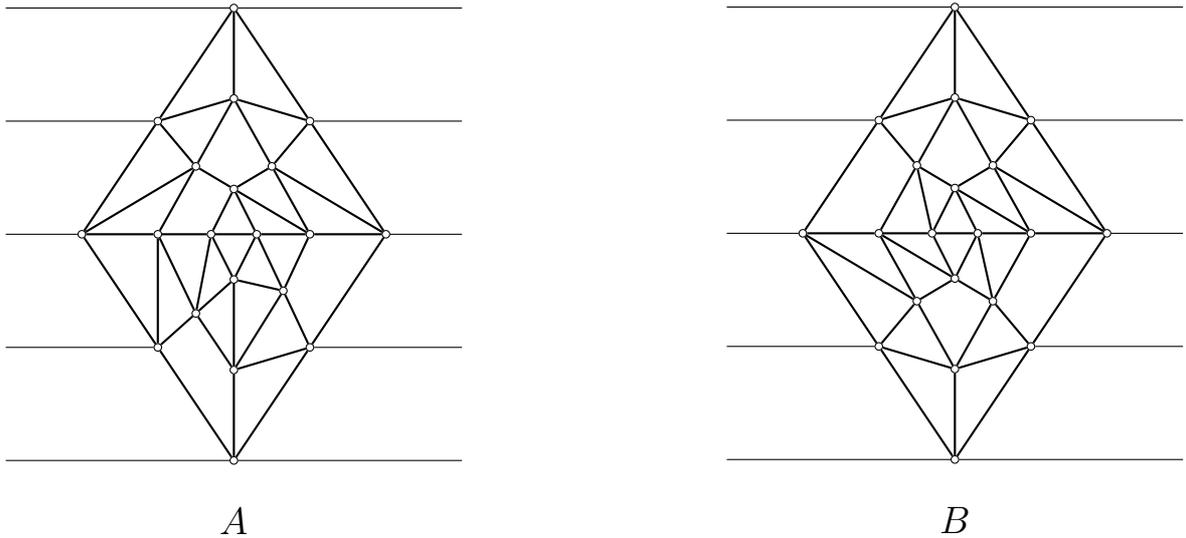

	Although the graphs described above all have order \(20k+2\), we can easily get an exponential number of graphs for any of the other orders by performing the marriage of the graphs constructed with one of the base graphs from Theorem~\ref{thm:orders}. A similar construction can be used to show that the number of planar cubic and quartic perfectly hamiltonian graphs also grows at least exponentially. The necessary building blocks for these two cases are shown in Figure~\ref{fig:exponential_building_blocks_cubic_quartic} in the Appendix. In Table~\ref{tab:counts} in the Appendix, we provide an overview of exact enumeration results on planar $k$-regular perfectly hamiltonian graphs of small order, for $k \in \{ 3, 4, 5 \}$.

	\subsection{Divorce}
	
	As we have discussed, Kotzig and Labelle introduced the marriage between two graphs. This technique has been frequently used, see for instance~\cite{bH14} or~\cite{Zh12}. With it, we can describe an infinite number of new perfectly hamiltonian graphs from a single perfectly hamiltonian graph. A defect of the marriage operation is the fact that it can only yield graphs with nontrivial 5-edge-cuts, so cyclically 6-edge-connected graphs cannot be obtained in this fashion. Planar quintic graphs have cyclical edge-connectivity at most 9 (consider the
	edges incident with the vertices of a triangle). We present in this subsection and the next subsection techniques aimed at producing quintic perfectly hamiltonian graphs with high cyclic edge-connectedness. Although the approach we now discuss has the advantage that no new non-trivial 5-edge-cuts are introduced (contrasting marriage), it has the disadvantage that we can only apply it once (marriage can be iterated ad infinitum, even with only one starting graph). We formulate our technique as applied to a certain graph, but it is natural to see this graph as a fragment of a larger quintic graph, therefore the term \emph{divorce}.

	\begin{theorem}\label{thm:divorce}
		There exist infinitely many planar quintic perfectly hamiltonian graphs with a non-trivial automorphism group containing a rotational symmetry.
	\end{theorem}
	
	\begin{proof} Let $F$ be a plane graph satisfying the following requirements. Let $F$ have four vertices $a, b, c, d$ lying in the same face $f$ and occurring in the facial walk of $f$ in the order $a, b, c, d$; let $F[ \{ a, b, c, d \}]$ contain exactly one edge, $bc$; let every vertex in $F$ be quintic except for $a,b,c,d$; let the vertices \(a, b, c\) have degree 3, and the vertex \(d\) degree 2. Finally, we require $F$ to be 5-edge-colourable (with colours 1, 2, 3, 4, 5), setting $c(bc) = 1$. We call $F$ \emph{suitable} if it satisfies the following properties listed by edge-colours (every path is bichromatic):
		
		\begin{enumerate}[(i)]
			\item (1,2): There is a hamiltonian $cd$-path using the edge $bc$.
			\item (1,3): There is a hamiltonian $bd$-path using the edge $bc$.
			\item (1,4): There is a hamiltonian $ac$-path in $F - d$ using the edge $bc$.
			\item (1,5): There is a hamiltonian $ab$-path in $F - d$ using the edge $bc$.
			\item (2,3): There is a hamiltonian $bc$-path.
			\item (2,4): There is a hamiltonian $ad$-path in $F - c$.
			\item (2,5): There is an $ad$-path $\mathfrak{p}_1$ and a $bc$-path $\mathfrak{q}_1$ such that $V(\mathfrak{p}_1)$ and $V(\mathfrak{q}_1)$ partition $V(F)$.
			\item (3,4): There is an $ab$-path $\mathfrak{p}_2$ and a $cd$-path $\mathfrak{q}_2$ such that $V(\mathfrak{p}_2)$ and $V(\mathfrak{q}_2)$ partition $V(F)$.
			\item (3,5): There is a hamiltonian $ad$-path in $F - b$.
			\item (4,5): There is a hamiltonian $bc$-path in $F - a - d$.
		\end{enumerate}
		
		Consider a suitable graph $B$, for instance the bottom half of the graph from Figure~\ref{fig:quintic_ph_4conn} in the Appendix, and $H$ a planar quintic perfectly hamiltonian graph. For a quintic vertex $x \in V(B)$ and an arbitrary vertex $y \in V(H)$ we obtain a suitable graph by considering $B_x \, \omega \, H_y$. Since there are infinitely many choices for $H$ by Theorem~\ref{thm:infinite_any_genus}, we obtain infinitely many suitable graphs.
		
		Let $F$ be one of these suitable graphs and $F'$ a copy thereof. For $v \in V(F)$, we denote by $v'$ its corresponding vertex in $F'$. We obtain the graph $G$ by identifying $a$ with $d'$ and $d$ with $a'$, as well as the edge $bc$ with the edge $c'b'$ (with these orientations). Then $G$ is a planar quintic perfectly hamiltonian graph with a non-trivial automorphism group containing a rotational symmetry.
		
		We will give a colouring of the edges of $G$ and show that every pair of colours induces a hamiltonian cycle. Every edge in $F$ already has a colour. We now colour the edges of $F'$ as follows: we apply the permutation $\pi$ to the colours such that $\pi(1) = 1$, $\pi(2) = 4$, $\pi(3) = 5$, $\pi(4) = 2$, $\pi(5) = 3$ (note that the edge $bc$, which is shared by $F$ and $F'$, has colour 1, which is invariant under $\pi$). The conditions for suitability imply that \(a\) is incident to the colours 1, 2, and 3; \(b\) is incident to the colours 1, 2, and 4; \(c\) is incident to the colours 1, 3, and 5; and \(d\) is incident to the colours 2 and 3. Therefore, the given colouring is a proper edge-colouring of \(G\). We treat the colour pairs one-by-one, giving two of the ten cases explicitly. All ten cases are illustrated in Figure~\ref{fig:divorce_proof} in the Supplementary material. The remaining situations can be dealt with similarly. All paths and cycles below are bichromatic, using only the two colours indicated in that case.
		
		\begin{itemize}
			\item (1,2): By (i), there is a hamiltonian $cd$-path $\mathfrak{h}$ in $F$ using the edge $bc$ with $c(bc) = 1$. Since $F'$ is a copy of $F$, by (iii) (and under the action of $\pi$) we have that there is a hamiltonian $a'c'$-path $\mathfrak{h}'$ in $F' - d'$ using $b'c'$, which after the identification of vertices yields a hamiltonian $bd$-path $\mathfrak{h}'$ in $F' - a$ using $bc$. Now $\mathfrak{h} \cup \mathfrak{h}'$ is a 1-2-coloured hamiltonian cycle in $G$. (Note that $\mathfrak{h}$ and $\mathfrak{h}'$ share the edge $bc$.)
			\item (2,5): By (vii), there is an $ad$-path $\mathfrak{p}$ and a $bc$-path $\mathfrak{q}$ in $F$ whose vertex sets partition $V(F)$ and by (viii) an $ab$-path $\mathfrak{p}'$ and a $cd$-path $\mathfrak{q}'$ in $F'$ whose vertex sets partition $V(F')$. Now $\mathfrak{p} \cup \mathfrak{q} \cup \mathfrak{p}' \cup \mathfrak{q}'$ is a 2-5-coloured hamiltonian cycle in $G$.
		\end{itemize}
	\end{proof}
	
	\subsection{Another operation yielding high cyclic edge-connectedness}
	
	As already mentioned, the `divorce' approach above is aimed at producing quintic perfectly hamiltonian graphs with high cyclic edge-connectedness. We pursue the same structural goal with yet another operation which we now describe. The drawbacks are that the former, by itself, only produces one new graph, while the latter does not preserve planarity.
	
	\begin{proposition}
		Let $G$ be a quintic perfectly hamiltonian $5$-edge-coloured graph containing a triangle $\Delta$ such that $N(V(\Delta)) = \{ x_1, \ldots, x_9 \}$ contains pairwise distinct vertices, and $H$ a copy of $G$ with $\Delta'$ denoting the triangle corresponding to $\Delta$. Set $N(V(\Delta')) = \{ y_1, \ldots, y_9 \}$, $G' := G - \Delta$, and $H' := H - \Delta'$. We write $V(\Delta) = \{ v_1, v_2, v_3 \}$ and $V(\Delta') = \{ w_1, w_2, w_3 \}$ such that, for $i \in \{ 1, 4, 7 \}$, $x_i,x_{i+1},x_{i+2}$ are adjacent to $v_{\lceil i/3 \rceil}$ and $y_i,y_{i+1},y_{i+2}$ are adjacent to $w_{\lceil i/3 \rceil}$.
		Set \(c(v_1v_2)=c(x_7v_3)=1\), \(c(v_2v_3)=c(x_1v_1)=2\), \(c(v_3v_1)=3\), \(c(x_jv_2)=j-1\) for \(4\leq j\leq6\), \(c(x_jv_1)=c(x_{j+6}v_3)=j+2\) for \(j\in\{2,3\}\), \(c(x_jv_{\left\lceil\frac{j}{3}\right\rceil})=c(y_jw_{\left\lceil\frac{j}{3}\right\rceil})\) for \(1\leq j \leq 9\), \(c(v_iv_j)=c(w_iw_j)\) for \(1\leq i \neq j \leq 3\).
		Then 
		$$(V(G') \cup V(H'), E(G') \cup E(H') \cup \{ x_1y_1, x_2y_8, x_3y_6, x_4y_4, x_5y_2, x_6y_9, x_7y_7, x_8y_5, x_9y_3 \})$$
		is a quintic perfectly hamiltonian graph.
	\end{proposition}
	
	\begin{proof}
		As in the proof of Theorem~\ref{thm:divorce}, the verification that the various bichromatic 2-factors in the new graph are indeed hamiltonian cycles is not complicated, but long and tedious. It is therefore omitted here, but we note that there are three distinct situations. When looking at two colours that are used in \(\Delta\), a single bichromatic path is left in \(G'\) and thus there is only one case for this situation. When one colour of \(\Delta\) and one colour not in \(\Delta\) is considered, the bichromatic hamiltonian cycle corresponds to two paths in \(G'\), and thus there are two distinct cases for this situation. Finally, considering the two colours not in \(\Delta\), the bichromatic hamiltonian cycle corresponds to three paths in \(G'\), and thus there are eight distinct cases for this situation.
	\end{proof}

	It remains open whether there exists an operation which yields from one input graph an infinite family, preserves planarity, 5-regularity, and perfect hamiltonicity, and has high cyclic edge-connectedness.

	\section{Counting colourings}
	
	We consider two edge-colourings to be distinct if one cannot be obtained from the other by a permutation of colours. Equivalently, distinct edge-colourings induce different partitions of the graph's edge set. The first question we want to address is motivated by the extendability of partial colourings. For which $k$ is it true that \emph{every} planar quintic 5-edge-colourable graph has a 1-factorisation with exactly $k$ perfect pairs? This is certainly not true for $k \in \{ 1, 9, 10 \}$ as shown by the icosahedron, see Proposition~\ref{obs:platonic} given toward the end of this section. Moreover, the following theorem proves that it does not hold for $k > 0$. The case $k=0$ remains open.
	
	\begin{theorem}\label{thm:infinite_zero_pairs}
		There exists an infinite family of planar $5$-connected quintic graphs in which every $1$-factorisation has zero perfect pairs.
	\end{theorem}
	
	\begin{proof}
		We use an argument the essential idea of which can already be found in work of Rosenfeld~\cite{Ro89} and Zaks~\cite{Za76}. Let $H$ be a planar 3-connected cyclically 5-edge-connected cubic graph. By the Four Colour Theorem, $E(H)$ can be partitioned into 1-factors $F_1$, $F_2$, $F_3$. Replacing, in $H$, each edge of $F_1 \cup F_2$ by a double edge, we obtain a planar quintic 5-edge-colourable multigraph $H_5$. Let $G$ be a planar 5-connected 5-edge-colourable quintic graph, and $g \in V(G)$. Put $G_g := G - g$. Construct $H(G_g)$ by substituting a copy of $G_g$ for each vertex of $H_5$. Rosenfeld showed that $H(G_g)$ is planar, 5-connected, and quintic. $H$ is now chosen to be non-hamiltonian (e.g.\ Grinberg's 44-vertex graph~\cite{Gr68}). Assume that \(H(G_g)\) has a \(1\)-factorisation which has at least one perfect pair. Any hamiltonian cycle \(\mathfrak{h}\) in $H(G_g)$ must visit every copy of $G_g$ at least once. As \(H\) is non-hamiltonian, there is at least one copy of \(G_g\) which is visited twice. Let \(M\) be the set of five edges from \(H_5\) incident to that copy of \(G_g\). Thus, \(\mathfrak{h}\) contains four edges of \(M\). If we remove the edges of the two colour classes constituting \(\mathfrak{h}\), what remains should be a planar cubic 3-edge-coloured graph, but since four edges of the 5-edge-cut \(M\) are contained in \(\mathfrak{h}\), this graph has a bridge, a contradiction since a cubic graph with a bridge is not 3-edge-colourable.
	\end{proof}
	
	The zero perfect pairs arise from the interplay of the hamiltonian cycle and an odd edge-cut which creates a partial edge-colouring of the graph that cannot be completed. We will now see that as soon as we add an additional constraint to prevent this situation, the behaviour changes drastically. We will call a hamiltonian graph \(G\) \emph{charonian} if for any hamiltonian cycle~\(\mathfrak{h}\) in $G$ the graph \(G-E(\mathfrak{h})\) is bridgeless. The word `charonian' comes from Charon, the ferryman of Hades who carries souls across the supposedly bridgeless river Styx that divides the world of the living from the world of the dead.
	
	A charonian graph \(G\) on $n$ vertices is \emph{strongly charonian} if for any 2-factor~\(\mathfrak{f}\) consisting of a \(4\)-cycle and an \( (n-4) \)-cycle the graph \(G-E(\mathfrak{f})\) is bridgeless. Let us briefly motivate the choice of this terminology: being strongly charonian is stronger than just being charonian (i.e.\ charonian with an additional condition), but we chose, in a certain sense, the weakest possible additional requirement in order to obtain an as large as possible family of graphs. The icosahedron is a planar quintic strongly charonian graph, and it even has the following stronger property, which might be of separate interest.
	
	\begin{proposition}
		For any $2$-factor $\mathfrak{f}$ in the icosahedron $G$, the graph $G - E(\mathfrak{f})$ is bridgeless.
	\end{proposition}
	
	There are three planar quintic charonian graphs on 24~vertices and two planar quintic charonian graphs on 28~vertices. Those are all the planar quintic charonian graphs on up to 30 vertices. Except for the icosahedron they all have vertex-connectivity 2.
	
	\begin{theorem}\label{thm:2connected_charonian_fam}
		Let \(G\) be a planar quintic charonian graph on $n$ vertices containing adjacent edges \(e, e'\) such that there are \(h>0\) hamiltonian cycles in \(G\) through $e$ and $e'$. Then there exists a family of planar quintic charonian graphs such that for each integer \(k>0\) there is a member in the family with \(kn\) vertices and at least \(h^k\) \(1\)-factorisations, each containing at least one perfect pair of \(1\)-factors.
	\end{theorem}
	
	\begin{proof}
		Assume \(e=xy\) and put \(H:=G-e\). For any integer \(k>0\) take \(k\) copies of \(H\), denoted by \(H_0, \dots, H_{k-1}\). Let the copy of \(x\), respectively \(y\), in \(H_i\) be \(x_i\), respectively \(y_i\). Connect \(y_i\) to \(x_{i+1}\) where the indices are taken modulo \(k\). Denote this graph by \(G_k\).
		
		We first show that \(G_k\) is charonian. The paths induced in the copies of \(H\) by any hamiltonian cycle of \(G\) containing \(xy\) together with the connecting edges between the copies form a hamiltonian cycle of \(G_k\). Let \(\mathfrak{h}\) be an arbitrary hamiltonian cycle in \(G_k\). Since any pair of connecting edges forms a 2-edge-cut, all connecting edges are contained in \(\mathfrak{h}\). The graph \(G_k-E(\mathfrak{h})\) is a disconnected graph. Let \(\mathfrak{p}_i\) be the hamiltonian \(x_iy_i\)-path of \(H_i\) induced by \(\mathfrak{h}\). The components of \(G_k-E(\mathfrak{h})\) corresponding to \(H_i\) are exactly the components of \(H_i-E(\mathfrak{p}_i)\). Since \(H_i + x_iy_i\) is charonian and \(\mathfrak{p}_i + x_iy_i\) is a hamiltonian cycle of this graph, \(H_i-E(\mathfrak{p}_i)\) is bridgeless.
		
		Denote the copy of \(e'\) in \(H_i\) by \(e'_i\). Since \(G\) has \(h\) hamiltonian cycles through \(e\) and \(e'\), \(H_i\) has \(h\) hamiltonian \(x_iy_i\)-paths through \(e'_i\). For any \(H_i\) we can independently choose any of these paths and combine them with the connecting edges. As a result we find that \(G_k\) has (at least) \(h^k\) hamiltonian cycles through the connecting edges and all edges \(e'_i\). Colour the connecting edges with the colour \(1\) and colour the edges \(e'_i\) with the colour \(2\). Let \(\mathfrak{h}\) be an arbitrary but fixed hamiltonian cycle through the connecting edges and all edges \(e'_i\). Extend this partial edge-colouring consistently on \(\mathfrak{h}\) with colours \(1\) and \(2\) (\(|E(\mathfrak{h}) \cap E(H_i)|\) is odd as \(G\) is quintic). Consider \(G'_k = G_k - E(\mathfrak{h})\). The graph \(G'_k\) is planar, cubic, and bridgeless. By the Four Colour Theorem, \(E(G'_k)\) can be partitioned into three 1-factors which are coloured \(3, 4, 5\), respectively. In this 1-factorisation of \(G_k\), at least one pair of 1-factors, namely the one formed by colours \(1\) and \(2\), is perfect. We have obtained \(h^k\) 1-factorisations of $G_k$, and these are indeed pairwise distinct since (i) no two hamiltonian cycles constructed above possess the same edge set, and (ii) the colours of the connecting edges and of the edges \(e'_i\) coincide for all these hamiltonian cycles.
	\end{proof}
	
	Our aim with the previous theorem is the following. Being charonian already implies that there exists at least one 1-factorisation containing a perfect pair. We also give the exact counts for charonian graphs on up to 30 vertices. Theorem~\ref{thm:2connected_charonian_fam} then proves that there are infinitely many and thus arbitrarily large charonian graphs, while also showing that they can even have a large number of 1-factorisations containing a perfect pair.
	
	The construction above yields for any charonian graph a specific family of planar quintic charonian graphs with a large number of \(1\)-factorisations containing at least one perfect pair. However, all members of this family have edge-connectivity~2. If we look at planar quintic charonian graphs with a higher vertex-connectivity we can prove the following results for any such graph.
	
	\begin{theorem}\label{thm:existence_non_perfect_1-factorisation}
		Let ${\cal G}_4$ (${\cal G}_5$) be the set of planar $4$-connected quintic charonian (planar $5$-connected quintic strongly charonian) graphs. Then there are constants $c > 0$ and $d > 0$ such that each graph $G \in {\cal G}_4$ has at least $c \cdot |V(G)|$ $1$-factorisations, each containing at least one perfect pair of $1$-factors, and each graph $G \in {\cal G}_5$ has at least $d \cdot |V(G)|$ $1$-factorisations, each containing at most nine perfect pairs of $1$-factors.
	\end{theorem}
	
	\begin{proof}
		Let $G \in {\cal G}_4$ have order $n$. By a recent theorem of Brinkmann and Van Cleemput~\cite{BV}, there exists a constant $c' > 0$ such that $G$ contains $c'n$ pairwise distinct hamiltonian cycles (i.e.\ on different edge sets). Thus, there is a constant $c'' > 0$ such that for an arbitrary vertex $v$ in $G$, there exist distinct edges incident with $v$ which are traversed by $c''n$ pairwise distinct hamiltonian cycles. We proceed as in the proof of Theorem~\ref{thm:2connected_charonian_fam} and obtain that there exists a constant $c \ge c'' > 0$ such that $G$ has $cn$ 1-factorisations, each containing at least one perfect pair of $1$-factors.
		
		We now prove the second statement. In a graph, we call two triangles sharing exactly one edge a \emph{diamond}.
		
		\smallskip
		
		\noindent \textbf{Claim.} \emph{The number of diamonds in an $n$-vertex planar quintic graph is $\Theta(n)$.}
		
		\smallskip
		
		\noindent \emph{Proof of the Claim.} Let $G$ be a plane quintic graph with $n$ vertices, $m$ edges, and $f_k$ $k$-faces. We first show that $f_3 \ge n + 8$. By Euler's formula $f_3 = 2 - n + m - \sum_{k \ge 4} f_k$, which implies, as $2m = 5n$ and $2m = \sum_{k \ge 3} k f_k$, that $$4f_3 = 8 - 4n + 2m + 2m - \sum_{k \ge 4} 4 f_k = 8 + n + \sum_{k \ge 3} k f_k - \sum_{k \ge 4} 4 f_k = 8 + n + 3f_3 + \sum_{k \ge 4} (k - 4) f_k,$$ whence $f_3 = 8 + n + \sum_{k \ge 4} (k - 4) f_k \ge n + 8$. The number of diamonds in $G$ is at least half of $s := 3f_3 - \sum_{k\geq4}kf_k$. We have $2m = \sum_{k\geq3}kf_k$, $f_3 \ge n + 8$ as we have just shown, and that $G$ is quintic, so
		$$s = 6f_3 - 2m \ge 6n + 48 - 2m = n + 48,$$
		which completes the Claim's proof since $f_3$ is at most linear in $n$ by Euler's formula.
		
		\smallskip
		
		Let $G$ be a planar $5$-connected quintic strongly charonian graph $G$ of order $n$. Consider in $G$ a diamond $D$ made of the two triangles $\Delta, \Delta'$ sharing an edge $e$. Consider $v \in V(\Delta) \setminus V(\Delta')$. Then $G_v := G - v$ is a planar 4-connected graph. By a theorem of Sanders~\cite{Sa96}, $G_v - \Delta'$ contains a hamiltonian cycle $\mathfrak{h}$ (the length of $\mathfrak{h}$ is even as the order of $G$ is even). Edge-colour with colours $1$ and $2$ the cycle $\mathfrak{h}$ as well as the 4-cycle formed by the edges $(E(\Delta) \cup E(\Delta')) \setminus \{e\}$. As before, removing from $G$ the edges of the 2-factor composed of these two cycles yields a planar cubic graph $H$. Since $H$ is bridgeless, by the Four Colour Theorem $H$ can be edge-coloured with three colours, whence, $G$ has a 1-factorisation with at most nine perfect pairs, as $E_1 \cup E_2$ is a non-hamiltonian 2-factor. By the Claim, we have a linear number of choices for $D$. Since the above procedure yields a 1-2-coloured 4-cycle and a 1-2-coloured $(n-4)$-cycle, at least for every eleventh diamond, a different colouring is obtained. From this the statement follows.
	\end{proof}
	
	The icosahedron is the only planar \(5\)-connected quintic (strongly) charonian graph that is known. It remains open whether other such graphs exist, but we note that marriage does not preserve the property of being (strongly) charonian; nor does performing the marriage at the \(8\)-valent vertex obtained by contracting an edge.
	
	The arguments above would work even if one would slightly relax the definition of ``charonian'' by allowing the removal of some hamiltonian cycles to produce graphs containing a bridge, as long as the number of such cycles is bounded by some constant, or perhaps even a suitable sublinear function. However, no such relaxation proved practical, so we opted for the most straightforward definition.
	
	\bigskip

	We conclude the paper with remarks on so-called Kempe equivalence classes. The prefix ``Kempe-'' typically refers to switching the colours of vertices in a bichromatic cycle, as this is what Kempe did in~\cite{Ke79}, while the prefix ``edge-Kempe-'' is used for switching the colours of edges in a bichromatic cycle. But since in this paper we only concern ourselves with the latter, we shall henceforth abbreviate ``edge-Kempe-'' to ``Kempe-''. Let $G$ be an edge-coloured graph containing a bichromatic cycle $C$. If we \emph{Kempe switch} $C$, this means permuting the two colours of $E(C)$, while all other edges of $G$ retain their colours. (This operation is also called a ``Kempe change'' or an ``interchange''.) If two 1-factorisations $F$ and $F'$ of $G$ can be obtained from one another through a sequence of Kempe switches, we call them \emph{Kempe equivalent} and write $F \sim F'$, noting that $\sim$ is an equivalence relation.
	
	Kempe equivalence classes have attracted significant attention. For edge-coloured graphs, we refer to~\cite{Mo06} and~\cite{bH14} for more details. One fundamental problem is to count these classes. We begin with the observation that we \emph{cannot} Kempe switch our way up to a perfect 1-factorisation, the proof of which now follows; note that a perfect 1-factorisation is only Kempe equivalent to itself.

	\begin{corollary}\label{cor:at_least_2_kempe_classes}
		Let $G$ be a regular graph and $F$ a non-perfect $1$-factorisation of $G$. Then $G$ admits no perfect $1$-factorisation $F' \sim F$. Thus, every planar quintic perfectly hamiltonian graph $G$ on $n$ vertices containing a $2$-factor $\mathfrak{f}$ composed of a $4$-cycle and an $(n-4)$-cycle such that $G - E(\mathfrak{f})$ is bridgeless, has at least two Kempe equivalence classes. In particular, planar $5$-connected quintic strongly charonian graphs that are perfectly hamiltonian have at least two Kempe equivalence classes.
	\end{corollary}

	\begin{proof}
		If we Kempe switch a hamiltonian cycle the number of perfect pairs remains unchanged: we are simply permuting the two colours of the edges of which the hamiltonian cycle is made. As $F$ is non-perfect there exist colours $1$ and $2$ such that $E_1 \cup E_2$ is a non-hamiltonian 2-factor. Kempe switching a (necessarily non-hamiltonian) 1-2-coloured cycle cannot render $E_1 \cup E_2$ hamiltonian.
		
		By Theorem~\ref{thm:existence_non_perfect_1-factorisation} a planar $5$-connected quintic strongly charonian graph $G$ has a 1-factorisation containing at most nine perfect pairs. Therefore, if $G$ is perfectly hamiltonian, it must have at least two Kempe equivalence classes.
	\end{proof}
	
	Theorem~\ref{thm:infinite_any_genus} implies that there exist infinitely many planar $5$-connected quintic graphs that have at least two Kempe equivalence classes. In fact, we believe the following to be true.
	
	\begin{conjecture}
		Every quintic graph contains at least two Kempe equivalence classes.
	\end{conjecture}
	
	A smallest planar counterexample would be cyclically 6-edge-connected and have at least 26 vertices.
	
	\bigskip
	
	We now enumerate perfect pairs and Kempe equivalence classes in the 1-skeleta of the Platonic solids. These results have, partially, already appeared in the literature. In the statements below, all counts are exact (and not lower bounds).

	\begin{proposition}\label{obs:platonic}
		The unique $1$-factorisation of the tetrahedron contains only perfect pairs; the octahedron has two $1$-factorisations which are each others mirror image and contain only perfect pairs; the cube has four $1$-factorisations of which three have two perfect pairs and one has zero perfect pairs; all ten $1$-factorisations of the dodecahedron have only perfect pairs; and there exist $1$-factorisations of the icosahedron with $k$ perfect pairs iff $k \in \{ 0, 2, \ldots, 8 \}$. Furthermore, the tetrahedron and the cube have one Kempe equivalence class, the octahedron and the icosahedron have two Kempe equivalence classes, and the dodecahedron has ten Kempe equivalence classes.
	\end{proposition}
	
	\begin{proof}
		The cases of the tetrahedron, cube, and octahedron are left to the reader. For the dodecahedron $D$, every 1-factorisation contains three perfect pairs. Assume $D$ contains a non-hamiltonian 2-factor $F$ composed of cycles of even length. In $D$, the shortest such cycle has length~8, so $F$ has exactly two components. We consider $D$ to be embedded in the plane and denote by $p$ ($p'$) the number of pentagons inside (outside) of $F$. A special case of Corollary~3.6 from \cite{BZ19} implies that $3(p' - p) = 4$, a contradiction since $p - p'$ is an integer. For the icosahedron, the eight relevant $1$-factorisations are given in Figure~\ref{fig:icosahedron} in the Appendix---that indeed no 1-factorisation exists with 1, 9, or 10 perfect pairs was established with the computer program described at the beginning of Section~\ref{sec:kotzig}.
		
		The number of Kempe equivalence classes for each of these graphs was determined with a computer program which exhaustively performs all possible Kempe switches on each 1-factorisation and uses a union-find data structure to obtain the classes.
	\end{proof}
	
	By the same argumentation as given for Corollary~\ref{cor:at_least_2_kempe_classes}, if a $d$-regular $d$-edge-colourable graph admits exactly $k$ perfect and $\ell$ non-perfect 1-factorisations, then it has at least $k + 1$ Kempe equivalence classes if $\ell > 0$ and exactly $k$ Kempe equivalence classes if $\ell = 0$ (as is the case for the octahedron, the dodecahedron, and $K_6$).

	\vspace{5mm}
	
	\noindent \textbf{Acknowledgement.} Zamfirescu's research is supported by a Postdoctoral Fellowship of the Research Foundation Flanders (FWO).

	\clearpage
	
	\section*{Appendix}
	
	\begin{figure}[h]
		\begin{center}
			\newenvironment{smallestquinticperfham}[1]{
				\begin{tikzpicture}[very thick,scale=.5]
					
					\node[rectangle,draw=none,font=\normalsize] at (0,-4) {#1 perfect pair\ifnum#1=1\else s\fi};
					
					\begin{scope}[no text]
						
						\node (0) at (3,0) {1};
						\node (1) at (0,3) {2};
						\node (2) at (0,-3) {3};
						\node (8) at (-3,0) {9};
						
						\node (4) at (1.5,0) {5};
						\node (6) at (0,1.5) {7};
						\node (10) at (0,-1.5) {11};
						\node (11) at (.5,-.75) {12};
						\node (12) at (.5,0) {13};
						\node (13) at (.5,.75) {14};
						\node (14) at (-.5,.9) {15};
						\node (15) at (-1.5,.4) {16};
						\node (16) at (-1.5,-.4) {17};
						\node (17) at (-.5,-.9) {18};
						\node (18) at (-.5,-.3) {19};
						\node (19) at (-.5,.3) {20};
						
						\node (5) at ($.5*(0)+.5*(1)$) {6};
						\node (3) at ($.5*(0)+.5*(2)$) {4};
						\node (7) at ($.5*(1)+.5*(8)$) {8};
						\node (9) at ($.5*(2)+.5*(8)$) {10};
					\end{scope}
				}{\end{tikzpicture}}

			\begin{smallestquinticperfham}{0}
				\draw[colour0] (0) edge[bend right, colour0] (1);
				\draw[colour0] (2) edge[bend left, colour0] (8);
				\draw[colour0] (3) -- (10);
				\draw[colour0] (4) -- (11);
				\draw[colour0] (5) -- (13);
				\draw[colour0] (6) -- (14);
				\draw[colour0] (7) -- (15);
				\draw[colour0] (9) -- (16);
				\draw[colour0] (12) -- (19);
				\draw[colour0] (17) -- (18);
				
				\draw[colour1] (0) edge[bend left, colour1] (2);
				\draw[colour1] (1) -- (5);
				\draw[colour1] (3) -- (4);
				\draw[colour1] (6) -- (13);
				\draw[colour1] (7) -- (14);
				\draw[colour1] (8) -- (9);
				\draw[colour1] (10) -- (11);
				\draw[colour1] (12) -- (18);
				\draw[colour1] (15) -- (19);
				\draw[colour1] (16) -- (17);
				
				\draw[colour2] (0) -- (3);
				\draw[colour2b] (0) -- (3);
				\draw[colour2] (1) -- (6);
				\draw[colour2b] (1) -- (6);
				\draw[colour2] (2) -- (9);
				\draw[colour2b] (2) -- (9);
				\draw[colour2] (4) -- (5);
				\draw[colour2b] (4) -- (5);
				\draw[colour2] (7) -- (8);
				\draw[colour2b] (7) -- (8);
				\draw[colour2] (10) -- (17);
				\draw[colour2b] (10) -- (17);
				\draw[colour2] (11) -- (12);
				\draw[colour2b] (11) -- (12);
				\draw[colour2] (13) -- (19);
				\draw[colour2b] (13) -- (19);
				\draw[colour2] (14) -- (15);
				\draw[colour2b] (14) -- (15);
				\draw[colour2] (16) -- (18);
				\draw[colour2b] (16) -- (18);
				
				\draw[colour3] (0) -- (4);
				\draw[colour3] (1) -- (7);
				\draw[colour3] (2) -- (10);
				\draw[colour3] (3) -- (11);
				\draw[colour3] (5) -- (6);
				\draw[colour3] (8) -- (16);
				\draw[colour3] (9) -- (17);
				\draw[colour3] (12) -- (13);
				\draw[colour3] (14) -- (19);
				\draw[colour3] (15) -- (18);
				
				\draw[colour4] (0) -- (5);
				\draw[colour4b] (0) -- (5);
				\draw[colour4] (1) edge[bend right, colour4] (8);
				\draw[colour4b] (1) edge[bend right, colour4b] (8);
				\draw[colour4] (2) -- (3);
				\draw[colour4b] (2) -- (3);
				\draw[colour4] (4) -- (12);
				\draw[colour4b] (4) -- (12);
				\draw[colour4] (6) -- (7);
				\draw[colour4b] (6) -- (7);
				\draw[colour4] (9) -- (10);
				\draw[colour4b] (9) -- (10);
				\draw[colour4] (11) -- (17);
				\draw[colour4b] (11) -- (17);
				\draw[colour4] (13) -- (14);
				\draw[colour4b] (13) -- (14);
				\draw[colour4] (15) -- (16);
				\draw[colour4b] (15) -- (16);
				\draw[colour4] (18) -- (19);
				\draw[colour4b] (18) -- (19);
				
			\end{smallestquinticperfham}
			\begin{smallestquinticperfham}{1}
				\draw[colour0] (0) edge[bend right, colour0] (1);
				\draw[colour0] (2) edge[bend left, colour0] (8);
				\draw[colour0] (3) -- (10);
				\draw[colour0] (4) -- (11);
				\draw[colour0] (5) -- (13);
				\draw[colour0] (6) -- (14);
				\draw[colour0] (7) -- (15);
				\draw[colour0] (9) -- (16);
				\draw[colour0] (12) -- (19);
				\draw[colour0] (17) -- (18);
				
				\draw[colour1] (0) edge[bend left, colour1] (2);
				\draw[colour1] (1) -- (5);
				\draw[colour1] (3) -- (4);
				\draw[colour1] (6) -- (13);
				\draw[colour1] (7) -- (14);
				\draw[colour1] (8) -- (9);
				\draw[colour1] (10) -- (17);
				\draw[colour1] (11) -- (12);
				\draw[colour1] (15) -- (16);
				\draw[colour1] (18) -- (19);
				
				\draw[colour2] (0) -- (3);
				\draw[colour2b] (0) -- (3);
				\draw[colour2] (1) -- (6);
				\draw[colour2b] (1) -- (6);
				\draw[colour2] (2) -- (9);
				\draw[colour2b] (2) -- (9);
				\draw[colour2] (4) -- (5);
				\draw[colour2b] (4) -- (5);
				\draw[colour2] (7) -- (8);
				\draw[colour2b] (7) -- (8);
				\draw[colour2] (10) -- (11);
				\draw[colour2b] (10) -- (11);
				\draw[colour2] (12) -- (18);
				\draw[colour2b] (12) -- (18);
				\draw[colour2] (13) -- (14);
				\draw[colour2b] (13) -- (14);
				\draw[colour2] (15) -- (19);
				\draw[colour2b] (15) -- (19);
				\draw[colour2] (16) -- (17);
				\draw[colour2b] (16) -- (17);
				
				\draw[colour3] (0) -- (4);
				\draw[colour3] (1) -- (7);
				\draw[colour3] (2) -- (10);
				\draw[colour3] (3) -- (11);
				\draw[colour3] (5) -- (6);
				\draw[colour3] (8) -- (16);
				\draw[colour3] (9) -- (17);
				\draw[colour3] (12) -- (13);
				\draw[colour3] (14) -- (19);
				\draw[colour3] (15) -- (18);
				
				\draw[colour4] (0) -- (5);
				\draw[colour4b] (0) -- (5);
				\draw[colour4] (1) edge[bend right, colour4] (8);
				\draw[colour4b] (1) edge[bend right, colour4b] (8);
				\draw[colour4] (2) -- (3);
				\draw[colour4b] (2) -- (3);
				\draw[colour4] (4) -- (12);
				\draw[colour4b] (4) -- (12);
				\draw[colour4] (6) -- (7);
				\draw[colour4b] (6) -- (7);
				\draw[colour4] (9) -- (10);
				\draw[colour4b] (9) -- (10);
				\draw[colour4] (11) -- (17);
				\draw[colour4b] (11) -- (17);
				\draw[colour4] (13) -- (19);
				\draw[colour4b] (13) -- (19);
				\draw[colour4] (14) -- (15);
				\draw[colour4b] (14) -- (15);
				\draw[colour4] (16) -- (18);
				\draw[colour4b] (16) -- (18);
				
			\end{smallestquinticperfham}
			\begin{smallestquinticperfham}{2}
				\draw[colour0] (0) edge[bend right, colour0] (1);
				\draw[colour0] (2) edge[bend left, colour0] (8);
				\draw[colour0] (3) -- (10);
				\draw[colour0] (4) -- (11);
				\draw[colour0] (5) -- (13);
				\draw[colour0] (6) -- (7);
				\draw[colour0] (9) -- (16);
				\draw[colour0] (12) -- (19);
				\draw[colour0] (14) -- (15);
				\draw[colour0] (17) -- (18);
				
				\draw[colour1] (0) edge[bend left, colour1] (2);
				\draw[colour1] (1) -- (5);
				\draw[colour1] (3) -- (4);
				\draw[colour1] (6) -- (13);
				\draw[colour1] (7) -- (14);
				\draw[colour1] (8) -- (9);
				\draw[colour1] (10) -- (11);
				\draw[colour1] (12) -- (18);
				\draw[colour1] (15) -- (19);
				\draw[colour1] (16) -- (17);
				
				\draw[colour2] (0) -- (3);
				\draw[colour2b] (0) -- (3);
				\draw[colour2] (1) -- (6);
				\draw[colour2b] (1) -- (6);
				\draw[colour2] (2) -- (9);
				\draw[colour2b] (2) -- (9);
				\draw[colour2] (4) -- (5);
				\draw[colour2b] (4) -- (5);
				\draw[colour2] (7) -- (8);
				\draw[colour2b] (7) -- (8);
				\draw[colour2] (10) -- (17);
				\draw[colour2b] (10) -- (17);
				\draw[colour2] (11) -- (12);
				\draw[colour2b] (11) -- (12);
				\draw[colour2] (13) -- (14);
				\draw[colour2b] (13) -- (14);
				\draw[colour2] (15) -- (16);
				\draw[colour2b] (15) -- (16);
				\draw[colour2] (18) -- (19);
				\draw[colour2b] (18) -- (19);
				
				\draw[colour3] (0) -- (4);
				\draw[colour3] (1) -- (7);
				\draw[colour3] (2) -- (10);
				\draw[colour3] (3) -- (11);
				\draw[colour3] (5) -- (6);
				\draw[colour3] (8) -- (16);
				\draw[colour3] (9) -- (17);
				\draw[colour3] (12) -- (13);
				\draw[colour3] (14) -- (19);
				\draw[colour3] (15) -- (18);
				
				\draw[colour4] (0) -- (5);
				\draw[colour4b] (0) -- (5);
				\draw[colour4] (1) edge[bend right, colour4] (8);
				\draw[colour4b] (1) edge[bend right, colour4b] (8);
				\draw[colour4] (2) -- (3);
				\draw[colour4b] (2) -- (3);
				\draw[colour4] (4) -- (12);
				\draw[colour4b] (4) -- (12);
				\draw[colour4] (6) -- (14);
				\draw[colour4b] (6) -- (14);
				\draw[colour4] (7) -- (15);
				\draw[colour4b] (7) -- (15);
				\draw[colour4] (9) -- (10);
				\draw[colour4b] (9) -- (10);
				\draw[colour4] (11) -- (17);
				\draw[colour4b] (11) -- (17);
				\draw[colour4] (13) -- (19);
				\draw[colour4b] (13) -- (19);
				\draw[colour4] (16) -- (18);
				\draw[colour4b] (16) -- (18);
				
			\end{smallestquinticperfham}
			\begin{smallestquinticperfham}{3}
				\draw[colour0] (0) edge[bend right, colour0] (1);
				\draw[colour0] (2) edge[bend left, colour0] (8);
				\draw[colour0] (3) -- (10);
				\draw[colour0] (4) -- (11);
				\draw[colour0] (5) -- (13);
				\draw[colour0] (6) -- (14);
				\draw[colour0] (7) -- (15);
				\draw[colour0] (9) -- (16);
				\draw[colour0] (12) -- (19);
				\draw[colour0] (17) -- (18);
				
				\draw[colour1] (0) edge[bend left, colour1] (2);
				\draw[colour1] (1) -- (5);
				\draw[colour1] (3) -- (4);
				\draw[colour1] (6) -- (13);
				\draw[colour1] (7) -- (14);
				\draw[colour1] (8) -- (16);
				\draw[colour1] (9) -- (17);
				\draw[colour1] (10) -- (11);
				\draw[colour1] (12) -- (18);
				\draw[colour1] (15) -- (19);
				
				\draw[colour2] (0) -- (3);
				\draw[colour2b] (0) -- (3);
				\draw[colour2] (1) -- (6);
				\draw[colour2b] (1) -- (6);
				\draw[colour2] (2) -- (9);
				\draw[colour2b] (2) -- (9);
				\draw[colour2] (4) -- (5);
				\draw[colour2b] (4) -- (5);
				\draw[colour2] (7) -- (8);
				\draw[colour2b] (7) -- (8);
				\draw[colour2] (10) -- (17);
				\draw[colour2b] (10) -- (17);
				\draw[colour2] (11) -- (12);
				\draw[colour2b] (11) -- (12);
				\draw[colour2] (13) -- (19);
				\draw[colour2b] (13) -- (19);
				\draw[colour2] (14) -- (15);
				\draw[colour2b] (14) -- (15);
				\draw[colour2] (16) -- (18);
				\draw[colour2b] (16) -- (18);
				
				\draw[colour3] (0) -- (4);
				\draw[colour3] (1) -- (7);
				\draw[colour3] (2) -- (10);
				\draw[colour3] (3) -- (11);
				\draw[colour3] (5) -- (6);
				\draw[colour3] (8) -- (9);
				\draw[colour3] (12) -- (13);
				\draw[colour3] (14) -- (19);
				\draw[colour3] (15) -- (18);
				\draw[colour3] (16) -- (17);
				
				\draw[colour4] (0) -- (5);
				\draw[colour4b] (0) -- (5);
				\draw[colour4] (1) edge[bend right, colour4] (8);
				\draw[colour4b] (1) edge[bend right, colour4b] (8);
				\draw[colour4] (2) -- (3);
				\draw[colour4b] (2) -- (3);
				\draw[colour4] (4) -- (12);
				\draw[colour4b] (4) -- (12);
				\draw[colour4] (6) -- (7);
				\draw[colour4b] (6) -- (7);
				\draw[colour4] (9) -- (10);
				\draw[colour4b] (9) -- (10);
				\draw[colour4] (11) -- (17);
				\draw[colour4b] (11) -- (17);
				\draw[colour4] (13) -- (14);
				\draw[colour4b] (13) -- (14);
				\draw[colour4] (15) -- (16);
				\draw[colour4b] (15) -- (16);
				\draw[colour4] (18) -- (19);
				\draw[colour4b] (18) -- (19);
				
			\end{smallestquinticperfham}
			
			\begin{smallestquinticperfham}{4}
				\draw[colour0] (0) edge[bend right, colour0] (1);
				\draw[colour0] (2) edge[bend left, colour0] (8);
				\draw[colour0] (3) -- (10);
				\draw[colour0] (4) -- (11);
				\draw[colour0] (5) -- (13);
				\draw[colour0] (6) -- (14);
				\draw[colour0] (7) -- (15);
				\draw[colour0] (9) -- (16);
				\draw[colour0] (12) -- (19);
				\draw[colour0] (17) -- (18);
				
				\draw[colour1] (0) edge[bend left, colour1] (2);
				\draw[colour1] (1) -- (5);
				\draw[colour1] (3) -- (4);
				\draw[colour1] (6) -- (13);
				\draw[colour1] (7) -- (14);
				\draw[colour1] (8) -- (16);
				\draw[colour1] (9) -- (10);
				\draw[colour1] (11) -- (17);
				\draw[colour1] (12) -- (18);
				\draw[colour1] (15) -- (19);
				
				\draw[colour2] (0) -- (3);
				\draw[colour2b] (0) -- (3);
				\draw[colour2] (1) -- (6);
				\draw[colour2b] (1) -- (6);
				\draw[colour2] (2) -- (9);
				\draw[colour2b] (2) -- (9);
				\draw[colour2] (4) -- (5);
				\draw[colour2b] (4) -- (5);
				\draw[colour2] (7) -- (8);
				\draw[colour2b] (7) -- (8);
				\draw[colour2] (10) -- (17);
				\draw[colour2b] (10) -- (17);
				\draw[colour2] (11) -- (12);
				\draw[colour2b] (11) -- (12);
				\draw[colour2] (13) -- (19);
				\draw[colour2b] (13) -- (19);
				\draw[colour2] (14) -- (15);
				\draw[colour2b] (14) -- (15);
				\draw[colour2] (16) -- (18);
				\draw[colour2b] (16) -- (18);
				
				\draw[colour3] (0) -- (4);
				\draw[colour3] (1) -- (7);
				\draw[colour3] (2) -- (10);
				\draw[colour3] (3) -- (11);
				\draw[colour3] (5) -- (6);
				\draw[colour3] (8) -- (9);
				\draw[colour3] (12) -- (13);
				\draw[colour3] (14) -- (19);
				\draw[colour3] (15) -- (18);
				\draw[colour3] (16) -- (17);
				
				\draw[colour4] (0) -- (5);
				\draw[colour4b] (0) -- (5);
				\draw[colour4] (1) edge[bend right, colour4] (8);
				\draw[colour4b] (1) edge[bend right, colour4b] (8);
				\draw[colour4] (2) -- (3);
				\draw[colour4b] (2) -- (3);
				\draw[colour4] (4) -- (12);
				\draw[colour4b] (4) -- (12);
				\draw[colour4] (6) -- (7);
				\draw[colour4b] (6) -- (7);
				\draw[colour4] (9) -- (17);
				\draw[colour4b] (9) -- (17);
				\draw[colour4] (10) -- (11);
				\draw[colour4b] (10) -- (11);
				\draw[colour4] (13) -- (14);
				\draw[colour4b] (13) -- (14);
				\draw[colour4] (15) -- (16);
				\draw[colour4b] (15) -- (16);
				\draw[colour4] (18) -- (19);
				\draw[colour4b] (18) -- (19);
				
			\end{smallestquinticperfham}
			\begin{smallestquinticperfham}{5}
				\draw[colour0] (0) edge[bend right, colour0] (1);
				\draw[colour0] (2) edge[bend left, colour0] (8);
				\draw[colour0] (3) -- (10);
				\draw[colour0] (4) -- (11);
				\draw[colour0] (5) -- (13);
				\draw[colour0] (6) -- (14);
				\draw[colour0] (7) -- (15);
				\draw[colour0] (9) -- (17);
				\draw[colour0] (12) -- (19);
				\draw[colour0] (16) -- (18);
				
				\draw[colour1] (0) edge[bend left, colour1] (2);
				\draw[colour1] (1) -- (5);
				\draw[colour1] (3) -- (4);
				\draw[colour1] (6) -- (13);
				\draw[colour1] (7) -- (8);
				\draw[colour1] (9) -- (10);
				\draw[colour1] (11) -- (17);
				\draw[colour1] (12) -- (18);
				\draw[colour1] (14) -- (19);
				\draw[colour1] (15) -- (16);
				
				\draw[colour2] (0) -- (3);
				\draw[colour2b] (0) -- (3);
				\draw[colour2] (1) -- (6);
				\draw[colour2b] (1) -- (6);
				\draw[colour2] (2) -- (9);
				\draw[colour2b] (2) -- (9);
				\draw[colour2] (4) -- (5);
				\draw[colour2b] (4) -- (5);
				\draw[colour2] (7) -- (14);
				\draw[colour2b] (7) -- (14);
				\draw[colour2] (8) -- (16);
				\draw[colour2b] (8) -- (16);
				\draw[colour2] (10) -- (17);
				\draw[colour2b] (10) -- (17);
				\draw[colour2] (11) -- (12);
				\draw[colour2b] (11) -- (12);
				\draw[colour2] (13) -- (19);
				\draw[colour2b] (13) -- (19);
				\draw[colour2] (15) -- (18);
				\draw[colour2b] (15) -- (18);
				
				\draw[colour3] (0) -- (4);
				\draw[colour3] (1) -- (7);
				\draw[colour3] (2) -- (10);
				\draw[colour3] (3) -- (11);
				\draw[colour3] (5) -- (6);
				\draw[colour3] (8) -- (9);
				\draw[colour3] (12) -- (13);
				\draw[colour3] (14) -- (15);
				\draw[colour3] (16) -- (17);
				\draw[colour3] (18) -- (19);
				
				\draw[colour4] (0) -- (5);
				\draw[colour4b] (0) -- (5);
				\draw[colour4] (1) edge[bend right, colour4] (8);
				\draw[colour4b] (1) edge[bend right, colour4b] (8);
				\draw[colour4] (2) -- (3);
				\draw[colour4b] (2) -- (3);
				\draw[colour4] (4) -- (12);
				\draw[colour4b] (4) -- (12);
				\draw[colour4] (6) -- (7);
				\draw[colour4b] (6) -- (7);
				\draw[colour4] (9) -- (16);
				\draw[colour4b] (9) -- (16);
				\draw[colour4] (10) -- (11);
				\draw[colour4b] (10) -- (11);
				\draw[colour4] (13) -- (14);
				\draw[colour4b] (13) -- (14);
				\draw[colour4] (15) -- (19);
				\draw[colour4b] (15) -- (19);
				\draw[colour4] (17) -- (18);
				\draw[colour4b] (17) -- (18);
				
			\end{smallestquinticperfham}
			\begin{smallestquinticperfham}{6}
				\draw[colour0] (0) edge[bend right, colour0] (1);
				\draw[colour0] (2) edge[bend left, colour0] (8);
				\draw[colour0] (3) -- (10);
				\draw[colour0] (4) -- (11);
				\draw[colour0] (5) -- (6);
				\draw[colour0] (7) -- (15);
				\draw[colour0] (9) -- (17);
				\draw[colour0] (12) -- (13);
				\draw[colour0] (14) -- (19);
				\draw[colour0] (16) -- (18);
				
				\draw[colour1] (0) edge[bend left, colour1] (2);
				\draw[colour1] (1) -- (5);
				\draw[colour1] (3) -- (4);
				\draw[colour1] (6) -- (13);
				\draw[colour1] (7) -- (8);
				\draw[colour1] (9) -- (16);
				\draw[colour1] (10) -- (11);
				\draw[colour1] (12) -- (19);
				\draw[colour1] (14) -- (15);
				\draw[colour1] (17) -- (18);
				
				\draw[colour2] (0) -- (3);
				\draw[colour2b] (0) -- (3);
				\draw[colour2] (1) -- (6);
				\draw[colour2b] (1) -- (6);
				\draw[colour2] (2) -- (9);
				\draw[colour2b] (2) -- (9);
				\draw[colour2] (4) -- (5);
				\draw[colour2b] (4) -- (5);
				\draw[colour2] (7) -- (14);
				\draw[colour2b] (7) -- (14);
				\draw[colour2] (8) -- (16);
				\draw[colour2b] (8) -- (16);
				\draw[colour2] (10) -- (17);
				\draw[colour2b] (10) -- (17);
				\draw[colour2] (11) -- (12);
				\draw[colour2b] (11) -- (12);
				\draw[colour2] (13) -- (19);
				\draw[colour2b] (13) -- (19);
				\draw[colour2] (15) -- (18);
				\draw[colour2b] (15) -- (18);
				
				\draw[colour3] (0) -- (4);
				\draw[colour3] (1) -- (7);
				\draw[colour3] (2) -- (10);
				\draw[colour3] (3) -- (11);
				\draw[colour3] (5) -- (13);
				\draw[colour3] (6) -- (14);
				\draw[colour3] (8) -- (9);
				\draw[colour3] (12) -- (18);
				\draw[colour3] (15) -- (19);
				\draw[colour3] (16) -- (17);
				
				\draw[colour4] (0) -- (5);
				\draw[colour4b] (0) -- (5);
				\draw[colour4] (1) edge[bend right, colour4] (8);
				\draw[colour4b] (1) edge[bend right, colour4b] (8);
				\draw[colour4] (2) -- (3);
				\draw[colour4b] (2) -- (3);
				\draw[colour4] (4) -- (12);
				\draw[colour4b] (4) -- (12);
				\draw[colour4] (6) -- (7);
				\draw[colour4b] (6) -- (7);
				\draw[colour4] (9) -- (10);
				\draw[colour4b] (9) -- (10);
				\draw[colour4] (11) -- (17);
				\draw[colour4b] (11) -- (17);
				\draw[colour4] (13) -- (14);
				\draw[colour4b] (13) -- (14);
				\draw[colour4] (15) -- (16);
				\draw[colour4b] (15) -- (16);
				\draw[colour4] (18) -- (19);
				\draw[colour4b] (18) -- (19);
				
			\end{smallestquinticperfham}
			\begin{smallestquinticperfham}{7}
				\draw[colour0] (0) edge[bend right, colour0] (1);
				\draw[colour0] (2) edge[bend left, colour0] (8);
				\draw[colour0] (3) -- (10);
				\draw[colour0] (4) -- (12);
				\draw[colour0] (5) -- (13);
				\draw[colour0] (6) -- (14);
				\draw[colour0] (7) -- (15);
				\draw[colour0] (9) -- (16);
				\draw[colour0] (11) -- (17);
				\draw[colour0] (18) -- (19);
				
				\draw[colour1] (0) edge[bend left, colour1] (2);
				\draw[colour1] (1) -- (5);
				\draw[colour1] (3) -- (4);
				\draw[colour1] (6) -- (7);
				\draw[colour1] (8) -- (16);
				\draw[colour1] (9) -- (17);
				\draw[colour1] (10) -- (11);
				\draw[colour1] (12) -- (19);
				\draw[colour1] (13) -- (14);
				\draw[colour1] (15) -- (18);
				
				\draw[colour2] (0) -- (3);
				\draw[colour2b] (0) -- (3);
				\draw[colour2] (1) -- (6);
				\draw[colour2b] (1) -- (6);
				\draw[colour2] (2) -- (9);
				\draw[colour2b] (2) -- (9);
				\draw[colour2] (4) -- (5);
				\draw[colour2b] (4) -- (5);
				\draw[colour2] (7) -- (8);
				\draw[colour2b] (7) -- (8);
				\draw[colour2] (10) -- (17);
				\draw[colour2b] (10) -- (17);
				\draw[colour2] (11) -- (12);
				\draw[colour2b] (11) -- (12);
				\draw[colour2] (13) -- (19);
				\draw[colour2b] (13) -- (19);
				\draw[colour2] (14) -- (15);
				\draw[colour2b] (14) -- (15);
				\draw[colour2] (16) -- (18);
				\draw[colour2b] (16) -- (18);
				
				\draw[colour3] (0) -- (4);
				\draw[colour3] (1) -- (7);
				\draw[colour3] (2) -- (10);
				\draw[colour3] (3) -- (11);
				\draw[colour3] (5) -- (6);
				\draw[colour3] (8) -- (9);
				\draw[colour3] (12) -- (13);
				\draw[colour3] (14) -- (19);
				\draw[colour3] (15) -- (16);
				\draw[colour3] (17) -- (18);
				
				\draw[colour4] (0) -- (5);
				\draw[colour4b] (0) -- (5);
				\draw[colour4] (1) edge[bend right, colour4] (8);
				\draw[colour4b] (1) edge[bend right, colour4b] (8);
				\draw[colour4] (2) -- (3);
				\draw[colour4b] (2) -- (3);
				\draw[colour4] (4) -- (11);
				\draw[colour4b] (4) -- (11);
				\draw[colour4] (6) -- (13);
				\draw[colour4b] (6) -- (13);
				\draw[colour4] (7) -- (14);
				\draw[colour4b] (7) -- (14);
				\draw[colour4] (9) -- (10);
				\draw[colour4b] (9) -- (10);
				\draw[colour4] (12) -- (18);
				\draw[colour4b] (12) -- (18);
				\draw[colour4] (15) -- (19);
				\draw[colour4b] (15) -- (19);
				\draw[colour4] (16) -- (17);
				\draw[colour4b] (16) -- (17);
				
			\end{smallestquinticperfham}
			
			\begin{smallestquinticperfham}{8}
				\draw[colour0] (0) edge[bend right, colour0] (1);
				\draw[colour0] (2) edge[bend left, colour0] (8);
				\draw[colour0] (3) -- (10);
				\draw[colour0] (4) -- (11);
				\draw[colour0] (5) -- (13);
				\draw[colour0] (6) -- (7);
				\draw[colour0] (9) -- (17);
				\draw[colour0] (12) -- (18);
				\draw[colour0] (14) -- (19);
				\draw[colour0] (15) -- (16);
				
				\draw[colour1] (0) edge[bend left, colour1] (2);
				\draw[colour1] (1) -- (5);
				\draw[colour1] (3) -- (4);
				\draw[colour1] (6) -- (13);
				\draw[colour1] (7) -- (8);
				\draw[colour1] (9) -- (16);
				\draw[colour1] (10) -- (17);
				\draw[colour1] (11) -- (12);
				\draw[colour1] (14) -- (15);
				\draw[colour1] (18) -- (19);
				
				\draw[colour2] (0) -- (3);
				\draw[colour2b] (0) -- (3);
				\draw[colour2] (1) -- (6);
				\draw[colour2b] (1) -- (6);
				\draw[colour2] (2) -- (9);
				\draw[colour2b] (2) -- (9);
				\draw[colour2] (4) -- (5);
				\draw[colour2b] (4) -- (5);
				\draw[colour2] (7) -- (14);
				\draw[colour2b] (7) -- (14);
				\draw[colour2] (8) -- (16);
				\draw[colour2b] (8) -- (16);
				\draw[colour2] (10) -- (11);
				\draw[colour2b] (10) -- (11);
				\draw[colour2] (12) -- (13);
				\draw[colour2b] (12) -- (13);
				\draw[colour2] (15) -- (19);
				\draw[colour2b] (15) -- (19);
				\draw[colour2] (17) -- (18);
				\draw[colour2b] (17) -- (18);
				
				\draw[colour3] (0) -- (4);
				\draw[colour3] (1) -- (7);
				\draw[colour3] (2) -- (10);
				\draw[colour3] (3) -- (11);
				\draw[colour3] (5) -- (6);
				\draw[colour3] (8) -- (9);
				\draw[colour3] (12) -- (19);
				\draw[colour3] (13) -- (14);
				\draw[colour3] (15) -- (18);
				\draw[colour3] (16) -- (17);
				
				\draw[colour4] (0) -- (5);
				\draw[colour4b] (0) -- (5);
				\draw[colour4] (1) edge[bend right, colour4] (8);
				\draw[colour4b] (1) edge[bend right, colour4b] (8);
				\draw[colour4] (2) -- (3);
				\draw[colour4b] (2) -- (3);
				\draw[colour4] (4) -- (12);
				\draw[colour4b] (4) -- (12);
				\draw[colour4] (6) -- (14);
				\draw[colour4b] (6) -- (14);
				\draw[colour4] (7) -- (15);
				\draw[colour4b] (7) -- (15);
				\draw[colour4] (9) -- (10);
				\draw[colour4b] (9) -- (10);
				\draw[colour4] (11) -- (17);
				\draw[colour4b] (11) -- (17);
				\draw[colour4] (13) -- (19);
				\draw[colour4b] (13) -- (19);
				\draw[colour4] (16) -- (18);
				\draw[colour4b] (16) -- (18);
				
			\end{smallestquinticperfham}
			\begin{smallestquinticperfham}{9}
				\draw[colour0] (0) edge[bend right, colour0] (1);
				\draw[colour0] (2) edge[bend left, colour0] (8);
				\draw[colour0] (3) -- (10);
				\draw[colour0] (4) -- (11);
				\draw[colour0] (5) -- (13);
				\draw[colour0] (6) -- (14);
				\draw[colour0] (7) -- (15);
				\draw[colour0] (9) -- (17);
				\draw[colour0] (12) -- (19);
				\draw[colour0] (16) -- (18);
				
				\draw[colour1] (0) edge[bend left, colour1] (2);
				\draw[colour1] (1) -- (5);
				\draw[colour1] (3) -- (4);
				\draw[colour1] (6) -- (13);
				\draw[colour1] (7) -- (8);
				\draw[colour1] (9) -- (16);
				\draw[colour1] (10) -- (17);
				\draw[colour1] (11) -- (12);
				\draw[colour1] (14) -- (19);
				\draw[colour1] (15) -- (18);
				
				\draw[colour2] (0) -- (3);
				\draw[colour2b] (0) -- (3);
				\draw[colour2] (1) -- (6);
				\draw[colour2b] (1) -- (6);
				\draw[colour2] (2) -- (9);
				\draw[colour2b] (2) -- (9);
				\draw[colour2] (4) -- (5);
				\draw[colour2b] (4) -- (5);
				\draw[colour2] (7) -- (14);
				\draw[colour2b] (7) -- (14);
				\draw[colour2] (8) -- (16);
				\draw[colour2b] (8) -- (16);
				\draw[colour2] (10) -- (11);
				\draw[colour2b] (10) -- (11);
				\draw[colour2] (12) -- (13);
				\draw[colour2b] (12) -- (13);
				\draw[colour2] (15) -- (19);
				\draw[colour2b] (15) -- (19);
				\draw[colour2] (17) -- (18);
				\draw[colour2b] (17) -- (18);
				
				\draw[colour3] (0) -- (4);
				\draw[colour3] (1) -- (7);
				\draw[colour3] (2) -- (10);
				\draw[colour3] (3) -- (11);
				\draw[colour3] (5) -- (6);
				\draw[colour3] (8) -- (9);
				\draw[colour3] (12) -- (18);
				\draw[colour3] (13) -- (19);
				\draw[colour3] (14) -- (15);
				\draw[colour3] (16) -- (17);
				
				\draw[colour4] (0) -- (5);
				\draw[colour4b] (0) -- (5);
				\draw[colour4] (1) edge[bend right, colour4] (8);
				\draw[colour4b] (1) edge[bend right, colour4b] (8);
				\draw[colour4] (2) -- (3);
				\draw[colour4b] (2) -- (3);
				\draw[colour4] (4) -- (12);
				\draw[colour4b] (4) -- (12);
				\draw[colour4] (6) -- (7);
				\draw[colour4b] (6) -- (7);
				\draw[colour4] (9) -- (10);
				\draw[colour4b] (9) -- (10);
				\draw[colour4] (11) -- (17);
				\draw[colour4b] (11) -- (17);
				\draw[colour4] (13) -- (14);
				\draw[colour4b] (13) -- (14);
				\draw[colour4] (15) -- (16);
				\draw[colour4b] (15) -- (16);
				\draw[colour4] (18) -- (19);
				\draw[colour4b] (18) -- (19);
				
			\end{smallestquinticperfham}
		\end{center}
		\caption{A planar quintic graph on 20 vertices with a $1$-factorisation containing exactly $k$ perfect pairs for every $k \in \{ 0, \ldots, 9 \}$.}\label{fig:20}
	\end{figure}
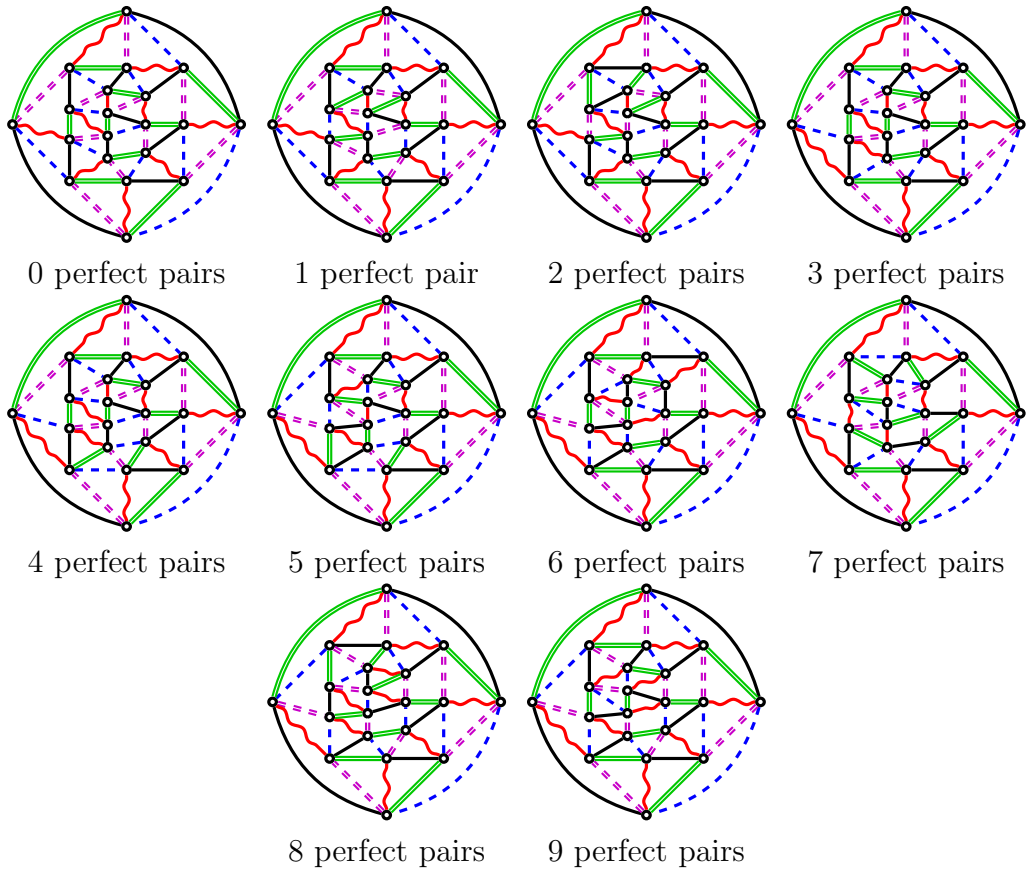
	
	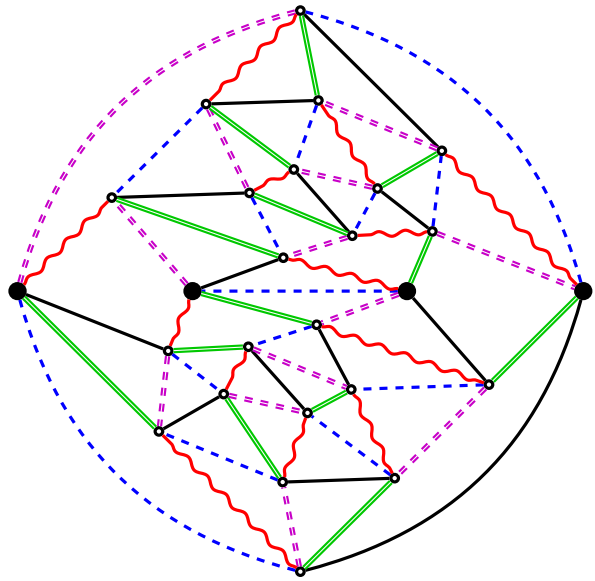
\begin{figure}
		\begin{center}
			\begin{tikzpicture}[very thick,scale=1.2,no text]
				
				\node[fill=black, inner sep=2pt] (7) at (3.1,0) {8};
				\node (8) at (0,3.1) {9};
				\node (15) at (0,-3.1) {16};
				\node[fill=black, inner sep=2pt] (16) at (-3.1,0) {17};
				
				\node[fill=black, inner sep=2pt] (6) at (1.168,0) {7};
				\node[fill=black, inner sep=2pt] (12) at (-1.182,0) {13};
				
				\node (0) at (0.845,1.134) {1};
				\node (3) at (0.197,2.107) {4};
				\node (4) at (-0.072,1.344) {5};
				\node (5) at (0.569,0.612) {6};
				\node (11) at (-0.187,0.369) {12};
				\node (13) at (0.178,-0.372) {14};
				\node (19) at (-0.570,-0.615) {20};
				\node (22) at (-0.194,-2.110) {23};
				\node (24) at (-0.840,-1.138) {25};
				\node (25) at (0.078,-1.346) {26};
				
				\node (2) at ($.5*(7)+.5*(8)$) {3};
				\node (23) at ($.5*(15)+.5*(16)$) {24};
				\node (14) at ($.6666*(7)+.3333*(15)$) {15};
				\node (21) at ($.3333*(7)+.6666*(15)$) {22};
				\node  (9) at ($.6666*(8)+.3333*(16)$) {10};
				\node (17) at ($.3333*(8)+.6666*(16)$) {18};

				\node (1) at ($.2*(0)+.2*(2)+.2*(5)+.2*(6)+.2*(7)$) {2};
				\node (10) at ($.2*(4)+.2*(5)+.2*(9)+.2*(11)+.2*(17)$) {11};
				\node (18) at ($.2*(12)+.2*(16)+.2*(19)+.2*(23)+.2*(24)$) {19};
				\node (20) at ($.2*(13)+.2*(14)+.2*(19)+.2*(21)+.2*(25)$) {21};
				
				\draw[colour0] (0) -- (1);
				\draw[colour0] (2) -- (8);
				\draw[colour0] (3) -- (9);
				\draw[colour0] (4) -- (5);
				\draw[colour0] (6) -- (14);
				\draw[colour0] (7) edge[bend left, colour0] (15);
				\draw[colour0] (10) -- (17);
				\draw[colour0] (11) -- (12);
				\draw[colour0] (13) -- (20);
				\draw[colour0] (16) -- (18);
				\draw[colour0] (19) -- (25);
				\draw[colour0] (21) -- (22);
				\draw[colour0] (23) -- (24);
				
				\draw[colour1] (0) -- (5);
				\draw[colour1] (1) -- (2);
				\draw[colour1] (3) -- (4);
				\draw[colour1] (6) -- (12);
				\draw[colour1] (7) edge[bend right, colour1] (8);
				\draw[colour1] (9) -- (17);
				\draw[colour1] (10) -- (11);
				\draw[colour1] (13) -- (19);
				\draw[colour1] (14) -- (20);
				\draw[colour1] (15) edge[bend left, colour1] (16);
				\draw[colour1] (18) -- (24);
				\draw[colour1] (21) -- (25);
				\draw[colour1] (22) -- (23);
				
				\draw[colour2] (0) -- (4);
				\draw[colour2b] (0) -- (4);
				\draw[colour2] (1) -- (7);
				\draw[colour2b] (1) -- (7);
				\draw[colour2] (2) -- (3);
				\draw[colour2b] (2) -- (3);
				\draw[colour2] (5) -- (11);
				\draw[colour2b] (5) -- (11);
				\draw[colour2] (6) -- (13);
				\draw[colour2b] (6) -- (13);
				\draw[colour2] (8) edge[bend right, colour2] (16);
				\draw[colour2b] (8) edge[bend right, colour2b] (16);
				\draw[colour2] (9) -- (10);
				\draw[colour2b] (9) -- (10);
				\draw[colour2] (12) -- (17);
				\draw[colour2b] (12) -- (17);
				\draw[colour2] (14) -- (21);
				\draw[colour2b] (14) -- (21);
				\draw[colour2] (15) -- (22);
				\draw[colour2b] (15) -- (22);
				\draw[colour2] (18) -- (23);
				\draw[colour2b] (18) -- (23);
				\draw[colour2] (19) -- (20);
				\draw[colour2b] (19) -- (20);
				\draw[colour2] (24) -- (25);
				\draw[colour2b] (24) -- (25);
				
				\draw[colour3] (0) -- (3);
				\draw[colour3] (1) -- (5);
				\draw[colour3] (2) -- (7);
				\draw[colour3] (4) -- (10);
				\draw[colour3] (6) -- (11);
				\draw[colour3] (8) -- (9);
				\draw[colour3] (12) -- (18);
				\draw[colour3] (13) -- (14);
				\draw[colour3] (15) -- (23);
				\draw[colour3] (16) -- (17);
				\draw[colour3] (19) -- (24);
				\draw[colour3] (20) -- (21);
				\draw[colour3] (22) -- (25);
				
				\draw[colour4] (0) -- (2);
				\draw[colour4b] (0) -- (2);
				\draw[colour4] (1) -- (6);
				\draw[colour4b] (1) -- (6);
				\draw[colour4] (3) -- (8);
				\draw[colour4b] (3) -- (8);
				\draw[colour4] (4) -- (9);
				\draw[colour4b] (4) -- (9);
				\draw[colour4] (5) -- (10);
				\draw[colour4b] (5) -- (10);
				\draw[colour4] (7) -- (14);
				\draw[colour4b] (7) -- (14);
				\draw[colour4] (11) -- (17);
				\draw[colour4b] (11) -- (17);
				\draw[colour4] (12) -- (13);
				\draw[colour4b] (12) -- (13);
				\draw[colour4] (15) -- (21);
				\draw[colour4b] (15) -- (21);
				\draw[colour4] (16) -- (23);
				\draw[colour4b] (16) -- (23);
				\draw[colour4] (18) -- (19);
				\draw[colour4b] (18) -- (19);
				\draw[colour4] (20) -- (25);
				\draw[colour4b] (20) -- (25);
				\draw[colour4] (22) -- (24);
				\draw[colour4b] (22) -- (24);
			\end{tikzpicture}
		\end{center}
		\caption{A planar quintic perfectly hamiltonian graph of connectivity~$4$. The 4-vertex-cut (shown slightly larger and coloured black) can be found along the equator of the figure.}\label{fig:quintic_ph_4conn}
	\end{figure}
	
	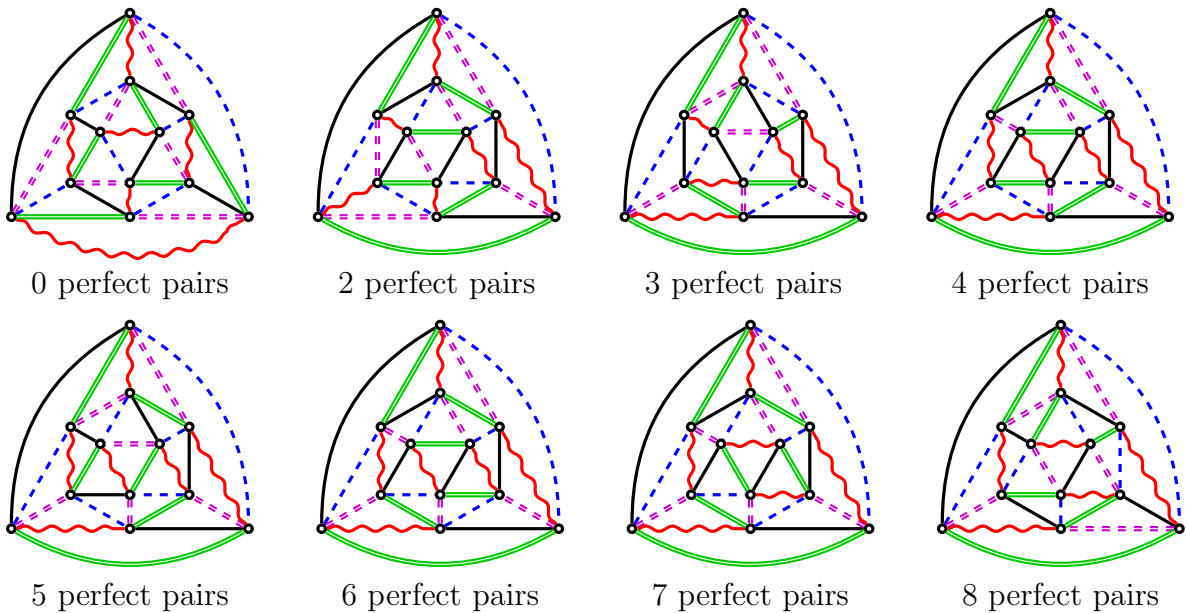
\begin{figure}
		\begin{center}
			\newenvironment{icosahedron}[1]{
				\begin{tikzpicture}[very thick,scale=.9]
					
					\node[rectangle,draw=none,font=\normalsize] at (270:2) {#1 perfect pairs};
					
					\begin{scope}[no text]
						\node (0) at (90:2) {1};
						\node (1) at (210:2) {2};
						\node (2) at (330:2) {3};
						\node (3) at (30:1) {4};
						\node (4) at (90:1) {5};
						\node (5) at (150:1) {6};
						\node (6) at (210:1) {7};
						\node (7) at (270:1) {8};
						\node (8) at (330:1) {9};
						\node (9) at (30:.5) {10};
						\node (10) at (150:.5) {11};
						\node (11) at (270:.5) {12};
					\end{scope}
				}{
			\end{tikzpicture}}
			
			\begin{icosahedron}{0}
				\draw[colour0] (0) edge[bend right, colour0] (1);
				\draw[colour0] (2) -- (8);
				\draw[colour0] (3) -- (4);
				\draw[colour0] (5) -- (10);
				\draw[colour0] (6) -- (7);
				\draw[colour0] (9) -- (11);
				
				\draw[colour1] (0) edge[bend left, colour1] (2);
				\draw[colour1] (1) -- (6);
				\draw[colour1] (3) -- (9);
				\draw[colour1] (4) -- (5);
				\draw[colour1] (7) -- (8);
				\draw[colour1] (10) -- (11);
				
				\draw[colour2] (0) -- (3);
				\draw[colour2b] (0) -- (3);
				\draw[colour2] (1) -- (5);
				\draw[colour2b] (1) -- (5);
				\draw[colour2] (2) -- (7);
				\draw[colour2b] (2) -- (7);
				\draw[colour2] (4) -- (10);
				\draw[colour2b] (4) -- (10);
				\draw[colour2] (6) -- (11);
				\draw[colour2b] (6) -- (11);
				\draw[colour2] (8) -- (9);
				\draw[colour2b] (8) -- (9);
				
				\draw[colour3] (0) -- (4);
				\draw (1) edge[bend right, colour3]  (2);
				\draw[colour3] (3) -- (8);
				\draw[colour3] (5) -- (6);
				\draw[colour3] (7) -- (11);
				\draw[colour3] (9) -- (10);
				
				\draw[colour4] (0) -- (5);
				\draw[colour4b] (0) -- (5);
				\draw[colour4] (1) -- (7);
				\draw[colour4b] (1) -- (7);
				\draw[colour4] (2) -- (3);
				\draw[colour4b] (2) -- (3);
				\draw[colour4] (4) -- (9);
				\draw[colour4b] (4) -- (9);
				\draw[colour4] (6) -- (10);
				\draw[colour4b] (6) -- (10);
				\draw[colour4] (8) -- (11);
				\draw[colour4b] (8) -- (11);
			\end{icosahedron}\hfill
			\begin{icosahedron}{2}
				\draw[colour0] (0) edge[bend right, colour0] (1);
				\draw[colour0] (2) -- (7);
				\draw[colour0] (3) -- (8);
				\draw[colour0] (4) -- (5);
				\draw[colour0] (6) -- (10);
				\draw[colour0] (9) -- (11);
				
				\draw[colour1] (0) edge[bend left, colour1] (2);
				\draw[colour1] (1) -- (5);
				\draw[colour1] (3) -- (9);
				\draw[colour1] (4) -- (10);
				\draw[colour1] (6) -- (7);
				\draw[colour1] (8) -- (11);
				
				\draw[colour2] (0) -- (3);
				\draw[colour2b] (0) -- (3);
				\draw[colour2] (1) -- (7);
				\draw[colour2b] (1) -- (7);
				\draw[colour2] (2) -- (8);
				\draw[colour2b] (2) -- (8);
				\draw[colour2] (4) -- (9);
				\draw[colour2b] (4) -- (9);
				\draw[colour2] (5) -- (6);
				\draw[colour2b] (5) -- (6);
				\draw[colour2] (10) -- (11);
				\draw[colour2b] (10) -- (11);
				
				\draw[colour3] (0) -- (4);
				\draw[colour3] (1) -- (6);
				\draw[colour3] (2) -- (3);
				\draw[colour3] (5) -- (10);
				\draw[colour3] (7) -- (11);
				\draw[colour3] (8) -- (9);
				
				\draw[colour4] (0) -- (5);
				\draw[colour4b] (0) -- (5);
				\draw[colour4] (1) edge[bend right, colour4] (2);
				\draw[colour4b] (1) edge[bend right, colour4b] (2);
				\draw[colour4] (3) -- (4);
				\draw[colour4b] (3) -- (4);
				\draw[colour4] (6) -- (11);
				\draw[colour4b] (6) -- (11);
				\draw[colour4] (7) -- (8);
				\draw[colour4b] (7) -- (8);
				\draw[colour4] (9) -- (10);
				\draw[colour4b] (9) -- (10);
			\end{icosahedron}\hfill
			\begin{icosahedron}{3}
				\draw[colour0] (0) edge[bend right, colour0] (1);
				\draw[colour0] (2) -- (7);
				\draw[colour0] (3) -- (8);
				\draw[colour0] (4) -- (9);
				\draw[colour0] (5) -- (6);
				\draw[colour0] (10) -- (11);
				
				\draw[colour1] (0) edge[bend left, colour1] (2);
				\draw[colour1] (1) -- (5);
				\draw[colour1] (3) -- (4);
				\draw[colour1] (6) -- (10);
				\draw[colour1] (7) -- (8);
				\draw[colour1] (9) -- (11);
				
				\draw[colour2] (0) -- (3);
				\draw[colour2b] (0) -- (3);
				\draw[colour2] (1) -- (6);
				\draw[colour2b] (1) -- (6);
				\draw[colour2] (2) -- (8);
				\draw[colour2b] (2) -- (8);
				\draw[colour2] (4) -- (5);
				\draw[colour2b] (4) -- (5);
				\draw[colour2] (7) -- (11);
				\draw[colour2b] (7) -- (11);
				\draw[colour2] (9) -- (10);
				\draw[colour2b] (9) -- (10);
				
				\draw[colour3] (0) -- (4);
				\draw[colour3] (1) -- (7);
				\draw[colour3] (2) -- (3);
				\draw[colour3] (5) -- (10);
				\draw[colour3] (6) -- (11);
				\draw[colour3] (8) -- (9);
				
				\draw[colour4] (0) -- (5);
				\draw[colour4b] (0) -- (5);
				\draw[colour4] (1) edge[bend right, colour4] (2);
				\draw[colour4b] (1) edge[bend right, colour4b] (2);
				\draw[colour4] (3) -- (9);
				\draw[colour4b] (3) -- (9);
				\draw[colour4] (4) -- (10);
				\draw[colour4b] (4) -- (10);
				\draw[colour4] (6) -- (7);
				\draw[colour4b] (6) -- (7);
				\draw[colour4] (8) -- (11);
				\draw[colour4b] (8) -- (11);
			\end{icosahedron}\hfill
			\begin{icosahedron}{4}
				\draw[colour0] (0) edge[bend right, colour0] (1);
				\draw[colour0] (2) -- (7);
				\draw[colour0] (3) -- (8);
				\draw[colour0] (4) -- (5);
				\draw[colour0] (6) -- (10);
				\draw[colour0] (9) -- (11);
				
				\draw[colour1] (0) edge[bend left, colour1] (2);
				\draw[colour1] (1) -- (5);
				\draw[colour1] (3) -- (9);
				\draw[colour1] (4) -- (10);
				\draw[colour1] (6) -- (7);
				\draw[colour1] (8) -- (11);
				
				\draw[colour2] (0) -- (3);
				\draw[colour2b] (0) -- (3);
				\draw[colour2] (1) -- (6);
				\draw[colour2b] (1) -- (6);
				\draw[colour2] (2) -- (8);
				\draw[colour2b] (2) -- (8);
				\draw[colour2] (4) -- (9);
				\draw[colour2b] (4) -- (9);
				\draw[colour2] (5) -- (10);
				\draw[colour2b] (5) -- (10);
				\draw[colour2] (7) -- (11);
				\draw[colour2b] (7) -- (11);
				
				\draw[colour3] (0) -- (4);
				\draw[colour3] (1) -- (7);
				\draw[colour3] (2) -- (3);
				\draw[colour3] (5) -- (6);
				\draw[colour3] (8) -- (9);
				\draw[colour3] (10) -- (11);
				
				\draw[colour4] (0) -- (5);
				\draw[colour4b] (0) -- (5);
				\draw[colour4] (1) edge[bend right, colour4] (2);
				\draw[colour4b] (1) edge[bend right, colour4b] (2);
				\draw[colour4] (3) -- (4);
				\draw[colour4b] (3) -- (4);
				\draw[colour4] (6) -- (11);
				\draw[colour4b] (6) -- (11);
				\draw[colour4] (7) -- (8);
				\draw[colour4b] (7) -- (8);
				\draw[colour4] (9) -- (10);
				\draw[colour4b] (9) -- (10);
			\end{icosahedron}
			\medskip
			
			\begin{icosahedron}{5}
				\draw[colour0] (0) edge[bend right, colour0] (1);
				\draw[colour0] (2) -- (7);
				\draw[colour0] (3) -- (8);
				\draw[colour0] (4) -- (9);
				\draw[colour0] (5) -- (10);
				\draw[colour0] (6) -- (11);
				
				\draw[colour1] (0) edge[bend left, colour1] (2);
				\draw[colour1] (1) -- (5);
				\draw[colour1] (3) -- (9);
				\draw[colour1] (4) -- (10);
				\draw[colour1] (6) -- (7);
				\draw[colour1] (8) -- (11);
				
				\draw[colour2] (0) -- (3);
				\draw[colour2b] (0) -- (3);
				\draw[colour2] (1) -- (6);
				\draw[colour2b] (1) -- (6);
				\draw[colour2] (2) -- (8);
				\draw[colour2b] (2) -- (8);
				\draw[colour2] (4) -- (5);
				\draw[colour2b] (4) -- (5);
				\draw[colour2] (7) -- (11);
				\draw[colour2b] (7) -- (11);
				\draw[colour2] (9) -- (10);
				\draw[colour2b] (9) -- (10);
				
				\draw[colour3] (0) -- (4);
				\draw[colour3] (1) -- (7);
				\draw[colour3] (2) -- (3);
				\draw[colour3] (5) -- (6);
				\draw[colour3] (8) -- (9);
				\draw[colour3] (10) -- (11);
				
				\draw[colour4] (0) -- (5);
				\draw[colour4b] (0) -- (5);
				\draw[colour4] (1) edge[bend right, colour4] (2);
				\draw[colour4b] (1) edge[bend right, colour4b] (2);
				\draw[colour4] (3) -- (4);
				\draw[colour4b] (3) -- (4);
				\draw[colour4] (6) -- (10);
				\draw[colour4b] (6) -- (10);
				\draw[colour4] (7) -- (8);
				\draw[colour4b] (7) -- (8);
				\draw[colour4] (9) -- (11);
				\draw[colour4b] (9) -- (11);
			\end{icosahedron}\hfill
			\begin{icosahedron}{6}
				\draw[colour0] (0) edge[bend right, colour0] (1);
				\draw[colour0] (2) -- (7);
				\draw[colour0] (3) -- (8);
				\draw[colour0] (4) -- (5);
				\draw[colour0] (6) -- (10);
				\draw[colour0] (9) -- (11);
				
				\draw[colour1] (0) edge[bend left, colour1] (2);
				\draw[colour1] (1) -- (5);
				\draw[colour1] (3) -- (9);
				\draw[colour1] (4) -- (10);
				\draw[colour1] (6) -- (11);
				\draw[colour1] (7) -- (8);
				
				\draw[colour2] (0) -- (3);
				\draw[colour2b] (0) -- (3);
				\draw[colour2] (1) -- (6);
				\draw[colour2b] (1) -- (6);
				\draw[colour2] (2) -- (8);
				\draw[colour2b] (2) -- (8);
				\draw[colour2] (4) -- (9);
				\draw[colour2b] (4) -- (9);
				\draw[colour2] (5) -- (10);
				\draw[colour2b] (5) -- (10);
				\draw[colour2] (7) -- (11);
				\draw[colour2b] (7) -- (11);
				
				\draw[colour3] (0) -- (4);
				\draw[colour3] (1) -- (7);
				\draw[colour3] (2) -- (3);
				\draw[colour3] (5) -- (6);
				\draw[colour3] (8) -- (9);
				\draw[colour3] (10) -- (11);
				
				\draw[colour4] (0) -- (5);
				\draw[colour4b] (0) -- (5);
				\draw[colour4] (1) edge[bend right, colour4] (2);
				\draw[colour4b] (1) edge[bend right, colour4b] (2);
				\draw[colour4] (3) -- (4);
				\draw[colour4b] (3) -- (4);
				\draw[colour4] (6) -- (7);
				\draw[colour4b] (6) -- (7);
				\draw[colour4] (8) -- (11);
				\draw[colour4b] (8) -- (11);
				\draw[colour4] (9) -- (10);
				\draw[colour4b] (9) -- (10);
			\end{icosahedron}\hfill
			\begin{icosahedron}{7}
				\draw[colour0] (0) edge[bend right, colour0] (1);
				\draw[colour0] (2) -- (7);
				\draw[colour0] (3) -- (8);
				\draw[colour0] (4) -- (5);
				\draw[colour0] (6) -- (10);
				\draw[colour0] (9) -- (11);
				
				\draw[colour1] (0) edge[bend left, colour1] (2);
				\draw[colour1] (1) -- (5);
				\draw[colour1] (3) -- (9);
				\draw[colour1] (4) -- (10);
				\draw[colour1] (6) -- (11);
				\draw[colour1] (7) -- (8);
				
				\draw[colour2] (0) -- (3);
				\draw[colour2b] (0) -- (3);
				\draw[colour2] (1) -- (6);
				\draw[colour2b] (1) -- (6);
				\draw[colour2] (2) -- (8);
				\draw[colour2b] (2) -- (8);
				\draw[colour2] (4) -- (9);
				\draw[colour2b] (4) -- (9);
				\draw[colour2] (5) -- (10);
				\draw[colour2b] (5) -- (10);
				\draw[colour2] (7) -- (11);
				\draw[colour2b] (7) -- (11);
				
				\draw[colour3] (0) -- (4);
				\draw[colour3] (1) -- (7);
				\draw[colour3] (2) -- (3);
				\draw[colour3] (5) -- (6);
				\draw[colour3] (8) -- (11);
				\draw[colour3] (9) -- (10);
				
				\draw[colour4] (0) -- (5);
				\draw[colour4b] (0) -- (5);
				\draw[colour4] (1) edge[bend right, colour4] (2);
				\draw[colour4b] (1) edge[bend right, colour4b] (2);
				\draw[colour4] (3) -- (4);
				\draw[colour4b] (3) -- (4);
				\draw[colour4] (6) -- (7);
				\draw[colour4b] (6) -- (7);
				\draw[colour4] (8) -- (9);
				\draw[colour4b] (8) -- (9);
				\draw[colour4] (10) -- (11);
				\draw[colour4b] (10) -- (11);
			\end{icosahedron}\hfill
			\begin{icosahedron}{8}
				\draw[colour0] (0) edge[bend right, colour0] (1);
				\draw[colour0] (2) -- (8);
				\draw[colour0] (3) -- (4);
				\draw[colour0] (5) -- (10);
				\draw[colour0] (6) -- (7);
				\draw[colour0] (9) -- (11);
				
				\draw[colour1] (0) edge[bend left, colour1] (2);
				\draw[colour1] (1) -- (5);
				\draw[colour1] (3) -- (8);
				\draw[colour1] (4) -- (9);
				\draw[colour1] (6) -- (10);
				\draw[colour1] (7) -- (11);
				
				\draw[colour2] (0) -- (3);
				\draw[colour2b] (0) -- (3);
				\draw[colour2] (1) -- (6);
				\draw[colour2b] (1) -- (6);
				\draw[colour2] (2) -- (7);
				\draw[colour2b] (2) -- (7);
				\draw[colour2] (4) -- (5);
				\draw[colour2b] (4) -- (5);
				\draw[colour2] (8) -- (9);
				\draw[colour2b] (8) -- (9);
				\draw[colour2] (10) -- (11);
				\draw[colour2b] (10) -- (11);
				
				\draw[colour3] (0) -- (4);
				\draw[colour3] (1) -- (7);
				\draw[colour3] (2) -- (3);
				\draw[colour3] (5) -- (6);
				\draw[colour3] (8) -- (11);
				\draw[colour3] (9) -- (10);
				
				\draw[colour4] (0) -- (5);
				\draw[colour4b] (0) -- (5);
				\draw[colour4] (1) edge[bend right, colour4] (2);
				\draw[colour4b] (1) edge[bend right, colour4b] (2);
				\draw[colour4] (3) -- (9);
				\draw[colour4b] (3) -- (9);
				\draw[colour4] (4) -- (10);
				\draw[colour4b] (4) -- (10);
				\draw[colour4] (6) -- (11);
				\draw[colour4b] (6) -- (11);
				\draw[colour4] (7) -- (8);
				\draw[colour4b] (7) -- (8);
			\end{icosahedron}
		\end{center}
		\caption{$1$-factorisations of the icosahedron with exactly $k$ perfect pairs for every $k \in \{ 0, 2, \ldots, 8 \}$. The icosahedron has no $1$-factorisations with 1, 9 or 10 perfect pairs.}\label{fig:icosahedron}
	\end{figure}

	\begin{figure}
		\begin{center}
			\begin{tikzpicture}[very thick,scale=.75,no text] 
				
				\node (0) at (2.12,0) {1};
				\node (2) at (1,0) {3};
				\node (3) at (0.7,0.9) {4};
				\node (5) at (3,0) {6};
				\node (6) at (0,-3) {7};
				\node (7) at (0,-2.12) {8};
				\node (8) at (0.571,-1) {9};
				\node (9) at (0.340,0) {10};
				\node (10) at (0,0.744) {11};
				\node (11) at (0,2.12) {12};
				\node (12) at (0,3) {13};
				\node (13) at (-3,0) {14};
				\node (15) at (-0.7,-0.9) {16};
				\node (16) at (0,-0.738) {17};
				\node (17) at (-0.344,0) {18};
				\node (18) at (-0.574,1) {19};
				\node (20) at (-2.12,0) {21};
				\node (21) at (-1,0) {22};
				
				\node (1) at ($.5*(5)+.5*(6)$) {2};
				\node (4) at ($.5*(5)+.5*(12)$) {5};
				\node (14) at ($.5*(6)+.5*(13)$) {15};
				\node (19) at ($.5*(12)+.5*(13)$) {20};
				
				\draw[colour0] (0) -- (1);
				\draw[colour0] (2) -- (8);
				\draw[colour0] (3) -- (10);
				\draw[colour0] (4) -- (11);
				\draw[colour0] (5) edge[bend right, colour0] (12);
				\draw[colour0] (6) -- (7);
				\draw[colour0] (9) -- (17);
				\draw[colour0] (13) -- (19);
				\draw[colour0] (14) -- (20);
				\draw[colour0] (15) -- (16);
				\draw[colour0] (18) -- (21);
				
				\draw[colour1] (0) -- (2);
				\draw[colour1] (1) -- (5);
				\draw[colour1] (3) -- (4);
				\draw[colour1] (6) edge[bend left, colour1] (13);
				\draw[colour1] (7) -- (8);
				\draw[colour1] (9) -- (10);
				\draw[colour1] (11) -- (12);
				\draw[colour1] (14) -- (15);
				\draw[colour1] (16) -- (17);
				\draw[colour1] (18) -- (19);
				\draw[colour1] (20) -- (21);
				
				\draw[colour2] (0) -- (3);
				\draw[colour2b] (0) -- (3);
				\draw[colour2] (1) -- (6);
				\draw[colour2b] (1) -- (6);
				\draw[colour2] (2) -- (9);
				\draw[colour2b] (2) -- (9);
				\draw[colour2] (4) -- (5);
				\draw[colour2b] (4) -- (5);
				\draw[colour2] (7) -- (14);
				\draw[colour2b] (7) -- (14);
				\draw[colour2] (8) -- (16);
				\draw[colour2b] (8) -- (16);
				\draw[colour2] (10) -- (17);
				\draw[colour2b] (10) -- (17);
				\draw[colour2] (11) -- (18);
				\draw[colour2b] (11) -- (18);
				\draw[colour2] (12) edge[bend right, colour2] (13);
				\draw[colour2b] (12) edge[bend right, colour2b] (13);
				\draw[colour2] (15) -- (21);
				\draw[colour2b] (15) -- (21);
				\draw[colour2] (19) -- (20);
				\draw[colour2b] (19) -- (20);
				
				\draw[colour3] (0) -- (4);
				\draw[colour3] (1) -- (7);
				\draw[colour3] (2) -- (10);
				\draw[colour3] (3) -- (11);
				\draw (5) edge[bend left,colour3] (6);
				\draw[colour3] (8) -- (9);
				\draw[colour3] (12) -- (19);
				\draw[colour3] (13) -- (14);
				\draw[colour3] (15) -- (20);
				\draw[colour3] (16) -- (21);
				\draw[colour3] (17) -- (18);
				
				\draw[colour4] (0) -- (5);
				\draw[colour4b] (0) -- (5);
				\draw[colour4] (1) -- (8);
				\draw[colour4b] (1) -- (8);
				\draw[colour4] (2) -- (3);
				\draw[colour4b] (2) -- (3);
				\draw[colour4] (4) -- (12);
				\draw[colour4b] (4) -- (12);
				\draw[colour4] (6) -- (14);
				\draw[colour4b] (6) -- (14);
				\draw[colour4] (7) -- (15);
				\draw[colour4b] (7) -- (15);
				\draw[colour4] (9) -- (16);
				\draw[colour4b] (9) -- (16);
				\draw[colour4] (10) -- (18);
				\draw[colour4b] (10) -- (18);
				\draw[colour4] (11) -- (19);
				\draw[colour4b] (11) -- (19);
				\draw[colour4] (13) -- (20);
				\draw[colour4b] (13) -- (20);
				\draw[colour4] (17) -- (21);
				\draw[colour4b] (17) -- (21);
			\end{tikzpicture}
			\begin{tikzpicture}[very thick,rotate=90,scale=.75,no text] 
				
				\node (1) at (3.620,0.002) {2};
				\node (3) at (0.584,-0.952) {4};
				\node (4) at (1.086,-0.328) {5};
				\node (5) at (2.323,-0.006) {6};
				\node (7) at (-1.811,3.137) {8};
				\node (8) at (-1.811,-3.134) {9};
				\node (9) at (-1.167,-2.011) {10};
				\node (10) at (-0.356,-1.282) {11};
				\node (11) at (-0.054,-0.426) {12};
				\node (12) at (0.394,0.167) {13};
				\node (13) at (1.286,0.334) {14};
				\node (14) at (0.529,0.987) {15};
				\node (16) at (-1.155,2.016) {17};
				\node (19) at (-0.830,-0.780) {20};
				\node (20) at (-0.343,0.254) {21};
				\node (21) at (-0.257,1.102) {22};
				\node (22) at (-0.927,0.946) {23};
				\node (23) at (-1.120,-0.036) {24};
				
				\node (0) at ($0.66666*(1) + 0.33333*(8)$) {1};
				\node (2) at ($0.66666*(8) + 0.33333*(1)$) {3};
				
				\node (6) at ($0.66666*(1) + 0.33333*(7)$) {7};
				\node (15) at ($0.66666*(7) + 0.33333*(1)$) {16};
				
				\node (17) at ($0.66666*(7) + 0.33333*(8)$) {18};
				\node (18) at ($0.66666*(8) + 0.33333*(7)$) {19};
				
				\draw[colour0] (0) -- (1);
				\draw[colour0] (2) -- (8);
				\draw[colour0] (3) -- (10);
				\draw[colour0] (4) -- (12);
				\draw[colour0] (5) -- (13);
				\draw[colour0] (6) -- (14);
				\draw[colour0] (7) -- (16);
				\draw[colour0] (9) -- (19);
				\draw[colour0] (11) -- (20);
				\draw[colour0] (15) -- (21);
				\draw[colour0] (17) -- (22);
				\draw[colour0] (18) -- (23);
				
				\draw[colour1] (0) -- (2);
				\draw[colour1] (1) -- (5);
				\draw[colour1] (3) -- (4);
				\draw[colour1] (6) -- (13);
				\draw[colour1] (7) -- (17);
				\draw[colour1] (8) -- (18);
				\draw[colour1] (9) -- (10);
				\draw[colour1] (11) -- (12);
				\draw[colour1] (14) -- (15);
				\draw[colour1] (16) -- (21);
				\draw[colour1] (19) -- (23);
				\draw[colour1] (20) -- (22);
				
				\draw[colour2] (0) -- (3);
				\draw[colour2b] (0) -- (3);
				\draw[colour2] (1) -- (6);
				\draw[colour2b] (1) -- (6);
				\draw[colour2] (2) -- (10);
				\draw[colour2b] (2) -- (10);
				\draw[colour2] (4) -- (5);
				\draw[colour2b] (4) -- (5);
				\draw[colour2] (7) edge[bend right, colour2] (8);
				\draw[colour2b] (7) edge[bend right, colour2b] (8);
				\draw[colour2] (9) -- (18);
				\draw[colour2b] (9) -- (18);
				\draw[colour2] (11) -- (19);
				\draw[colour2b] (11) -- (19);
				\draw[colour2] (12) -- (20);
				\draw[colour2b] (12) -- (20);
				\draw[colour2] (13) -- (14);
				\draw[colour2b] (13) -- (14);
				\draw[colour2] (15) -- (16);
				\draw[colour2b] (15) -- (16);
				\draw[colour2] (17) -- (23);
				\draw[colour2b] (17) -- (23);
				\draw[colour2] (21) -- (22);
				\draw[colour2b] (21) -- (22);
				
				\draw[colour3] (0) -- (4);
				\path (1) edge[bend left,colour3] (8);
				\draw[colour3] (2) -- (9);
				\draw[colour3] (3) -- (11);
				\draw[colour3] (5) -- (6);
				\draw[colour3] (7) -- (15);
				\draw[colour3] (10) -- (19);
				\draw[colour3] (12) -- (13);
				\draw[colour3] (14) -- (21);
				\draw[colour3] (16) -- (22);
				\draw[colour3] (17) -- (18);
				\draw[colour3] (20) -- (23);
				
				\draw[colour4] (0) -- (5);
				\draw[colour4b] (0) -- (5);
				\draw[colour4] (1) edge[bend right, colour4] (7);
				\draw[colour4b] (1) edge[bend right, colour4b] (7);
				\draw[colour4] (2) -- (3);
				\draw[colour4b] (2) -- (3);
				\draw[colour4] (4) -- (13);
				\draw[colour4b] (4) -- (13);
				\draw[colour4] (6) -- (15);
				\draw[colour4b] (6) -- (15);
				\draw[colour4] (8) -- (9);
				\draw[colour4b] (8) -- (9);
				\draw[colour4] (10) -- (11);
				\draw[colour4b] (10) -- (11);
				\draw[colour4] (12) -- (14);
				\draw[colour4b] (12) -- (14);
				\draw[colour4] (16) -- (17);
				\draw[colour4b] (16) -- (17);
				\draw[colour4] (18) -- (19);
				\draw[colour4b] (18) -- (19);
				\draw[colour4] (20) -- (21);
				\draw[colour4b] (20) -- (21);
				\draw[colour4] (22) -- (23);
				\draw[colour4b] (22) -- (23);
				
			\end{tikzpicture}
			\begin{tikzpicture}[very thick,scale=.75,no text] 
				
				\node (1) at (2.25,0) {2};
				\node (2) at (3.16,0) {3};
				\node (3) at (0,-3.16) {4};
				\node (4) at (0,-2.25) {5};
				\node (10) at (-3.16,0) {11};
				\node (21) at (-2.25,0) {22};
				\node (9) at (0,3.16) {10};
				\node (19) at (0,2.25) {20};
				
				\node (5) at (0.6,-1.4) {6};
				\node (12) at (-0.6,-1.4) {13};
				\node (18) at (0.6,1.4) {19};
				\node (25) at (-0.6,1.4) {26};
				
				\node (13) at (0,-0.9) {14};
				\node (17) at (0,0.9) {18};
				
				\node (6) at (1.2,-0.5) {7};
				\node (7) at (1.2,0.5) {8};
				\node (22) at (-1.2,-0.5) {23};
				\node (24) at (-1.2,0.5) {25};
				
				\node (14) at (0,-0.35) {15};
				\node (15) at (0.8,0) {16};
				\node (16) at (0,0.35) {17};
				\node (23) at (-0.8,0) {24};
				
				\node (0) at ($.5*(2)+.5*(3)$) {1};
				\node (8) at ($.5*(2)+.5*(9)$) {9};
				\node (20) at ($.5*(10)+.5*(9)$) {21};
				\node (11) at ($.5*(10)+.5*(3)$) {12};
				
				\draw[colour0] (0) -- (1);
				\draw[colour0] (2) -- (8);
				\draw[colour0] (3) edge[bend left, colour0] (10);
				\draw[colour0] (4) -- (12);
				\draw[colour0] (5) -- (6);
				\draw[colour0] (7) -- (17);
				\draw[colour0] (9) -- (20);
				\draw[colour0] (11) -- (21);
				\draw[colour0] (13) -- (22);
				\draw[colour0] (14) -- (15);
				\draw[colour0] (16) -- (23);
				\draw[colour0] (18) -- (19);
				\draw[colour0] (24) -- (25);
				
				\draw[colour1] (0) -- (2);
				\draw[colour1] (1) -- (6);
				\draw[colour1] (3) -- (11);
				\draw[colour1] (4) -- (5);
				\draw[colour1] (7) -- (18);
				\draw[colour1] (8) -- (9);
				\draw[colour1] (10) -- (20);
				\draw[colour1] (12) -- (22);
				\draw[colour1] (13) -- (14);
				\draw[colour1] (15) -- (23);
				\draw[colour1] (16) -- (17);
				\draw[colour1] (19) -- (25);
				\draw[colour1] (21) -- (24);
				
				\draw[colour2] (0) -- (3);
				\draw[colour2b] (0) -- (3);
				\draw[colour2] (1) -- (7);
				\draw[colour2b] (1) -- (7);
				\draw[colour2] (2) edge[bend right, colour2] (9);
				\draw[colour2b] (2) edge[bend right, colour2b] (9);
				\draw[colour2] (4) -- (11);
				\draw[colour2b] (4) -- (11);
				\draw[colour2] (5) -- (12);
				\draw[colour2b] (5) -- (12);
				\draw[colour2] (6) -- (13);
				\draw[colour2b] (6) -- (13);
				\draw[colour2] (8) -- (19);
				\draw[colour2b] (8) -- (19);
				\draw[colour2] (10) -- (21);
				\draw[colour2b] (10) -- (21);
				\draw[colour2] (14) -- (22);
				\draw[colour2b] (14) -- (22);
				\draw[colour2] (15) -- (16);
				\draw[colour2b] (15) -- (16);
				\draw[colour2] (17) -- (18);
				\draw[colour2b] (17) -- (18);
				\draw[colour2] (20) -- (25);
				\draw[colour2b] (20) -- (25);
				\draw[colour2] (23) -- (24);
				\draw[colour2b] (23) -- (24);
				
				\draw[colour3] (0) -- (4);
				\draw[colour3] (1) -- (8);
				\path (2) edge[bend left, colour3] (3);
				\draw[colour3] (5) -- (13);
				\draw[colour3] (6) -- (15);
				\draw[colour3] (7) -- (16);
				\path (9) edge[bend right, colour3] (10);
				\draw[colour3] (11) -- (12);
				\draw[colour3] (14) -- (23);
				\draw[colour3] (17) -- (24);
				\draw[colour3] (18) -- (25);
				\draw[colour3] (19) -- (20);
				\draw[colour3] (21) -- (22);
				
				\draw[colour4] (0) -- (5);
				\draw[colour4b] (0) -- (5);
				\draw[colour4] (1) -- (2);
				\draw[colour4b] (1) -- (2);
				\draw[colour4] (3) -- (4);
				\draw[colour4b] (3) -- (4);
				\draw[colour4] (6) -- (14);
				\draw[colour4b] (6) -- (14);
				\draw[colour4] (7) -- (15);
				\draw[colour4b] (7) -- (15);
				\draw[colour4] (8) -- (18);
				\draw[colour4b] (8) -- (18);
				\draw[colour4] (9) -- (19);
				\draw[colour4b] (9) -- (19);
				\draw[colour4] (10) -- (11);
				\draw[colour4b] (10) -- (11);
				\draw[colour4] (12) -- (13);
				\draw[colour4b] (12) -- (13);
				\draw[colour4] (16) -- (24);
				\draw[colour4b] (16) -- (24);
				\draw[colour4] (17) -- (25);
				\draw[colour4b] (17) -- (25);
				\draw[colour4] (20) -- (21);
				\draw[colour4b] (20) -- (21);
				\draw[colour4] (22) -- (23);
				\draw[colour4b] (22) -- (23);
			\end{tikzpicture}
			
			\begin{tikzpicture}[very thick,scale=.75,no text] 
				
				\node (0) at (3.347,-0.000) {1};
				\node (1) at (-0.001,-3.348) {2};
				\node (3) at (2.304,0.255) {4};
				\node (5) at (-0.001,3.348) {6};
				\node (6) at (-3.349,-0.000) {7};
				\node (8) at (0.104,-2.251) {9};
				\node (9) at (0.475,-1.371) {10};
				\node (10) at (1.847,-0.595) {11};
				\node (11) at (1.336,-0.095) {12};
				\node (12) at (1.332,0.649) {13};
				\node (13) at (0.623,1.318) {14};
				\node (14) at (-0.104,2.252) {15};
				\node (16) at (-2.307,-0.255) {17};
				\node (17) at (-0.623,-1.320) {18};
				\node (18) at (0.587,-0.761) {19};
				\node (19) at (0.456,0.014) {20};
				\node (20) at (0.218,0.637) {21};
				\node (21) at (-0.471,1.373) {22};
				\node (22) at (-1.846,0.597) {23};
				\node (23) at (-1.337,0.093) {24};
				\node (24) at (-1.335,-0.650) {25};
				\node (25) at (-0.219,-0.638) {26};
				\node (26) at (-0.455,-0.015) {27};
				\node (27) at (-0.588,0.760) {28};
				
				\node (2) at ($.5*(0)+.5*(1)$) {3};
				\node (15) at ($.5*(5)+.5*(6)$) {16};
				
				\node (4) at ($.4*(0)+.6*(5)$) {5};
				\node (7) at ($.4*(6)+.6*(1)$) {8};
				
				\draw[colour0] (0) edge[bend left, colour0] (1);
				\draw[colour0] (2) -- (9);
				\draw[colour0] (3) -- (10);
				\draw[colour0] (4) -- (5);
				\draw[colour0] (6) -- (15);
				\draw[colour0] (7) -- (8);
				\draw[colour0] (11) -- (19);
				\draw[colour0] (12) -- (13);
				\draw[colour0] (14) -- (21);
				\draw[colour0] (16) -- (23);
				\draw[colour0] (17) -- (24);
				\draw[colour0] (18) -- (25);
				\draw[colour0] (20) -- (26);
				\draw[colour0] (22) -- (27);
				
				\draw[colour1] (0) -- (2);
				\draw[colour1] (1) edge[bend left, colour1] (6);
				\draw[colour1] (3) -- (11);
				\draw[colour1] (4) -- (13);
				\draw[colour1] (5) -- (14);
				\draw[colour1] (7) -- (16);
				\draw[colour1] (8) -- (9);
				\draw[colour1] (10) -- (18);
				\draw[colour1] (12) -- (19);
				\draw[colour1] (15) -- (21);
				\draw[colour1] (17) -- (25);
				\draw[colour1] (20) -- (27);
				\draw[colour1] (22) -- (23);
				\draw[colour1] (24) -- (26);
				
				\draw[colour2] (0) -- (3);
				\draw[colour2b] (0) -- (3);
				\draw[colour2] (1) -- (7);
				\draw[colour2b] (1) -- (7);
				\draw[colour2] (2) -- (8);
				\draw[colour2b] (2) -- (8);
				\draw[colour2] (4) -- (14);
				\draw[colour2b] (4) -- (14);
				\draw[colour2] (5) edge[bend right, colour2] (6);
				\draw[colour2b] (5) edge[bend right, colour2b] (6);
				\draw[colour2] (9) -- (17);
				\draw[colour2b] (9) -- (17);
				\draw[colour2] (10) -- (11);
				\draw[colour2b] (10) -- (11);
				\draw[colour2] (12) -- (20);
				\draw[colour2b] (12) -- (20);
				\draw[colour2] (13) -- (21);
				\draw[colour2b] (13) -- (21);
				\draw[colour2] (15) -- (22);
				\draw[colour2b] (15) -- (22);
				\draw[colour2] (16) -- (24);
				\draw[colour2b] (16) -- (24);
				\draw[colour2] (18) -- (19);
				\draw[colour2b] (18) -- (19);
				\draw[colour2] (23) -- (27);
				\draw[colour2b] (23) -- (27);
				\draw[colour2] (25) -- (26);
				\draw[colour2b] (25) -- (26);
				
				\draw[colour3] (0) -- (4);
				\draw[colour3] (1) -- (2);
				\draw[colour3] (3) -- (12);
				\draw[colour3] (5) -- (15);
				\draw[colour3] (6) -- (7);
				\draw[colour3] (8) -- (17);
				\draw[colour3] (9) -- (10);
				\draw[colour3] (11) -- (18);
				\draw[colour3] (13) -- (14);
				\draw[colour3] (16) -- (22);
				\draw[colour3] (19) -- (20);
				\draw[colour3] (21) -- (27);
				\draw[colour3] (23) -- (26);
				\draw[colour3] (24) -- (25);
				
				\draw[colour4] (0) edge[bend right, colour4] (5);
				\draw[colour4b] (0) edge[bend right, colour4b] (5);
				\draw[colour4] (1) -- (8);
				\draw[colour4b] (1) -- (8);
				\draw[colour4] (2) -- (10);
				\draw[colour4b] (2) -- (10);
				\draw[colour4] (3) -- (4);
				\draw[colour4b] (3) -- (4);
				\draw[colour4] (6) -- (16);
				\draw[colour4b] (6) -- (16);
				\draw[colour4] (7) -- (17);
				\draw[colour4b] (7) -- (17);
				\draw[colour4] (9) -- (18);
				\draw[colour4b] (9) -- (18);
				\draw[colour4] (11) -- (12);
				\draw[colour4b] (11) -- (12);
				\draw[colour4] (13) -- (20);
				\draw[colour4b] (13) -- (20);
				\draw[colour4] (14) -- (15);
				\draw[colour4b] (14) -- (15);
				\draw[colour4] (19) -- (25);
				\draw[colour4b] (19) -- (25);
				\draw[colour4] (21) -- (22);
				\draw[colour4b] (21) -- (22);
				\draw[colour4] (23) -- (24);
				\draw[colour4b] (23) -- (24);
				\draw[colour4] (26) -- (27);
				\draw[colour4b] (26) -- (27);
			\end{tikzpicture}
			\begin{tikzpicture}[very thick,scale=.75,no text] 
				
				\node (0) at (18:3.5) {1};
				\node (1) at (90:3.5) {2};
				\node (8) at (162:3.5) {9};
				\node (9) at (234:3.5) {10};
				\node (2) at (306:3.5) {3};
				
				\node (4) at (18:2.85) {5};
				\node (5) at (54:2.85) {6};
				\node (6) at (90:2.85) {7};
				\node (7) at (126:2.85) {8};
				\node (17) at (162:2.85) {18};
				\node (18) at (198:2.85) {19};
				\node (19) at (234:2.85) {20};
				\node (10) at (270:2.85) {11};
				\node (11) at (306:2.85) {12};
				\node (3) at (342:2.85) {4};
				
				\node (13) at (18:1.8) {14};
				\node (14) at (54:1.8) {15};
				\node (15) at (90:1.8) {16};
				\node (16) at (126:1.8) {17};
				\node (25) at (162:1.8) {26};
				\node (26) at (198:1.8) {27};
				\node (27) at (234:1.8) {28};
				\node (20) at (270:1.8) {21};
				\node (21) at (306:1.8) {22};
				\node (12) at (342:1.8) {13};
				
				\node (23) at (54:0.8) {24};
				\node (24) at (126:0.8) {25};
				\node (29) at (198:0.8) {30};
				\node (28) at (270:0.8) {29};
				\node (22) at (342:0.8) {23};
				
				\draw[colour0] (0) edge[bend right, colour0] (1);
				\draw[colour0] (2) -- (3);
				\draw[colour0] (4) -- (13);
				\draw[colour0] (5) -- (14);
				\draw[colour0] (6) -- (7);
				\draw[colour0] (8) -- (17);
				\draw[colour0] (9) -- (18);
				\draw[colour0] (10) -- (11);
				\draw[colour0] (12) -- (21);
				\draw[colour0] (15) -- (24);
				\draw[colour0] (16) -- (25);
				\draw[colour0] (19) -- (26);
				\draw[colour0] (20) -- (28);
				\draw[colour0] (22) -- (23);
				\draw[colour0] (27) -- (29);
				
				\draw[colour1] (0) edge[bend left, colour1] (2);
				\draw[colour1] (1) -- (6);
				\draw[colour1] (3) -- (12);
				\draw[colour1] (4) -- (5);
				\draw[colour1] (7) -- (8);
				\draw[colour1] (9) -- (10);
				\draw[colour1] (11) -- (21);
				\draw[colour1] (13) -- (23);
				\draw[colour1] (14) -- (15);
				\draw[colour1] (16) -- (17);
				\draw[colour1] (18) -- (19);
				\draw[colour1] (20) -- (27);
				\draw[colour1] (22) -- (28);
				\draw[colour1] (24) -- (25);
				\draw[colour1] (26) -- (29);
				
				\draw[colour2] (0) -- (3);
				\draw[colour2b] (0) -- (3);
				\draw[colour2] (1) -- (7);
				\draw[colour2b] (1) -- (7);
				\draw[colour2] (2) -- (11);
				\draw[colour2b] (2) -- (11);
				\draw[colour2] (4) -- (12);
				\draw[colour2b] (4) -- (12);
				\draw[colour2] (5) -- (6);
				\draw[colour2b] (5) -- (6);
				\draw[colour2] (8) edge[bend right, colour2] (9);
				\draw[colour2b] (8) edge[bend right, colour2b] (9);
				\draw[colour2] (10) -- (19);
				\draw[colour2b] (10) -- (19);
				\draw[colour2] (13) -- (22);
				\draw[colour2b] (13) -- (22);
				\draw[colour2] (14) -- (23);
				\draw[colour2b] (14) -- (23);
				\draw[colour2] (15) -- (16);
				\draw[colour2b] (15) -- (16);
				\draw[colour2] (17) -- (25);
				\draw[colour2b] (17) -- (25);
				\draw[colour2] (18) -- (26);
				\draw[colour2b] (18) -- (26);
				\draw[colour2] (20) -- (21);
				\draw[colour2b] (20) -- (21);
				\draw[colour2] (24) -- (29);
				\draw[colour2b] (24) -- (29);
				\draw[colour2] (27) -- (28);
				\draw[colour2b] (27) -- (28);
				
				\draw[colour3] (0) -- (4);
				\draw[colour3] (1) -- (5);
				\draw (2) edge[bend left,colour3] (9);
				\draw[colour3] (3) -- (11);
				\draw[colour3] (6) -- (14);
				\draw[colour3] (7) -- (17);
				\draw[colour3] (8) -- (18);
				\draw[colour3] (10) -- (20);
				\draw[colour3] (12) -- (13);
				\draw[colour3] (15) -- (23);
				\draw[colour3] (16) -- (24);
				\draw[colour3] (19) -- (27);
				\draw[colour3] (21) -- (22);
				\draw[colour3] (25) -- (26);
				\draw[colour3] (28) -- (29);
				
				\draw[colour4] (0) -- (5);
				\draw[colour4b] (0) -- (5);
				\draw[colour4] (1) edge[bend right, colour4] (8);
				\draw[colour4b] (1) edge[bend right, colour4b] (8);
				\draw[colour4] (2) -- (10);
				\draw[colour4b] (2) -- (10);
				\draw[colour4] (3) -- (4);
				\draw[colour4b] (3) -- (4);
				\draw[colour4] (6) -- (15);
				\draw[colour4b] (6) -- (15);
				\draw[colour4] (7) -- (16);
				\draw[colour4b] (7) -- (16);
				\draw[colour4] (9) -- (19);
				\draw[colour4b] (9) -- (19);
				\draw[colour4] (11) -- (20);
				\draw[colour4b] (11) -- (20);
				\draw[colour4] (12) -- (22);
				\draw[colour4b] (12) -- (22);
				\draw[colour4] (13) -- (14);
				\draw[colour4b] (13) -- (14);
				\draw[colour4] (17) -- (18);
				\draw[colour4b] (17) -- (18);
				\draw[colour4] (21) -- (28);
				\draw[colour4b] (21) -- (28);
				\draw[colour4] (23) -- (24);
				\draw[colour4b] (23) -- (24);
				\draw[colour4] (25) -- (29);
				\draw[colour4b] (25) -- (29);
				\draw[colour4] (26) -- (27);
				\draw[colour4b] (26) -- (27);
				
			\end{tikzpicture}
			
			\begin{tikzpicture}[very thick,scale=.75,no text] 
				\node (0) at (3.6,0) {1};
				\node (1) at (0,3.6) {2};
				\node (2) at (0,-3.6) {3};
				\node (4) at (2.5,0) {5};
				\node (6) at (0,2.650) {7};
				\node (8) at (-3.6,0) {9};
				\node (10) at (0,-2.652) {11};
				\node (12) at (1.8,-0.6) {13};
				\node (13) at (1.8,0.6) {14};
				\node (16) at (-.9,1.4) {17};
				\node (17) at (-2.7,0) {18};
				\node (18) at (-.9,-1.4) {19};
				\node (20) at (0.9,-1) {21};
				\node (21) at (1.1,0) {22};
				\node (22) at (0.9,1) {23};
				\node (29) at (0.4,0) {30};
				\node (30) at (-0.3,0) {31};
				
				\node (3) at ($.5*(0)+.5*(2)$) {4};
				\node (5) at ($.5*(0)+.5*(1)$) {6};
				\node (7) at ($.5*(1)+.5*(8)$) {8};
				\node (9) at ($.5*(2)+.5*(8)$) {10};
				
				\node (23) at ($.25*(17)+.75*(22)$) {24};
				\node (24) at ($.5*(17)+.5*(22)$) {25};
				\node (25) at ($.75*(17)+.25*(22)$) {26};
				\node (26) at ($.75*(17)+.25*(20)$) {27};
				\node (27) at ($.5*(17)+.5*(20)$) {28};
				\node (28) at ($.25*(17)+.75*(20)$) {29};
				
				\node (14) at ($.3333*(16)+.6666*(5)$) {15};
				\node (15) at ($.6666*(16)+.3333*(5)$) {16};
				\node (11) at ($.3333*(18)+.6666*(3)$) {12};
				\node (19) at ($.6666*(18)+.3333*(3)$) {20};
				
				\node (31) at ($.2*(24)+.2*(25)+.2*(26)+.2*(27)+.2*(30)$) {32};
				
				\draw[colour0] (0) edge[bend right, colour0] (1);
				\draw[colour0] (2) edge[bend left, colour0] (8);
				\draw[colour0] (3) -- (4);
				\draw[colour0] (5) -- (13);
				\draw[colour0] (6) -- (7);
				\draw[colour0] (9) -- (10);
				\draw[colour0] (11) -- (19);
				\draw[colour0] (12) -- (21);
				\draw[colour0] (14) -- (15);
				\draw[colour0] (16) -- (24);
				\draw[colour0] (17) -- (25);
				\draw[colour0] (18) -- (27);
				\draw[colour0] (20) -- (29);
				\draw[colour0] (22) -- (23);
				\draw[colour0] (26) -- (31);
				\draw[colour0] (28) -- (30);
				
				\draw[colour1] (0) edge[bend left, colour1] (2);
				\draw[colour1] (1) -- (5);
				\draw[colour1] (3) -- (11);
				\draw[colour1] (4) -- (12);
				\draw[colour1] (6) -- (16);
				\draw[colour1] (7) -- (17);
				\draw[colour1] (8) -- (9);
				\draw[colour1] (10) -- (19);
				\draw[colour1] (13) -- (14);
				\draw[colour1] (15) -- (22);
				\draw[colour1] (18) -- (26);
				\draw[colour1] (20) -- (28);
				\draw[colour1] (21) -- (29);
				\draw[colour1] (23) -- (24);
				\draw[colour1] (25) -- (31);
				\draw[colour1] (27) -- (30);
				
				\draw[colour2] (0) -- (3);
				\draw[colour2b] (0) -- (3);
				\draw[colour2] (1) -- (6);
				\draw[colour2b] (1) -- (6);
				\draw[colour2] (2) -- (9);
				\draw[colour2b] (2) -- (9);
				\draw[colour2] (4) -- (5);
				\draw[colour2b] (4) -- (5);
				\draw[colour2] (7) -- (8);
				\draw[colour2b] (7) -- (8);
				\draw[colour2] (10) -- (18);
				\draw[colour2b] (10) -- (18);
				\draw[colour2] (11) -- (12);
				\draw[colour2b] (11) -- (12);
				\draw[colour2] (13) -- (21);
				\draw[colour2b] (13) -- (21);
				\draw[colour2] (14) -- (22);
				\draw[colour2b] (14) -- (22);
				\draw[colour2] (15) -- (16);
				\draw[colour2b] (15) -- (16);
				\draw[colour2] (17) -- (26);
				\draw[colour2b] (17) -- (26);
				\draw[colour2] (19) -- (20);
				\draw[colour2b] (19) -- (20);
				\draw[colour2] (23) -- (29);
				\draw[colour2b] (23) -- (29);
				\draw[colour2] (24) -- (25);
				\draw[colour2b] (24) -- (25);
				\draw[colour2] (27) -- (28);
				\draw[colour2b] (27) -- (28);
				\draw[colour2] (30) -- (31);
				\draw[colour2b] (30) -- (31);
				
				\draw[colour3] (0) -- (4);
				\draw[colour3] (1) -- (7);
				\draw[colour3] (2) -- (3);
				\draw[colour3] (5) -- (14);
				\draw[colour3] (6) -- (15);
				\draw[colour3] (8) -- (17);
				\draw[colour3] (9) -- (18);
				\draw[colour3] (10) -- (11);
				\draw[colour3] (12) -- (13);
				\draw[colour3] (16) -- (25);
				\draw[colour3] (19) -- (28);
				\draw[colour3] (20) -- (21);
				\draw[colour3] (22) -- (29);
				\draw[colour3] (23) -- (30);
				\draw[colour3] (24) -- (31);
				\draw[colour3] (26) -- (27);
				
				\draw[colour4] (0) -- (5);
				\draw[colour4b] (0) -- (5);
				\draw[colour4] (1) edge[bend right, colour4] (8);
				\draw[colour4b] (1) edge[bend right, colour4b] (8);
				\draw[colour4] (2) -- (10);
				\draw[colour4b] (2) -- (10);
				\draw[colour4] (3) -- (12);
				\draw[colour4b] (3) -- (12);
				\draw[colour4] (4) -- (13);
				\draw[colour4b] (4) -- (13);
				\draw[colour4] (6) -- (14);
				\draw[colour4b] (6) -- (14);
				\draw[colour4] (7) -- (16);
				\draw[colour4b] (7) -- (16);
				\draw[colour4] (9) -- (17);
				\draw[colour4b] (9) -- (17);
				\draw[colour4] (11) -- (20);
				\draw[colour4b] (11) -- (20);
				\draw[colour4] (15) -- (23);
				\draw[colour4b] (15) -- (23);
				\draw[colour4] (18) -- (19);
				\draw[colour4b] (18) -- (19);
				\draw[colour4] (21) -- (22);
				\draw[colour4b] (21) -- (22);
				\draw[colour4] (24) -- (30);
				\draw[colour4b] (24) -- (30);
				\draw[colour4] (25) -- (26);
				\draw[colour4b] (25) -- (26);
				\draw[colour4] (27) -- (31);
				\draw[colour4b] (27) -- (31);
				\draw[colour4] (28) -- (29);
				\draw[colour4b] (28) -- (29);
			\end{tikzpicture}
			\begin{tikzpicture}[very thick,scale=.75,no text] 
				
				\node (0) at (2.479,0.029) {1};
				\node (1) at (2.610,-1.082) {2};
				\node (2) at (1.534,-0.639) {3};
				\node (3) at (1.346,0.367) {4};
				\node (4) at (2.520,1.132) {5};
				\node (5) at (3.647,-0.000) {6};
				\node (6) at (2.579,-2.579) {7};
				\node (7) at (1.731,-1.777) {8};
				\node (8) at (0.689,-1.214) {9};
				\node (9) at (0.480,-0.201) {10};
				\node (10) at (1.500,1.500) {11};
				\node (11) at (2.579,2.578) {12};
				\node (12) at (-0.000,-3.647) {13};
				\node (13) at (0.981,-2.584) {14};
				\node (14) at (-0.001,-2.124) {15};
				\node (15) at (-0.447,-1.071) {16};
				\node (16) at (-0.479,0.201) {17};
				\node (17) at (0.445,1.070) {18};
				\node (18) at (1.044,2.516) {19};
				\node (19) at (-0.000,3.646) {20};
				\node (20) at (-2.579,-2.579) {21};
				\node (21) at (-1.042,-2.516) {22};
				\node (22) at (-1.501,-1.501) {23};
				\node (23) at (-1.346,-0.366) {24};
				\node (24) at (-0.690,1.216) {25};
				\node (25) at (0.002,2.124) {26};
				\node (26) at (-0.981,2.583) {27};
				\node (27) at (-2.579,2.578) {28};
				\node (28) at (-3.647,-0.000) {29};
				\node (29) at (-2.521,-1.130) {30};
				\node (30) at (-2.476,-0.027) {31};
				\node (31) at (-1.533,0.639) {32};
				\node (32) at (-1.732,1.776) {33};
				\node (33) at (-2.611,1.082) {34};
				
				\draw[colour0] (0) -- (1);
				\draw[colour0] (2) -- (7);
				\draw[colour0] (3) -- (4);
				\draw[colour0] (5) -- (11);
				\draw[colour0] (6) -- (13);
				\draw[colour0] (8) -- (14);
				\draw[colour0] (9) -- (15);
				\draw[colour0] (10) -- (18);
				\draw[colour0] (12) -- (21);
				\draw[colour0] (16) -- (24);
				\draw[colour0] (17) -- (25);
				\draw[colour0] (19) -- (26);
				\draw[colour0] (20) -- (28);
				\draw[colour0] (22) -- (23);
				\draw[colour0] (27) -- (33);
				\draw[colour0] (29) -- (30);
				\draw[colour0] (31) -- (32);
				
				\draw[colour1] (0) -- (2);
				\draw[colour1] (1) -- (5);
				\draw[colour1] (3) -- (10);
				\draw[colour1] (4) -- (11);
				\draw[colour1] (6) -- (7);
				\draw[colour1] (8) -- (9);
				\draw[colour1] (12) -- (20);
				\draw[colour1] (13) -- (14);
				\draw[colour1] (15) -- (21);
				\draw[colour1] (16) -- (17);
				\draw[colour1] (18) -- (25);
				\draw[colour1] (19) -- (27);
				\draw[colour1] (22) -- (29);
				\draw[colour1] (23) -- (31);
				\draw[colour1] (24) -- (26);
				\draw[colour1] (28) -- (30);
				\draw[colour1] (32) -- (33);
				
				\draw[colour2] (0) -- (3);
				\draw[colour2b] (0) -- (3);
				\draw[colour2] (1) -- (6);
				\draw[colour2b] (1) -- (6);
				\draw[colour2] (2) -- (8);
				\draw[colour2b] (2) -- (8);
				\draw[colour2] (4) -- (5);
				\draw[colour2b] (4) -- (5);
				\draw[colour2] (7) -- (13);
				\draw[colour2b] (7) -- (13);
				\draw[colour2] (9) -- (17);
				\draw[colour2b] (9) -- (17);
				\draw[colour2] (10) -- (11);
				\draw[colour2b] (10) -- (11);
				\draw[colour2] (12) -- (14);
				\draw[colour2b] (12) -- (14);
				\draw[colour2] (15) -- (22);
				\draw[colour2b] (15) -- (22);
				\draw[colour2] (16) -- (23);
				\draw[colour2b] (16) -- (23);
				\draw[colour2] (18) -- (19);
				\draw[colour2b] (18) -- (19);
				\draw[colour2] (20) -- (21);
				\draw[colour2b] (20) -- (21);
				\draw[colour2] (24) -- (31);
				\draw[colour2b] (24) -- (31);
				\draw[colour2] (25) -- (26);
				\draw[colour2b] (25) -- (26);
				\draw[colour2] (27) -- (32);
				\draw[colour2b] (27) -- (32);
				\draw[colour2] (28) -- (29);
				\draw[colour2b] (28) -- (29);
				\draw[colour2] (30) -- (33);
				\draw[colour2b] (30) -- (33);
				
				\draw[colour3] (0) -- (4);
				\draw[colour3] (1) -- (2);
				\draw[colour3] (3) -- (9);
				\draw[colour3] (5) -- (6);
				\draw[colour3] (7) -- (8);
				\draw[colour3] (10) -- (17);
				\draw[colour3] (11) -- (18);
				\draw[colour3] (12) -- (13);
				\draw[colour3] (14) -- (21);
				\draw[colour3] (15) -- (16);
				\draw[colour3] (19) -- (25);
				\draw[colour3] (20) -- (22);
				\draw[colour3] (23) -- (29);
				\draw[colour3] (24) -- (32);
				\draw[colour3] (26) -- (27);
				\draw[colour3] (28) -- (33);
				\draw[colour3] (30) -- (31);
				
				\draw[colour4] (0) -- (5);
				\draw[colour4b] (0) -- (5);
				\draw[colour4] (1) -- (7);
				\draw[colour4b] (1) -- (7);
				\draw[colour4] (2) -- (3);
				\draw[colour4b] (2) -- (3);
				\draw[colour4] (4) -- (10);
				\draw[colour4b] (4) -- (10);
				\draw[colour4] (6) -- (12);
				\draw[colour4b] (6) -- (12);
				\draw[colour4] (8) -- (13);
				\draw[colour4b] (8) -- (13);
				\draw[colour4] (9) -- (16);
				\draw[colour4b] (9) -- (16);
				\draw[colour4] (11) -- (19);
				\draw[colour4b] (11) -- (19);
				\draw[colour4] (14) -- (15);
				\draw[colour4b] (14) -- (15);
				\draw[colour4] (17) -- (18);
				\draw[colour4b] (17) -- (18);
				\draw[colour4] (20) -- (29);
				\draw[colour4b] (20) -- (29);
				\draw[colour4] (21) -- (22);
				\draw[colour4b] (21) -- (22);
				\draw[colour4] (23) -- (30);
				\draw[colour4b] (23) -- (30);
				\draw[colour4] (24) -- (25);
				\draw[colour4b] (24) -- (25);
				\draw[colour4] (26) -- (32);
				\draw[colour4b] (26) -- (32);
				\draw[colour4] (27) -- (28);
				\draw[colour4b] (27) -- (28);
				\draw[colour4] (31) -- (33);
				\draw[colour4b] (31) -- (33);
				
			\end{tikzpicture}
			
			\begin{tikzpicture}[very thick,rotate=90,scale=.9,no text] 
				
				\node (0) at (0:4.2) {1};
				\node (1) at (0:3.1) {2};
				\node (7) at (0:2) {8};
				\node (25) at (0:1) {26};
				\node (34) at (180:.4) {35};
				\node (35) at (180:1) {36};
				\node (29) at (180:1.6) {30};
				
				\node (3) at (120:4.2) {4};
				\node (10) at (120:3.1) {11};
				\node (19) at (120:2) {20};
				\node (27) at (120:1) {28};
				\node (32) at (300:.4) {33};
				\node (24) at (300:1) {25};
				\node (14) at (300:1.6) {15};
				
				\node (4) at (240:4.2) {5};
				\node (12) at (240:3.1) {13};
				\node (22) at (240:2) {23};
				\node (31) at (240:1) {32};
				\node (33) at (60:.4) {34};
				\node (26) at (60:1) {27};
				\node (17) at (60:1.6) {18};
				
				\node (15) at (1.171,-0.529) {16};
				\node (18) at (-0.136,1.257) {19};
				\node (30) at (-1.024,-0.754) {31};
				
				\node (2) at ($0.5*(0)+0.5*(3)$) {3};
				\node (5) at ($0.5*(0)+0.5*(4)$) {6};
				\node (11) at ($0.5*(3)+0.5*(4)$) {12};
				
				\node (8) at ($0.5*(1)+0.5*(2)$) {9};
				\node (20) at ($0.5*(10)+0.5*(11)$) {21};
				\node (13) at ($0.5*(5)+0.5*(12)$) {14};
				
				\node (9) at ($.5*(10)+.5*(17)$) {10};
				\node (6) at ($.5*(1)+.5*(14)$) {7};
				\node (21) at ($.5*(12)+.5*(29)$) {22};
				
				\node (28) at ($.2*(19)+.2*(20)+.2*(29)+.2*(35)+.2*(27)$) {29};
				\node (16) at ($.2*(7)+.2*(8)+.2*(17)+.2*(26)+.2*(25)$) {17};
				\node (23) at ($.2*(13)+.2*(14)+.2*(24)+.2*(31)+.2*(22)$) {24};
				
				\draw[colour0] (0) -- (1);
				\draw[colour0] (2) -- (9);
				\draw[colour0] (3) -- (10);
				\draw[colour0] (4) -- (12);
				\draw[colour0] (5) -- (6);
				\draw[colour0] (7) -- (8);
				\draw[colour0] (11) -- (20);
				\draw[colour0] (13) -- (23);
				\draw[colour0] (14) -- (15);
				\draw[colour0] (16) -- (26);
				\draw[colour0] (17) -- (18);
				\draw[colour0] (19) -- (28);
				\draw[colour0] (21) -- (22);
				\draw[colour0] (24) -- (31);
				\draw[colour0] (25) -- (32);
				\draw[colour0] (27) -- (33);
				\draw[colour0] (29) -- (30);
				\draw[colour0] (34) -- (35);
				
				\draw[colour1] (0) -- (2);
				\draw[colour1] (1) -- (5);
				\draw[colour1] (3) edge[bend right, colour1] (4);
				\draw[colour1] (6) -- (7);
				\draw[colour1] (8) -- (16);
				\draw[colour1] (9) -- (17);
				\draw[colour1] (10) -- (20);
				\draw[colour1] (11) -- (21);
				\draw[colour1] (12) -- (13);
				\draw[colour1] (14) -- (24);
				\draw[colour1] (15) -- (25);
				\draw[colour1] (18) -- (19);
				\draw[colour1] (22) -- (30);
				\draw[colour1] (23) -- (31);
				\draw[colour1] (26) -- (33);
				\draw[colour1] (27) -- (35);
				\draw[colour1] (28) -- (29);
				\draw[colour1] (32) -- (34);
				
				\draw[colour2] (0) edge[bend right, colour2] (3);
				\draw[colour2b] (0) edge[bend right, colour2b] (3);
				\draw[colour2] (1) -- (6);
				\draw[colour2b] (1) -- (6);
				\draw[colour2] (2) -- (8);
				\draw[colour2b] (2) -- (8);
				\draw[colour2] (4) -- (5);
				\draw[colour2b] (4) -- (5);
				\draw[colour2] (7) -- (15);
				\draw[colour2b] (7) -- (15);
				\draw[colour2] (9) -- (10);
				\draw[colour2b] (9) -- (10);
				\draw[colour2] (11) -- (12);
				\draw[colour2b] (11) -- (12);
				\draw[colour2] (13) -- (22);
				\draw[colour2b] (13) -- (22);
				\draw[colour2] (14) -- (23);
				\draw[colour2b] (14) -- (23);
				\draw[colour2] (16) -- (17);
				\draw[colour2b] (16) -- (17);
				\draw[colour2] (18) -- (27);
				\draw[colour2b] (18) -- (27);
				\draw[colour2] (19) -- (20);
				\draw[colour2b] (19) -- (20);
				\draw[colour2] (21) -- (29);
				\draw[colour2b] (21) -- (29);
				\draw[colour2] (24) -- (32);
				\draw[colour2b] (24) -- (32);
				\draw[colour2] (25) -- (26);
				\draw[colour2b] (25) -- (26);
				\draw[colour2] (28) -- (35);
				\draw[colour2b] (28) -- (35);
				\draw[colour2] (30) -- (31);
				\draw[colour2b] (30) -- (31);
				\draw[colour2] (33) -- (34);
				\draw[colour2b] (33) -- (34);
				
				\path (0) edge[bend left, colour3] (4);
				\draw[colour3] (1) -- (7);
				\draw[colour3] (2) -- (10);
				\draw[colour3] (3) -- (11);
				\draw[colour3] (5) -- (13);
				\draw[colour3] (6) -- (14);
				\draw[colour3] (8) -- (17);
				\draw[colour3] (9) -- (19);
				\draw[colour3] (12) -- (21);
				\draw[colour3] (15) -- (24);
				\draw[colour3] (16) -- (25);
				\draw[colour3] (18) -- (26);
				\draw[colour3] (20) -- (29);
				\draw[colour3] (22) -- (23);
				\draw[colour3] (27) -- (28);
				\draw[colour3] (30) -- (35);
				\draw[colour3] (31) -- (34);
				\draw[colour3] (32) -- (33);
				
				\draw[colour4] (0) -- (5);
				\draw[colour4b] (0) -- (5);
				\draw[colour4] (1) -- (8);
				\draw[colour4b] (1) -- (8);
				\draw[colour4] (2) -- (3);
				\draw[colour4b] (2) -- (3);
				\draw[colour4] (4) -- (11);
				\draw[colour4b] (4) -- (11);
				\draw[colour4] (6) -- (15);
				\draw[colour4b] (6) -- (15);
				\draw[colour4] (7) -- (16);
				\draw[colour4b] (7) -- (16);
				\draw[colour4] (9) -- (18);
				\draw[colour4b] (9) -- (18);
				\draw[colour4] (10) -- (19);
				\draw[colour4b] (10) -- (19);
				\draw[colour4] (12) -- (22);
				\draw[colour4b] (12) -- (22);
				\draw[colour4] (13) -- (14);
				\draw[colour4b] (13) -- (14);
				\draw[colour4] (17) -- (26);
				\draw[colour4b] (17) -- (26);
				\draw[colour4] (20) -- (28);
				\draw[colour4b] (20) -- (28);
				\draw[colour4] (21) -- (30);
				\draw[colour4b] (21) -- (30);
				\draw[colour4] (23) -- (24);
				\draw[colour4b] (23) -- (24);
				\draw[colour4] (25) -- (33);
				\draw[colour4b] (25) -- (33);
				\draw[colour4] (27) -- (34);
				\draw[colour4b] (27) -- (34);
				\draw[colour4] (29) -- (35);
				\draw[colour4b] (29) -- (35);
				\draw[colour4] (31) -- (32);
				\draw[colour4b] (31) -- (32);
				
			\end{tikzpicture}\vspace{-20pt}
		\end{center}
		\caption{Planar quintic perfectly hamiltonian graphs of order 22, 24, 26, 28, 30, 32, 34, and 36.}\label{fig:orders}
	\end{figure}
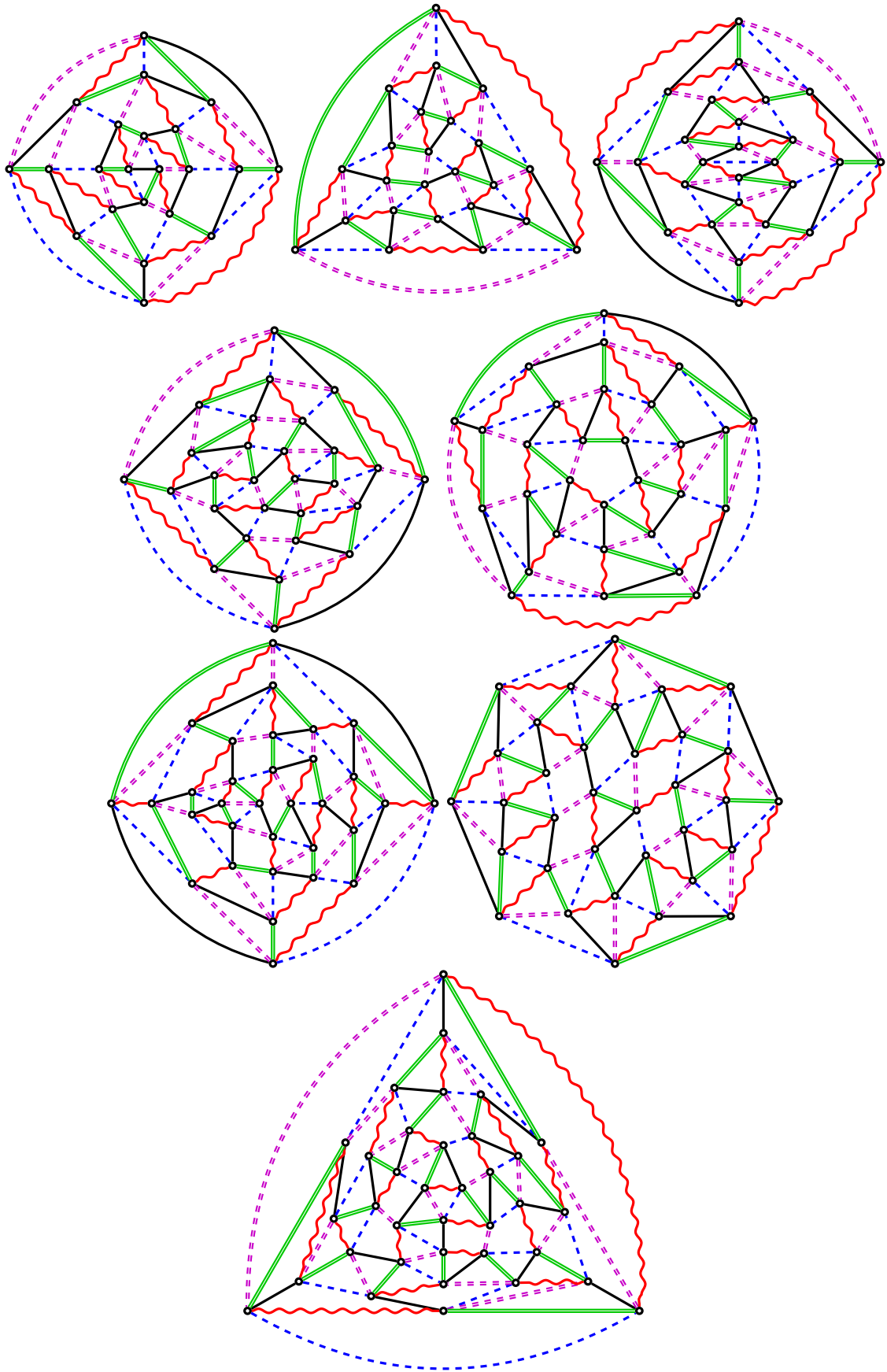

	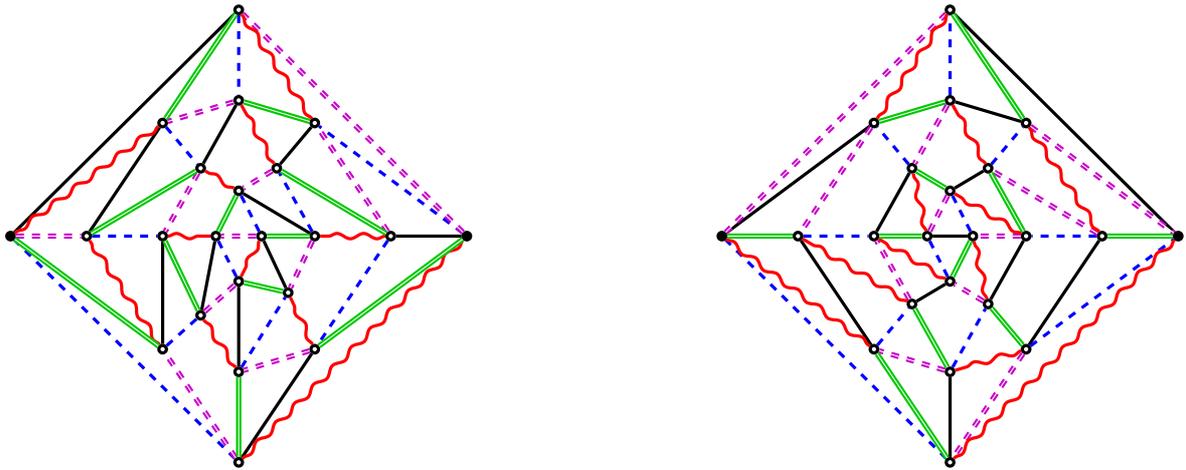
\begin{figure}
		\begin{tikzpicture}[very thick]
			
			\node (1) at (2,0) {};
			\node (3) at (0,3) {};
			\node (4) at (0,-3) {};
			\node (6) at (1,0) {};
			\node (8) at (0,1.8) {};
			\node (12) at (0,-1.8) {};
			\node (14) at (.3,0) {};
			\node (15) at (0,.6) {};
			\node (17) at (-2,0) {};
			\node (18) at (-1,0) {};
			\node (20) at (0,-.6) {};
			\node (21) at (-.3,0) {};
			
			\node (2) at ($.5*(3) + .5*(1)$) {};
			\node (9) at ($.5*(3) + .5*(17)$) {};
			\node (5) at ($.5*(4) + .5*(1)$) {};
			\node (11) at ($.5*(4) + .5*(17)$) {};
			
			\node (19) at ($.5*(11) + .5*(20)$) {};
			\node (13) at ($.5*(14) + .5*(5)$) {};
			\node (16) at ($.5*(18) + .5*(8)$) {};
			\node (7) at ($.5*(6) + .5*(8)$) {};
			
			\node[fill=black] (0) at (3,0) {};
			\node[fill=black] (10) at (-3,0) {};
			
			\draw[colour0] (0) -- (1);
			\draw[colour0] (2) -- (7);
			\draw[colour0] (3) -- (10);
			\draw[colour0] (4) -- (5);
			\draw[colour0] (6) -- (15);
			\draw[colour0] (8) -- (16);
			\draw[colour0] (9) -- (17);
			\draw[colour0] (11) -- (18);
			\draw[colour0] (12) -- (20);
			\draw[colour0] (13) -- (14);
			\draw[colour0] (19) -- (21);
			
			\draw[colour1] (0) -- (2);
			\draw[colour1] (1) -- (5);
			\draw[colour1] (3) -- (8);
			\draw[colour1] (4) -- (10);
			\draw[colour1] (6) -- (7);
			\draw[colour1] (9) -- (16);
			\draw[colour1] (11) -- (19);
			\draw[colour1] (12) -- (13);
			\draw[colour1] (14) -- (15);
			\draw[colour1] (17) -- (18);
			\draw[colour1] (20) -- (21);
			
			\draw[colour2] (0) -- (3);
			\draw[colour2b] (0) -- (3);
			\draw[colour2] (1) -- (2);
			\draw[colour2b] (1) -- (2);
			\draw[colour2] (4) -- (11);
			\draw[colour2b] (4) -- (11);
			\draw[colour2] (5) -- (12);
			\draw[colour2b] (5) -- (12);
			\draw[colour2] (6) -- (13);
			\draw[colour2b] (6) -- (13);
			\draw[colour2] (7) -- (15);
			\draw[colour2b] (7) -- (15);
			\draw[colour2] (8) -- (9);
			\draw[colour2b] (8) -- (9);
			\draw[colour2] (10) -- (17);
			\draw[colour2b] (10) -- (17);
			\draw[colour2] (14) -- (21);
			\draw[colour2b] (14) -- (21);
			\draw[colour2] (16) -- (18);
			\draw[colour2b] (16) -- (18);
			\draw[colour2] (19) -- (20);
			\draw[colour2b] (19) -- (20);
			
			\draw[colour3] (0) -- (4);
			\draw[colour3] (1) -- (6);
			\draw[colour3] (2) -- (3);
			\draw[colour3] (5) -- (13);
			\draw[colour3] (7) -- (8);
			\draw[colour3] (9) -- (10);
			\draw[colour3] (11) -- (17);
			\draw[colour3] (12) -- (19);
			\draw[colour3] (14) -- (20);
			\draw[colour3] (15) -- (16);
			\draw[colour3] (18) -- (21);
			
			\draw[colour4] (0) -- (5);
			\draw[colour4b] (0) -- (5);
			\draw[colour4] (1) -- (7);
			\draw[colour4b] (1) -- (7);
			\draw[colour4] (2) -- (8);
			\draw[colour4b] (2) -- (8);
			\draw[colour4] (3) -- (9);
			\draw[colour4b] (3) -- (9);
			\draw[colour4] (4) -- (12);
			\draw[colour4b] (4) -- (12);
			\draw[colour4] (6) -- (14);
			\draw[colour4b] (6) -- (14);
			\draw[colour4] (10) -- (11);
			\draw[colour4b] (10) -- (11);
			\draw[colour4] (13) -- (20);
			\draw[colour4b] (13) -- (20);
			\draw[colour4] (15) -- (21);
			\draw[colour4b] (15) -- (21);
			\draw[colour4] (16) -- (17);
			\draw[colour4b] (16) -- (17);
			\draw[colour4] (18) -- (19);
			\draw[colour4b] (18) -- (19);
		\end{tikzpicture}
		\hfill
		\begin{tikzpicture}[very thick]
			\node (0) at (2,0) {};
			\node (2) at (1,0) {};
			\node (6) at (0,-3) {};
			\node (7) at (0,-1.8) {};
			\node (9) at (.3,0) {};
			\node (10) at (0,.6) {};
			\node (11) at (0,1.8) {};
			\node (12) at (0,3) {};
			\node (16) at (0,-.6) {};
			\node (17) at (-.3,0) {};
			\node (20) at (-2,0) {};
			\node (21) at (-1,0) {};
			
			\node (4) at ($.5*(12)+.5*(0)$) {};
			\node (19) at ($.5*(12)+.5*(20)$) {};
			\node (1) at ($.5*(6)+.5*(0)$) {};
			\node (14) at ($.5*(6)+.5*(20)$) {};
			
			\node (3) at ($.5*(11)+.5*(2)$) {};
			\node (8) at ($.5*(7)+.5*(2)$) {};
			\node (15) at ($.5*(7)+.5*(21)$) {};
			\node (18) at ($.5*(11)+.5*(21)$) {};
			
			\node[fill=black] (5) at (3,0) {};
			\node[fill=black] (13) at (-3,0) {};
			
			\draw[colour0] (0) -- (1);
			\draw[colour0] (2) -- (8);
			\draw[colour0] (3) -- (10);
			\draw[colour0] (4) -- (11);
			\draw[colour0] (5) -- (12);
			\draw[colour0] (6) -- (7);
			\draw[colour0] (9) -- (17);
			\draw[colour0] (13) -- (19);
			\draw[colour0] (14) -- (20);
			\draw[colour0] (15) -- (16);
			\draw[colour0] (18) -- (21);
			
			\draw[colour1] (0) -- (2);
			\draw[colour1] (1) -- (5);
			\draw[colour1] (3) -- (4);
			\draw[colour1] (6) -- (13);
			\draw[colour1] (7) -- (8);
			\draw[colour1] (9) -- (10);
			\draw[colour1] (11) -- (12);
			\draw[colour1] (14) -- (15);
			\draw[colour1] (16) -- (17);
			\draw[colour1] (18) -- (19);
			\draw[colour1] (20) -- (21);
			
			\draw[colour2] (0) -- (3);
			\draw[colour2b] (0) -- (3);
			\draw[colour2] (1) -- (6);
			\draw[colour2b] (1) -- (6);
			\draw[colour2] (2) -- (9);
			\draw[colour2b] (2) -- (9);
			\draw[colour2] (4) -- (5);
			\draw[colour2b] (4) -- (5);
			\draw[colour2] (7) -- (14);
			\draw[colour2b] (7) -- (14);
			\draw[colour2] (8) -- (16);
			\draw[colour2b] (8) -- (16);
			\draw[colour2] (10) -- (17);
			\draw[colour2b] (10) -- (17);
			\draw[colour2] (11) -- (18);
			\draw[colour2b] (11) -- (18);
			\draw[colour2] (12) -- (13);
			\draw[colour2b] (12) -- (13);
			\draw[colour2] (15) -- (21);
			\draw[colour2b] (15) -- (21);
			\draw[colour2] (19) -- (20);
			\draw[colour2b] (19) -- (20);
			
			\draw[colour3] (0) -- (4);
			\draw[colour3] (1) -- (7);
			\draw[colour3] (2) -- (10);
			\draw[colour3] (3) -- (11);
			\draw[colour3] (5) -- (6);
			\draw[colour3] (8) -- (9);
			\draw[colour3] (12) -- (19);
			\draw[colour3] (13) -- (14);
			\draw[colour3] (15) -- (20);
			\draw[colour3] (16) -- (21);
			\draw[colour3] (17) -- (18);
			
			\draw[colour4] (0) -- (5);
			\draw[colour4b] (0) -- (5);
			\draw[colour4] (1) -- (8);
			\draw[colour4b] (1) -- (8);
			\draw[colour4] (2) -- (3);
			\draw[colour4b] (2) -- (3);
			\draw[colour4] (4) -- (12);
			\draw[colour4b] (4) -- (12);
			\draw[colour4] (6) -- (14);
			\draw[colour4b] (6) -- (14);
			\draw[colour4] (7) -- (15);
			\draw[colour4b] (7) -- (15);
			\draw[colour4] (9) -- (16);
			\draw[colour4b] (9) -- (16);
			\draw[colour4] (10) -- (18);
			\draw[colour4b] (10) -- (18);
			\draw[colour4] (11) -- (19);
			\draw[colour4b] (11) -- (19);
			\draw[colour4] (13) -- (20);
			\draw[colour4b] (13) -- (20);
			\draw[colour4] (17) -- (21);
			\draw[colour4b] (17) -- (21);
		\end{tikzpicture}
		\caption{The two planar quintic perfectly hamiltonian graphs of order 22. Removing the black vertices yields the two building blocks from Figure~\ref{fig:exponential_building_blocks}.}\label{fig:22}
	\end{figure}
	
	\begin{figure}
		\begin{center}
			\begin{tikzpicture}[scale=.75]
				\node (1) at (-1, 0) {};
				\node (2) at (1, 0) {};
				\node (3) at (-1, 0.8) {};
				\node (4) at (1, 0.8) {};
				\node (5) at (-1, 1.6) {};
				\node (6) at (-1.5, 2) {};
				\node (7) at (-.5, 2) {};
				\node (8) at (1, 2) {};
				\node (9) at (-1, 2.4) {};
				\node (10) at (1, 2.6) {};
				\node (11) at (-1, 3.2) {};
				\node (12) at (1, 3.2) {};
				\node (13) at (-1, 4) {};
				\node (14) at (1, 4) {};
				
				\draw [thick] (1) -- (2) -- (4) -- (8) -- (10) -- (12) -- (14) -- (13) -- (11) -- (9) -- (7) -- (5) -- (3) -- (1) (3)--(4) (7) -- (8) (11) -- (12) (9) -- (6) -- (5);
				
				\draw (1) to ++(180:1) (6) to ++(180:.5) (13) to ++(180:1);
				\draw (2) to ++(0:1) (10) to ++(0:1) (14) to ++(0:1);
			\end{tikzpicture}
			\hspace{5em}
			\begin{tikzpicture}[scale=.75]
				\node (16) at (1,.8) {};
				\node (15) at (1,1.6) {};
				\node (14) at (1,2.4) {};
				\node (12) at (1,0) {};
				\node (11) at (-1,.8) {};
				\node (10) at (-1,1.6) {};
				\node (9) at (1,3.2) {};
				\node (8) at (1,4) {};
				\node (7) at (-1,0) {};
				\node (6) at (-1,2.4) {};
				\node (5) at (0,3.2) {};
				\node (4) at (0,4) {};
				\node (2) at (-1,3.2) {};
				\node (1) at (-1,4) {};
				
				\draw[thick] (2) -- (1) -- (4) -- (5) -- (2) -- (6) -- (10) -- (15) -- (16) -- (11) -- (7) -- (12) -- (16) (4) -- (8) -- (9) -- (14) -- (15) (9) -- (5) (10) -- (11);
				
				\draw (7) to ++(180:1) (6) to ++(180:1) (1) to ++(180:1);
				\draw (14) to ++(0:1) (8) to ++(0:1) (12) to ++(0:1);
			\end{tikzpicture}
			\vspace{40pt}
			
			\begin{tikzpicture}
				\node (1) at (1, 0) {};
				\node (2) at (1, 1) {};
				\node (3) at (1, 2) {};
				\node (4) at (1, 3.25) {};
				\node (5) at (0, .5) {};
				\node (6) at (0, 1.5) {};
				\node (7) at (0, 2.5) {};
				\node (8) at (0, 4) {};
				\node (9) at (-1, 0) {};
				\node (10) at (-1, 1) {};
				\node (11) at (-1, 2) {};
				\node (12) at (-1, 3.25) {};
				
				\draw[thick] (1) -- (2) -- (3) -- (7) -- (4) -- (8) -- (12) -- (7) -- (11) -- (10) -- (9) -- (1) -- (5) -- (2) -- (6) -- (10) -- (5) -- (9) (3) -- (6) -- (11) (4) -- (12);
				
				\draw (9) to ++(180:.5) (8) to ++(180:1.5) (11) to ++(180:.5) (12) to ++(180:.5);
				\draw (1) to ++(0:.5) (3) to ++(0:.5) (4) to ++(0:.5) (8) to ++(0:1.5);
			\end{tikzpicture}
			\hspace{5em}
			\begin{tikzpicture}
				\node (1) at (1, 3) {};
				\node (2) at (1, 2) {};
				\node (3) at (1, 1) {};
				\node (4) at (0, 0) {};
				\node (5) at (0, 1) {};
				\node (6) at (0, 1.5) {};
				\node (7) at (0, 2.5) {};
				\node (8) at (0, 3) {};
				\node (9) at (0, 4) {};
				\node (10) at (-1, 3) {};
				\node (11) at (-1, 2) {};
				\node (12) at (-1, 1) {};
				
				\draw[thick] (1) -- (2) -- (3) -- (4) -- (5) -- (6) -- (2) -- (7) -- (8) -- (9) -- (1) -- (8) -- (10) -- (11) -- (12) -- (6) -- (11) -- (7) -- (10) (3) -- (5) -- (12);
				
				\draw (4) to ++(180:1.5) (9) to ++(180:1.5) (10) to ++(180:.5) (12) to ++(180:.5);
				\draw (1) to ++(0:.5) (3) to ++(0:.5) (4) to ++(0:1.5) (9) to ++(0:1.5);
			\end{tikzpicture}
		\end{center}
		\caption{Two cubic and two quartic building blocks which can be used to show that the number of perfectly hamiltonian planar cubic (respectively quartic) graphs grows at least exponentially.}\label{fig:exponential_building_blocks_cubic_quartic}
	\end{figure}
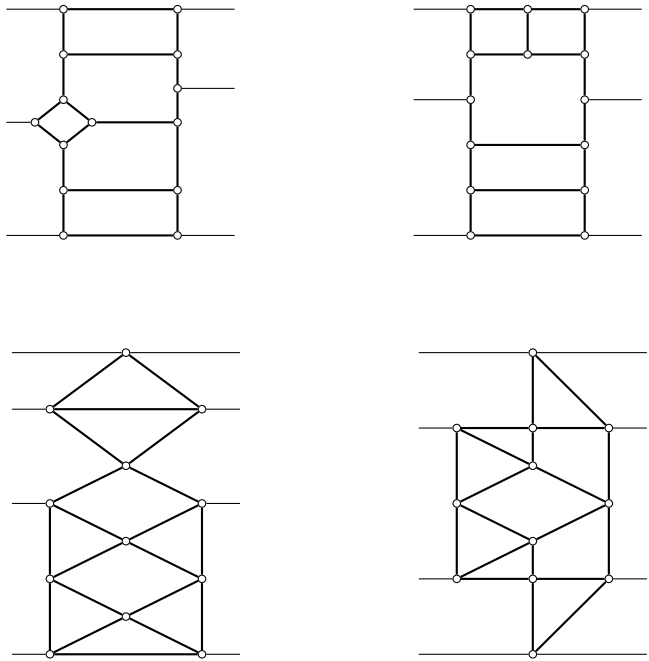
	
	\newpage
	
	\begin{table}
		\begin{center}
			\begin{tabular}{ccccccc}
				\toprule
				\ & \multicolumn{2}{c}{Cubic} & \multicolumn{2}{c}{Quartic} & \multicolumn{2}{c}{Quintic}\\
				\ & Polyhedra & Perfectly & Polyhedra & Perfectly & Polyhedra & Perfectly\\
				\(n\) &  & hamiltonian &  & hamiltonian &  & hamiltonian\\
				\midrule
				4 &        1 &       1\\
				6 &        1 &       1 &       1 &      1\\
				8 &        2 &       1 &       1 &      0\\
				10 &        5 &       3 &       3 &      2\\
				12 &       14 &       8 &      11 &      1 &    1 &  0\\
				14 &       50 &      27 &      58 &     20\\
				16 &      233 &     115 &     451 &     36 &    1 &   0\\
				18 &     1249 &     570 &    4461 &    717 &    1 &   0\\
				20 &     7595 &    3116 &   49957 &   3230 &    6 &   1\\
				22 &    49566 &   18221 &  598102 &  53431 &   14 &   2\\
				24 &   339722 &  111720 & 7437910 & 396224 &   96 &  24\\
				26 &  2406841 &  709282 & 94944685 & 5627207 &  518 &  297\\
				28 & 17490241 & 4631674 & 1236864842 & 53565173 & 3917 & 1763\\
				\bottomrule
			\end{tabular}
		\end{center}
		\caption{Counts of planar $k$-regular perfectly hamiltonian graphs of order $n$ for $k \in \{ 3, 4, 5\}$. The column `Polyhedra' gives the total number of \(k\)-regular planar 3-connected graphs of order \(n\).}\label{tab:counts}
	\end{table}
	
	\clearpage
	
	\section*{Supplementary material}
	
	\begin{figure}[h]
		\newenvironment{divorcedparts}[3]{
			\begin{tikzpicture}[yscale=.47,font=\small]
				\begin{scope}[every node/.style={draw,circle,inner sep=1pt,thick,fill=white}]
					\filldraw[fill=gray!50,draw=black] (-.5,2) arc (150:210:1.5) (-.5,.5) -- (-.5,-.5) arc (150:210:1.5) (-.5,-2) arc (270:90:2) (-.5,2);
					\node[label=right:a] (a) at (-.5,2) {};
					\node[label=right:b] (b) at (-.5,.5) {};
					\node[label=right:c] (c) at (-.5,-.5) {};
					\node[label=right:d] (d) at (-.5,-2) {};
					
					\filldraw[fill=gray!50,draw=black] (.5,2) arc (30:-30:1.5) (.5,.5) -- (.5,-.5) arc (30:-30:1.5) (.5,-2) arc (-90:90:2) (.5,2);
					\node[label=left:a'] (a') at (.5,-2) {};
					\node[label=left:b'] (b') at (.5,-.5) {};
					\node[label=left:c'] (c') at (.5,.5) {};
					\node[label=left:d'] (d') at (.5,2) {};
					
					\draw[ultra thick] (b) -- (c) (b') -- (c');
				\end{scope}
				\node[draw=none] at (0,-3) {#1};
				\node[draw=none] at (-1.5,-2.5) {#2};
				\node[draw=none] at (1.5,-2.5) {#3};
			}{
			\end{tikzpicture}
		}
		
		\begin{divorcedparts}{1-2}{(i)}{(iii)}
			\draw[very thick, red] (c) -- (b) (b) to[out=135,in=215] (a) (d) to[out=180,in=180] (a);
			\draw[very thick, red] (c') -- (b') (b') to[out=60,in=45,looseness=2] (a');
		\end{divorcedparts}
		\hfill
		\begin{divorcedparts}{1-3}{(ii)}{(iv)}
			\draw[very thick, red] (c) -- (b) (c) to[out=135,in=215] (a) (d) to[out=180,in=180] (a);
			\draw[very thick, red] (c') -- (b') (c') to[out=60,in=45,looseness=2] (a');
		\end{divorcedparts}
		\bigskip
		
		\begin{divorcedparts}{1-4}{(iii)}{(i)}
			\draw[very thick, red] (c) -- (b) (b) to[out=240,in=215,looseness=2] (a);
			\draw[very thick, red] (c') -- (b') (b') to[out=-45,in=45] (a') (d') to[out=0,in=0] (a');
		\end{divorcedparts}
		\hfill
		\begin{divorcedparts}{1-5}{(iv)}{(ii)}
			\draw[very thick, red] (c) -- (b) (c) to[out=240,in=215,looseness=2] (a);
			\draw[very thick, red] (c') -- (b') (c') to[out=-45,in=45] (a') (d') to[out=0,in=0] (a');
		\end{divorcedparts}
		\bigskip
		
		\begin{divorcedparts}{2-3}{(v)}{(x)}
			\draw[very thick, red] (b) to[out=135,in=215] (a) (d) to[out=135,in=215] (c) (d) to[out=180,in=180] (a);
			\draw[very thick, red] (c') to[out=60,in=-60,looseness=4] (b');
		\end{divorcedparts}
		\hfill
		\begin{divorcedparts}{2-4}{(vi)}{(vi)}
			\draw[very thick, red] (b) to[out=135,in=215] (a) (d) to[out=135,in=215] (b);
			\draw[very thick, red] (b') to[out=-45,in=45] (a') (d') to[out=-45,in=45] (b');
		\end{divorcedparts}
		\bigskip
		
		\begin{divorcedparts}{2-5}{(vii)}{(viii)}
			\draw[very thick, red] (b) to[out=135,in=215] (c) (d) to[out=180,in=180] (a);
			\draw[very thick, red] (b') to[out=0,in=45,looseness=2] (a') (d') to[out=-45,in=0,looseness=2] (c');
		\end{divorcedparts}
		\hfill
		\begin{divorcedparts}{3-4}{(viii)}{(vii)}
			\draw[very thick, red] (b) to[out=180,in=215,looseness=2] (a) (d) to[out=135,in=180,looseness=2] (c);
			\draw[very thick, red] (c') to[out=45,in=-45] (b') (d') to[out=0,in=0] (a');
		\end{divorcedparts}
		\bigskip
		
		\begin{divorcedparts}{3-5}{(ix)}{(ix)}
			\draw[very thick, red] (c) to[out=135,in=215] (a) (d) to[out=135,in=215] (c);
			\draw[very thick, red] (c') to[out=-45,in=45] (a') (d') to[out=-45,in=45] (c');
		\end{divorcedparts}
		\hfill
		\begin{divorcedparts}{4-5}{(x)}{(v)}
			\draw[very thick, red] (b') to[out=-45,in=45] (a') (d') to[out=-45,in=45] (c') (d') to[out=0,in=0] (a');
			\draw[very thick, red] (c) to[out=240,in=120,looseness=4] (b);
		\end{divorcedparts}
		
		\caption{The ten cases from the proof of Theorem~\ref{thm:divorce}. Each path or pair of paths always contains all vertices of the fragment to which it belongs except for vertices that are explicitly not drawn on the path. Below each fragment we show on which condition of ``suitable'' we rely for the existence of the corresponding path or pair of paths.}\label{fig:divorce_proof}
	\end{figure}
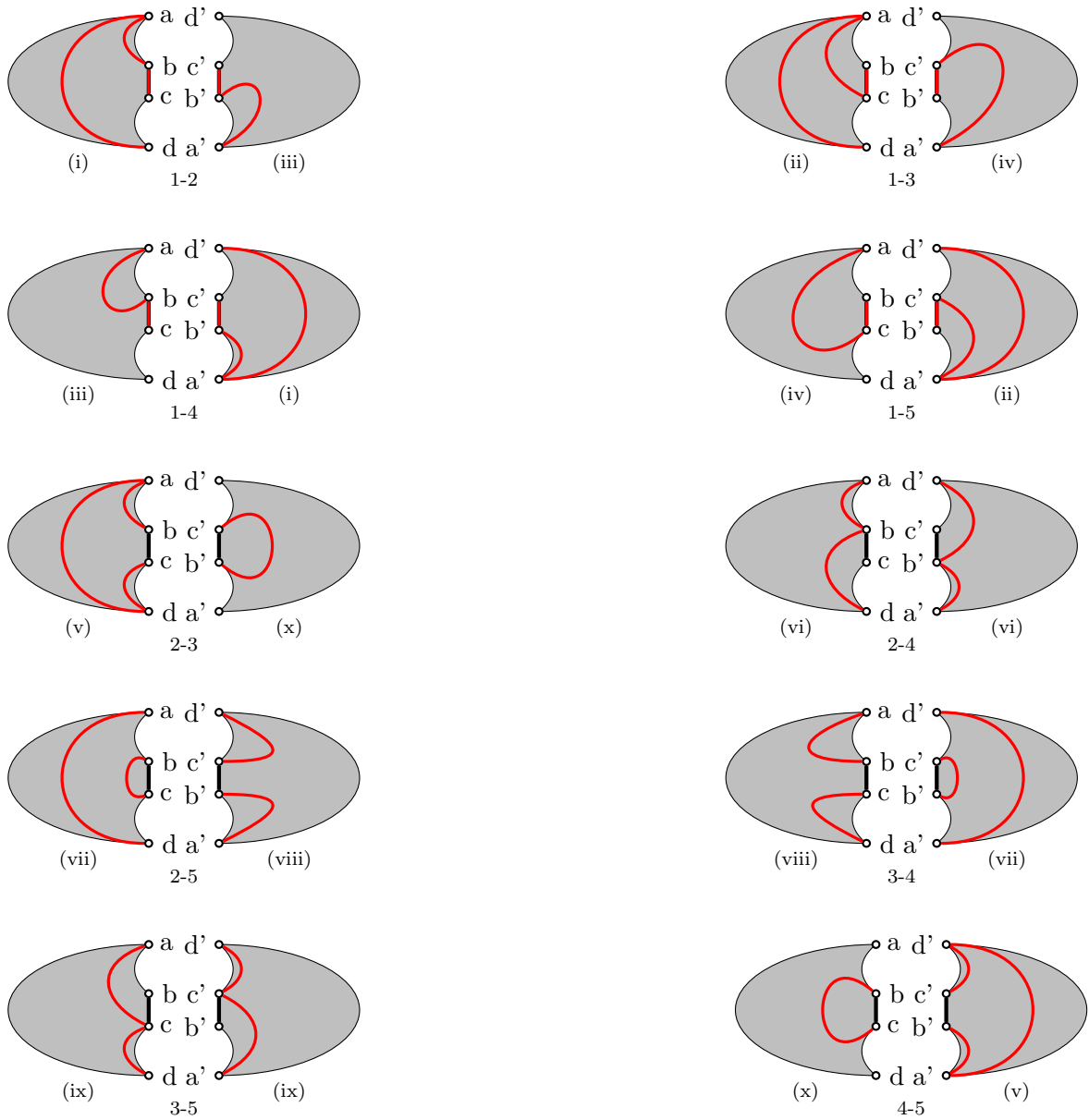

\end{document}